\documentclass[11pt]{article}
\usepackage[margin=1in]{geometry}

\usepackage{listings}
\usepackage{geometry}

\usepackage{amsmath,amsthm,amsfonts,amssymb,amscd}
\usepackage[colorlinks=true, pdfstartview=FitV, linkcolor=blue,citecolor=blue, urlcolor=blue]{hyperref}
\usepackage[abbrev,lite,nobysame]{amsrefs}
\usepackage[usenames,dvipsnames]{color}
\usepackage{times}

\usepackage{mathtools}
\usepackage{mathabx}
\usepackage{float}
\usepackage{makeidx}
\usepackage{stmaryrd}
\usepackage{mathrsfs}
\usepackage{emptypage}
\usepackage{amsrefs}
\usepackage{xfrac}
\usepackage{bbm}

\usepackage{upgreek}

\renewcommand{\epsilon}{\varepsilon}

\usepackage{graphicx}

\newcommand{\p}{\ensuremath{\partial}}
\newcommand{\n}{\ensuremath{\nonumber}}

\def\les{\lesssim}

\def\eps{\varepsilon}
\renewcommand*{\div}{\ensuremath{\mathrm{div\,}}}

\newcommand{\ud}{\,\mathrm{d}}

\newcommand{\norm}[1]{\left \| #1 \right\|} 
\newcommand{\abs}[1]{\left|#1\right|}

\newcommand*{\supp}{\ensuremath{\mathrm{supp\,}}}

\newcommand{\OO}{\mathcal O}

\renewcommand*{\tilde}{\widetilde}
\renewcommand*{\hat}{\widehat}
\renewcommand*{\bar}{\overline}

\newcommand{\T}{{\mathbb T}}

\newtheorem{theorem}{Theorem}[section]
\newtheorem{lemma}[theorem]{Lemma}

\newtheorem{proposition}[theorem]{Proposition}
\newtheorem{corollary}[theorem]{Corollary}
\theoremstyle{definition}
\newtheorem{definition}[theorem]{Definition}

\newtheorem{remark}[theorem]{Remark}

\numberwithin{equation}{section}

\def\p{\partial}

\def\f1r{{\frac{1}{r}}  }
\def\p{\partial}

\def\f1r{{\frac{1}{r}}  }

\usepackage{fancyhdr}
\title{{\bf Formation of unstable shocks  for 2D isentropic compressible Euler  }}

\author{
{ \small {\bf Tristan Buckmaster}\footnote{Department of Mathematics, Princeton University; email: \url{buckmaster@math.princeton.edu}; partially supported by NSF grant DMS-1900149 and a Simons Foundation Mathematical and Physical Sciences Collaborative Grant.  } }
\and  
{\small {\bf Sameer Iyer}\footnote{Department of Mathematics, University of California, Davis; email: \url{sameer@math.ucdavis.edu}; partially supported by NSF grant DMS-1802940.} }
}

\date{} %
\begin{document}

\maketitle

\begin{abstract}

In this paper we construct unstable shocks in the context of 2D isentropic compressible Euler in azimuthal symmetry. More specifically, we construct initial data that when viewed in self-similar coordinates, converges asymptotically to the unstable $C^{\frac15}$ self-similar solution to the Burgers' equation. Moreover, we show the behavior is stable in $C^8$ modulo a two dimensional linear subspace. Under the azimuthal symmetry assumption, one cannot impose additional symmetry assumptions in order to isolate the corresponding manifold of initial data leading to stability: rather, we rely on modulation variable techniques in conjunction with a Newton scheme.

\end{abstract}
\setcounter{tocdepth}{1}
\tableofcontents

\allowdisplaybreaks

\section{Introduction}

\subsection{Setup of Compressible Euler under azimuthal symmetry}

In this paper we study asymptotically self-similar formation of unstable shocks for the 2D isentropic compressible Euler equations under azimuthal symmetry.  The 2D isentropic compressible Euler equations take the form
\begin{subequations}
\label{eq:Euler}
\begin{align}
\partial_t (\rho u) + \div (\rho\, u \otimes u) + \nabla p(\rho) &= 0 \,,  \label{eq:momentum} \\
\partial_t \rho  + \div (\rho u)&=0 \,,  \label{eq:mass}
\end{align}
\end{subequations}
where $u :\mathbb{R}^2  \times \mathbb{R}  \to \mathbb{R}^2  $ is the velocity of the fluid, $\rho: \mathbb{R}^2  \times \mathbb{R}  \to \mathbb{R}  _+$ is the density, 
and $p: \mathbb{R}^2  \times \mathbb{R}  \to \mathbb{R}  _+$ is the pressure defined by the ideal gas law
$$
 p(\rho) := \tfrac{1}{\gamma} \rho^\gamma\,, \qquad \gamma >1 \,.
$$
The associated sound speed $\sigma$ is given by $\sigma=\rho^\lambda$ where $\lambda=\frac{\gamma-1}{2}$.

It was shown in \cite{BuShVi2019}, that if one imposes the following azimuthal symmetry 
\begin{equation}\label{e:symmetry}
u(x,t)\cdot \frac{x}{\abs{x}}= r a(\theta,t), \quad u(x,t)\cdot \frac{x^{\perp}}{\abs{x}}=r b(\theta,t),\quad \rho=r^{\frac2{\gamma-1}}  P(\theta,t),
\end{equation}
where $(r,\theta)$ are the usual polar coordinates,
then the equations \eqref{eq:Euler} reduce to the 1D system of equations
\begin{subequations}
\label{eq:Euler:polar3}
\begin{align} \label{g3_a_evo}
\left(\partial_t  + b\partial_{\theta}\right) a  + a^2-b^2+  \lambda^{-1} P^{2 \lambda }   &=0 \\ \label{g3_b_evo}
\left(\partial_t  + b\partial_{\theta}\right)b+2a b+  P^{ 2 \lambda -1}\partial_\theta P&=0 \\ \label{g3_P_evo}
\left(\partial_t  + b\partial_{\theta}\right) P+  \tfrac{\gamma}{ \lambda } a P+  P  \p_\theta b &=0 \, .
\end{align}
\end{subequations}

An important difference between Euler under azimuthal rather than radial symmetry is that azimuthal symmetry allows for the presence of non-trivial vorticity. We remark that it was shown in \cite{BuShVi2019}, that the system \eqref{eq:Euler:polar3} is locally well-posed in $C^n$ for any $n\geq 1$.

In order to avoid issues regarding the irregularity at the origin $r=0$ (created by the azimuthal symmetry), and in order to ensure finite kinetic-energy, following \cite{BuShVi2019}, we can exploit locality and restrict the solution \eqref{e:symmetry} to the push forward of an annulus under the flow induced by $u$. To be more precise, define $A _{\underline{r},\overline{r}}$ to be the annular region
\begin{align*}
A _{\underline{r},\overline{r}}= \{ x\in\mathbb R^2 \colon \underline{r} < \abs{x} < \overline{r} \} \,.
\end{align*}
Fixing $0<   r_0 < r_1$; then, if $\eta_u$ is the solution to 
$\p_t \eta_u = u \circ \eta_u$ for $t>t_0$ with $\eta_u(x,t_0)=x$, define the time dependent domain
\begin{equation}\label{Omega}
\Omega(t)= \eta_u(A_{r_0,r_1},t)\,.
\end{equation} 
Now set $0< R_0 < r_0 < r_1 <R_1$ and let $K>0$. Assuming that $\abs{u}\leq K$  for all $(x,t)\in A _{R_0,R_1}\times [t_0,T_*)$, then it follows that
\begin{equation*}
\Omega(t)\subset A_{R_0,R_1} \quad \text{ for } \quad t \in [t_0,T_*]\,,
\end{equation*} 
so long as $\abs{T_*-t_0}$ is assumed to be sufficiently small (depending or $r_0$, $r_1$, $R_0$, $R_1$ and $K$). Then given a solution $(a,b,P)$ to the system \eqref{eq:Euler:polar3}, we relate these to solutions to \eqref{eq:Euler} via the transformation \eqref{e:symmetry}, restricted to the domain $\Omega$ given in \eqref{Omega}.

\subsection{Brief historical overview}

The formation of shocks is a classical problem in hyperbolic PDE. The first rigorous proof of shock formation is due to the pioneering work of Lax \cite{Lax1964} that employed invariants devised by Riemann \cite{Ri1860} and the method of characteristics. The work of Lax was further generalized and refined by  John~\cite{John1974},  Liu~\cite{Li1979}, and  Majda~\cite{Ma1984} (cf.\ \cite{Da2010}).

In the multi-dimensional setting, Sideris in \cite{Si1997} demonstrated using a virial type argument the existence of solutions that form singularities in finite time. The method of proof does not however lead to a classification of the type of singularity produced. The first proof of shock formation in the multi-dimensional setting was given by Christodoulou \cite{Ch2007}, whereby he proved shock formation in the irrotational, relativistic setting. The work was later generalized to non-relativistic, irrotational setting \cite{ChMi2014}, and then further extended by Luk and Speck to the 2D setting with non-trivial vorticity \cite{LuSp2018}. It is important to note that while the cited work  are capable of proving shock formation (or simply singularity formation in the case of Sideris), the methods of proof are incapable of distinguishing precise information on the shock's profile.  For example, none of the cited work determine whether the shock occurs at one specific location or whether multiple shocks occur simultaneously. In the recent work by the first author, Shkoller and Vicol \cite{BuShVi2019}, it was shown than in 2D under the azimuthal symmetry \eqref{e:symmetry} one can prove the existence of stable shocks (stable with respect to perturbations that preserve the azimuthal symmetry) whose self-similar profile can be precisely described. This work in \cite{BuShVi2019b} was extended to 3D in the absence of any symmetry assumption, and further extended to the non-isentropic case in \cite{BuShVi2019c}. In a different direction, we would like to also bring to attention of the remarkable recent works of Merle, Raphael, Rodnianski, and Szeftel, \cite{MRRS1}, \cite{MRRS2}, which demonstrated the existence of radially symmetric imploding solutions to the isentropic Euler equation -- a completely new form of singularity for the Euler equations.

\subsection{Unstable shocks for the Burgers' equation}\label{s:burgers}

Before we state a rough version of the main theorem, let us first review the concept of an \textit{unstable shock} in the context of the 1D Burgers' equation:
\begin{equation}\label{eq:Burgers}
\p_t w+ w \p_y w = 0\,\quad\mbox{for }y\in\mathbb R\,.
\end{equation}
The Burgers' equation satisfies the following four invariances:
\begin{enumerate}
\item Galilean symmetry: If $w(y,t)$ is solves \eqref{eq:Burgers} then $w(y-v,t)+v$ solves \eqref{eq:Burgers} for any $v\in\mathbb R$.
\item Temporal rescaling: If $w(y,t)$ is solves \eqref{eq:Burgers} then $\lambda w(y,\lambda t)$ solves \eqref{eq:Burgers} for any $\lambda>0$.
\item Translation invariance: If $w(y,t)$ is solves \eqref{eq:Burgers} then $w(y-y_0, t)$ solves \eqref{eq:Burgers} for any $y_0\in\mathbb R$.
\item Spatial rescaling: If $w(y,t)$ is solves \eqref{eq:Burgers} then $\lambda^{-1} w(\lambda y,t)$ solves \eqref{eq:Burgers} for any $\lambda>0$.
\end{enumerate}
Any initial data $w_0$ with a negative slope at some point $y_0$ will shock in finite time. Let us assume that $w_0$ has a global minimum slope. By temporal rescaling and translation invariance, without loss of generality, we may assume the global minimum slope is $-1$, occurring at $y=0$. Let us take the initial time to be $t=-1$. By Galilean symmetry, without loss of generality, we may further assume $w_0(0)=0$, then by methods of characteristics that the solution $w$ will shock at $(y,t)=(0,0)$.

If in addition $w'''_0(0) =\nu> 0$, then the solution $w$ will converge asymptotically at the blow up to a self-similar profile $\bar W_1$; in particular, we have
\begin{align} \label{limit:1}
\lim_{t\rightarrow 0} (-t)^{\frac12}w(x(-t)^{-\frac32},t)= \left(\frac{\nu}{6}\right)^{-\frac12}\bar{W}_1\left(\left(\frac{\nu}{6}\right)^{\frac12}x\right)\,,
\end{align}
for any $x\in\mathbb R$, where
\begin{align} \label{explicit:1}
 \bar{W}_1(x) = \left(- \frac x2 + \left(\frac{1}{27} + \frac{x^2}{4}\right)^{\frac 12}\right)^{\frac 13} - \left( \frac x2 + \left(\frac{1}{27} + \frac{x^2}{4}\right)^{\frac 12} \right)^{\frac 13} \,.
\end{align}
\begin{remark}
Note one can fix $\nu$ by making use of the spacial rescaling invariance of Burgers' equation. 
\end{remark}
The shock profile is stable in the sense that given any initial data in a suitably small $C^4$ neighborhood of $w_0$, the resulting solution will satisfy \eqref{limit:1} modulo the invariances of Burgers' equation. The profile $\bar W_1$ (together its $\nu$ rescaling given on the right hand side of \eqref{limit:1}) satisfy the following self-similar Burgers' equation
\[
-\frac{1}{2}  \bar W_1 + \left( \frac{3}{2}x + \bar W_1 \right) \partial_x  \bar W_1=0\,.
\]
In addition to $\bar{W}_1$ defined above, the Burgers' equation admits a countable family of  smooth self-similar profiles \cite{Fontelos}. For each $i\in\mathbb N$, there exists a unique non-trivial analytic profile $W_i$ satisfying the ODE
\[ -\frac 1{2i}  \bar{W}_i + \left( \frac{(2i+1)x}{2i} + \bar{W}_i \right) \partial_x  \bar{W}_i = 0\,. 
\]
such that 
\[w_i(x,t)=(-t)^{\frac1{2i}}\bar{W}_i(x(-t)^{\frac{2i+1}{2i}})\,,\]
defines a self-similar solution to the Burgers' equation. Unlike $\bar{W}_1$, the solutions $\bar{W}_{i}$ for $i>1$ are unstable: generic small perturbations of initial data $w_i(\cdot,0)$ lead to singularities described by the stable self-similar profile $\bar W_1$. Indeed a  generic smooth perturbation of $w_i(x,0)$ leads to initial data with a global minimum at a point where the third derivative is positive, which by the discussion above leads to a shock with asymptotic profile $\bar W_1$.

The profiles $\bar{W}_{i}$ for $i>1$ are nevertheless stable modulo a finite co-dimension of initial data: Suppose we are given initial data $w_0$ with a global minimum, as a consequence of the invariances of Burgers' equation, we may further assume $w_0(0)=0$ and $w_0'(0)=-1$. If we further assume that $w_0^{(n)}(0)=0$ for $n=2,\dots,2i $ and that $w_0^{(2i+1)}(0)=\nu>0$, then
\begin{align} \label{limit:2}
\lim_{t\rightarrow 0} (-t)^{\frac1{2i}}w(x(-t)^{-\frac{2i+1}{2i}},t)= \left(\frac{\nu}{(2i+1)!}\right)^{-\frac1{2i}}\bar{W}_i\left(\left(\frac{\nu}{(2i+1)!}\right)^{\frac1{2i}}x\right)\,,
\end{align}
for all $x\in\mathbb R$. Thus the initial data leading to the unstable shock profiles $\bar W_i $ for $i>1$ are described by a  manifold of finite codimension.

We note that in the paper \cite{CoGhMa2018}, the authors study stable and unstable self-similar solutions to the Burgers equation in order to investigate the Burgers equation with transverse viscosity. 

Our main objective in this work is to identify an analogous manifold, $\mathcal{M}$, for the compressible Euler equations which lead to unstable blowup dynamics according to the profile $\bar{W}_2$. Unlike the case for Burgers described above, the specification of $\mathcal{M}$ is not as explicit as that described above, and must be found via very careful Newton scheme.  

\subsection{Rough statement of main theorem}
In this paper, we prove the existence of asymptotically self similar solutions to 2D isentropic compressible Euler equations under azimuthal symmetry that under the appropriate self-similar transformations are described by the self-similar Burgers' profile $\bar W_2$:
\begin{theorem}
There exists initial data $(a_0,b_0, P_0)$ in $C^8$ for which the corresponding solutions $(a,b,P)$ to \eqref{eq:Euler:polar3} develop a $C^{\frac15}$-cusp singularity in finite time.
 At blow-up, the solutions $(a,b,P)$ form singularity at a unique angle; moreover, the singularities may be described in terms of the self-similar Burgers' profile $\bar W_2$ in a manner made precise in Theorem \ref{thm:general}. The behavior described is stable in $C^8$ with regards to the initial data modulo a two dimensional linear subspace.
\end{theorem}
We note that analogous results exist for the Burger's equation with traversal viscosity \cite{CoGhMa2018}, the Prandtl equations \cites{CoGhIbMa18,CoGhMa19} and the Burgers-Hilbert equation \cite{yang2020shock}. We also note that the formation of \textit{unstable} shocks (defined and discussed below) in the context of Bourgain-Wang solutions to NLS was obtained in \cite{MR3086066} through virial type identities and backwards integration techniques. These papers however rely on a symmetry to constrain the position of the singularity which leads to a comparatively simple classification of initial data leading to unstable blow up profiles. Isentropic Euler does not satisfy analogous symmetries leading us to develop a new shooting method in order to describe initial data leading to unstable blowup. We believe that techniques developed are suitably malleable and could find potential use in proving the existence of unstable blowup for other PDE.

\section{Statement of main theorem}

\subsection{Riemann invariants}

Before we can state our main theorem, we must first introduce the concept of Riemann invariants, since it is our aim to show that we can prescribe initial data such that one of the Riemann invariants shocks according to the self-similar profile $\bar W_2$.

As was done in \cite{BuShVi2019}, in order to diagonalize the system \eqref{g3_a_evo}  - \eqref{g3_P_evo} and isolate the Burgers-like behavior of the shock development, we will rewrite  \eqref{g3_a_evo} - \eqref{g3_P_evo} in terms of the Riemann invariants
\[w= b+  {\frac{1}{\lambda }} P^\lambda \,, \qquad z= b- {\frac{1}{\lambda }}  P^\lambda \,,\]
and the wave speeds
\[\Lambda_1= b - P^\lambda= \frac{1- \lambda }{2} w +  \frac{1+ \lambda }{2} z \,, \qquad  \Lambda_2= b+ P^ \lambda =  \frac{1+ \lambda }{2} w +  \frac{1- \lambda }{2} z\,.\]
With these substitutions we obtain the following system of nonlinear transport equations
\begin{subequations}
\label{eq:euler:wza}
\begin{align}
\partial_t w +\left(w+\tfrac{1-\lambda}{1+\lambda}z\right)\partial_{\theta}w 
&= -a  \left(\tfrac{1-2\lambda}{1+\lambda} z+ \tfrac{3+2\lambda}{1+\lambda} w\right) 
\label{eq:w:evo} \,,\\
\partial_t z +\left(z+\tfrac{1-\lambda}{1+\lambda}w\right)\partial_{\theta}z 
&=  -a  \left( \tfrac{1-2\lambda}{1+\lambda} w+ \tfrac{3+2\lambda}{1+\lambda} z\right)
\label{eq:z:evo} \,, \\
\partial_t a +\tfrac{1}{1+\lambda} (w+z) \partial_{\theta}a  
&=-\tfrac{2}{1+\lambda}a^2+\tfrac{1}{2(1+\lambda)}(w+z)^2 - \tfrac{\lambda}{2(1+\lambda)}(w-z)^2   \,.
\label{eq:a:evo}
\end{align}
\end{subequations}

\subsection{Initial data assumptions}\label{ss:initial}

In this section we will describe the initial data used to construct unstable shock solutions. We introduce a large constant $M$ which will be used to bound certain implicit constants appearing in the paper. We also let $\eps>0$ be a small constant which will parameterize the slope of the initial data.

We will denote the initial data at initial time $t=-\eps$ by
\[w(\theta,-\eps)=w_0,\quad z(\theta,-\eps)=z_0, \quad a(\theta,-\eps)=a_0\,.\]
The initial will be assumed to satisfy the follow support assumptions
\[\supp(w_0-\kappa_0)\cup \supp (z_0)\cup\supp(a_0) \subset\left[ -\frac{M\eps}{2} ,\frac{M\eps}{2}\right]\,,
\]
where $\kappa_0>0$ will be a predetermined constant. 

We will further decompose $w_0$ as a sum
\begin{align} \label{datum:0}
w_0=\underbrace{\kappa_0+\eps^{\frac14} \bar{W_2}\left(\eps^{-\frac54}\theta\right) \chi(\eps^{-1}\theta) + \eps^{\frac 1 4} \hat{W}_0(\eps^{- \frac 5 4} \theta)}_{=: \check w_0(\theta)}+ \eps^{\frac 1 4} \Big(\alpha (\eps^{- \frac 5 4} \theta)^2+\beta (\eps^{- \frac 5 4} \theta)^3 \Big) \chi(\eps^{- \frac 5 4} \theta)\,. 
\end{align}
for some smooth fixed cut-off, $\chi$, satisfying $\chi(x) = 1$ for $|x| \le 1$ and is supported in a ball of radius $2$. Above the constants $\alpha, \beta$ are determined by $\hat W_0$ and are not free parameters that we choose as part of the data. The perturbation $\hat W_0$ will be assumed to satisfy the following
\begin{align}
\norm{\hat W_0}_{C^8\left(\left[ -\frac{M\eps}{2} ,\frac{M\eps}{2}\right]\right)} &\leq \eps^2\label{eq:C8:bnd}
\\
\hat W_0^{(n)}(0)&=0,\quad\mbox{for } n=0,1,4, 5\label{eq:W0:diff}\\ \label{est:hatW:in}
\abs{\hat W_0^{(n)}(0)}&\leq \eps,\quad\mbox{for } n=2, 3 \,.
\end{align}
We also assume the following bounds on $z_0$ and $a_0$
\begin{align*}
\norm{z_0}_{C^8} + \norm{a_0}_{C^8} &\leq \eps^2\,.
\end{align*}

\subsection{Main theorem}

We now state our main theorem:
\begin{theorem}\label{thm:general}
Let  $\gamma>1$ be given and set $ \lambda = {\tfrac{\gamma-1}{2}}$. Then there are values $\underline{\kappa_0}, \underline{M}, \underline{\eps}$, $0 < \underline{\eps} < 1$, such that for any $\kappa_0 \in [\underline{\kappa_0}(\lambda), \infty), M \in [\underline{M}(\lambda, \kappa_0), \infty),$ and $\eps \in (0, \underline{\eps}(\lambda, \kappa_0, M)]$, the following holds: 

Let $(w_0,z_0,a_0)$ be initial data satisfying the assumptions stipulated in Section \ref{ss:initial}, with the constants $\alpha$ and $\beta$ are left to be chosen. Then, there exists $\alpha$, $\beta$ satisfying $\abs{\alpha}+\abs{\beta}\leq \eps^{\frac{9}{10}}$ and a corresponding solution $(a,z,w) \in C([-\eps,T_*); C^8(\T))$ to  \eqref{eq:euler:wza}  satisfying the following properties:
\begin{itemize} 
\item The solution forms a singularity at a computable time $T_*$ and angle $\theta_*$.
\item $\sup_{t\in[-\eps, T_*)} \left( \| a\|_{ W^{1, \infty }(\T)} + \| z\|_{ W^{1, \infty} (\T)} + \| w\|_{ L^\infty(\T)} \right) \leq  C_M,$
\item $\lim_{t \to T_*}   \p_\theta w(\xi(t),t) = -\infty $ and $\frac{1}{2(T_*-t)} \leq  \norm{\p_\theta w(\cdot,t)}_{L^\infty} \leq \frac{2}{T_*-t}$ as  $t \to T_*$,
\item $w( \cdot , T_*)$ has a cusp singularity of H\"{o}lder $C^ {\sfrac{1}{5}} $ regularity
\end{itemize} 
Moreover, $w$ blows up in an asymptotically self-similar manner described by the profile $\bar{W}_2$. Specifically, there exist a $\nu>0$ and $\kappa_*$ such that
\begin{align} \label{thm:state:1}
\lim_{t\rightarrow T_{\ast}} (T_{\ast}-t)^{\frac1{4}}\left(w(x(T_{\ast}-t)^{\frac{5}{4}} + \xi_{\ast},t)-\kappa_*\right)= \left(\frac{\nu}{120}\right)^{-\frac14}\bar{W}_2\left(\left(\frac{\nu}{120}\right)^{\frac14}x\right)\,,
\end{align}
where $\nu, \kappa_*, \xi_{\ast}$ are explicitly computable, and satisfy the bounds $\abs{\nu-120}\leq \eps^{\frac 3 4}$, $\abs{\kappa_0-\kappa_*}\leq \eps$, and $|\xi_{\ast}| \le 4 \kappa_0 \eps$. The variable $x$ appearing in \eqref{thm:state:1} can be thought of as a self-similar spatial variable as $t$ approaches $T_{\ast}$. 
\end{theorem} 

As a corollary, we show that Theorem \ref{thm:general} is stable modulo a two dimensional linear subspace of initial data:
\begin{corollary}\label{c:open}
There exists an open set $\Xi$ of initial data $(\check w_0,z_0,a_0)$ in the $C^8$ for which we have the following: for every $(\check w_0,z_0,a_0)\in \Xi$ there exist $\alpha,\beta\in \mathbb R$ such that if we define $w_0$ by \eqref{datum:0} then the conclusion of Theorem \ref{thm:general} holds for initial data $(w_0,z_0,a_0)$.
\end{corollary}

\subsection{Modulation variables and unstable ODEs at $x = 0$}

In order to isolate the self-similar profile, we will need to introduce modulated self-similar variables. These modulation variables allow one to control the time, location, and amplitude of the eventual shock. The idea of using modulation variables is by now classical (cf.\ \cite{Merle96, MeZa97, MeRa05}). We give the precise definitions of our self-similar variables and modulation variables in Section \ref{sec:var}, but to facilitate the forthcoming discussion, let us consider the self-similar quantities $(W, Z, A)$ defined through $w(\theta, t) = e^{- \frac{s}{4}} W(x, s) + \kappa(t),  z(\theta, t) = Z(x, s)$ and $a(\theta, t) = A(x, s)$, where we rescale time via $s =- \log( \tau - t)$ and space via $x = \frac{\theta - \xi(t)}{(\tau - t)^{\frac 5 4}}$.

In our case, we introduce the dynamical modulation variables $\tau, \xi, \kappa$ found in \eqref{sec3:1}, \eqref{sec3:2} to enable us to constrain 
\begin{align} \label{const:intro}
W(0, s) = 0, \qquad \p_x W(0, s) = -1, \qquad \p_x^4 W(0, s) = 0\,,  
\end{align}
where the final constraint is notably different than in the works \cite{BuShVi2019,BuShVi2019b,BuShVi2019c}, and reflects the different structure of the Taylor coefficients at $x = 0$ of the self-similar profile $\bar{W}_2$. 

In so doing, we obtain from \eqref{eq:w:evo} - \eqref{eq:a:evo} the system that we ultimately analyze, which 
\begin{align} \label{W:0:i}
(\p_ s- \frac 1 4) W + (g_W + \frac{5}{4}x ) \p_x W &= -e^{- \frac 3 4 s} \frac{\dot{\kappa}}{1 - \dot{\tau}} +  F_W, \\ \label{Z:0:i}
\p_s Z + (g_Z + \frac 5 4 x ) \p_x Z &= F_Z\,, \\ \label{A:0:i}
\p_s A + (g_A + \frac 5 4 x) \p_x A &= F_A\,. 
\end{align}
Above, the quantities $g_W, g_Z, g_A$ are transport speeds, and $F_W, F_Z, F_A$ are forcing terms that we also leave unspecified for the purposes of this discussion. The reader may find the precise definitions in \eqref{gw:def} - \eqref{ga:def} and \eqref{def:FW} - \eqref{def:FA}. 

In addition, we control the evolution of $\tau, \xi, \kappa$ through ODEs obtained by restricting to the constrains, \eqref{const:intro}. Importantly the three modulation variables enable us to constrain only the three quantities appearing in  \eqref{const:intro}. However, a feature of \eqref{limit:1} with $i \ge 2$ is that $W^{(2)}(0, s)$ and $W^{(3)}(0, s)$ need to be zero in the limit as $s \rightarrow \infty$. This in turn cannot be enforced by the introduction of further modulation variables due to the lack of further symmetries in the compressible Euler equations, and so must be enforced by the choosing initial data on a codimension two manifold, $\mathcal{M}$.  

The equations describing the second and third derivatives of $W$ at $x = 0$ are given by
\begin{align} \label{super:1}
&(\p_s - \frac 3 4) W^{(2)}(0, s) = \text{rapidly decaying forcing terms}\,, \\ \label{super:2}
&(\p_s - \frac{1}{2}) W^{(2)}(0, s) = \text{rapidly decaying forcing terms}\,.
\end{align} 
One can see the instability of the manifold due to the negative damping coefficients appearing on the left-hand side of \eqref{super:1} - \eqref{super:2}. Indeed, negatively damped ODEs such as \eqref{super:1} - \eqref{super:2} generically grow as $s \rightarrow \infty$, but certain data (as determined by the right-hand side) can lead to decaying solutions. 

In the context of the Euler equations, the right-hand sides above themselves depend on other elements of the system (such as the modulation variables, and other derivatives of $(W, Z, A)$). For this reason, we are led to develop a Newton scheme which identifies $\mathcal{M}$. 

\subsection{An iterative scheme to search for unstable solutions}

In this subsection, we briefly discuss the Newton scheme that can be used to identify a manifold of initial data which leads to a globally decaying solution to \eqref{super:1} - \eqref{super:2}. For the present discussion, we focus on a model ODE problem. We consider 
\begin{align} \label{ODE:model}
(\p_s - \frac 12) u_{\alpha} = g + \eps f(u_{\alpha}), \qquad u_\alpha(0) = \alpha\,. 
\end{align}

\noindent We assume for now that the forcing, $g$, has sufficiently strong decay and the nonlinearity, $f$, is an explicit quadratic nonlinearity via
\begin{align} \label{as:setup}
|g| \lesssim e^{- \gamma s}, \qquad f(u) = u^2, \qquad \gamma > 0\,. 
\end{align}

For general data, $\alpha$, writing the solution to \eqref{ODE:model} via the Duhamel formula yields 
\begin{align} \label{aloha:1}
u_\alpha(s) = e^{\frac s 2} \alpha + e^{\frac s 2} \int_0^s e^{- \frac{s'}{2}} g(s') \ud s' + \eps e^{\frac s 2} \int_0^s e^{- \frac{s'}{2}} u_\alpha(s')^2 \ud s'\,. 
\end{align}
From \eqref{aloha:1}, it is that even under the assumption of $g$ decaying exponentially one cannot expect the solution $u_\alpha$ to decay to zero as $s \rightarrow \infty$ for \textit{generic data}, $\alpha$. Thus, to obtain decaying solutions to \eqref{ODE:model}, one needs to find a manifold of data (in this example, a manifold of codimension one). In the case of this ODE, this amounts to finding a \textit{particular value} of $\alpha$ which ensures a globally decaying solution.   

To illustrate how to find this choice of $\alpha$, we now consider the linear version of \eqref{ODE:model} (setting $\eps = 0$ in \eqref{ODE:model}). Upon setting $\eps = 0$ in \eqref{aloha:1}, sending $s \rightarrow \infty$, and demanding the asymptotic behavior $u_\alpha(s) \rightarrow 0$ as $s \rightarrow \infty$, we obtain the following relation   
\begin{align*}
\alpha_0 + \int_0^\infty e^{- \frac{s'}{2}} g(s') \ud s' = 0\,,
\end{align*}
which links the choice of data, $\alpha_0$, to the forcing, $g$, and guarantees the solution $|u_\alpha(s)| \lesssim e^{- \gamma s}$ inherits the decay of $g$.  

We would now like to modify the choice of data, $\alpha_0$, by an $\eps$ perturbation in order to account for the nonlinear effects when $\eps > 0$ in \eqref{ODE:model}. The overall strategy will be to fix a sequence of times $\{s_n\}$ for $n \in \mathbb{N}$ with the property that $s_n \rightarrow \infty$ as $n \rightarrow \infty$, and a corresponding sequence of data choices $\{ \alpha_n\}$ for $n \in \mathbb{N}$ so that $u_{\alpha_n}(s_n) = 0$. With suitably strong estimates, we will show that $\alpha_n \rightarrow \alpha_\infty$ and correspondingly $u_{\alpha_\infty}(s) \rightarrow 0$ as $s \rightarrow \infty$. To compute the iterate of $\alpha_{n+1}$ requires an application of the Implicit Function Theorem, which in turn requires sufficiently strong estimates on the solution. 

Let us now take the particular selection of times, $s_n = n$. To initiate the induction, we will choose $\alpha_0 = 0$, and $u_0(s)$ the corresponding solution (clearly, $u_0(s_0) = u_0(0) = \alpha_0 = 0$). We describe now the $n \rightarrow n+1$ step of the iteration. We now assume inductively that there exists a choice of $\alpha_n$ so that $u_{\alpha_n}(s_{n}) =  0$ and describe the choice of $\alpha_{n+1}$, which is achieved through the Implicit Function Theorem.

We define now the map $\mathcal{T}_n$ given by $\mathcal{T}_n(\alpha) := u_{\alpha}(s_{n+1})$. We now seek an $\alpha_{n+1}$ in a small neighborhood, $\mathcal{B}_n$, of $\alpha_n$ so that $\mathcal{T}_{n}(\alpha_{n+1}) = 0$. According to a Taylor expansion of $\mathcal{T}_n$ in $\alpha$, we obtain for some $\alpha_\ast$ satisfying $|\alpha_\ast - \alpha_n| \le |\alpha_n - \alpha_{n+1}|$,   
\begin{align*}
\mathcal{T}_n(\alpha_{n+1}) = \mathcal{T}_n(\alpha_n) + (\alpha_{n+1} - \alpha_n) \frac{\p \mathcal{T}_n}{\p \alpha}(\alpha_n) + \frac 1 2 (\alpha_{n+1} - \alpha_n)^2 \frac{\p^2 \mathcal{T}_n}{\p \alpha^2}(\alpha_\ast)\,.
\end{align*}
Accordingly, we may apply the Implicit Function Theorem to identify a $\alpha_{n+1}$ so that the left-hand side is zero if we can obtain three estimates: an upper bound on  $| \mathcal{T}_n(\alpha_n)|$,  a lower bound on $ \frac{\p \mathcal{T}_n}{\p \alpha}(\alpha_n)$, and an upper bound over $\sup_{\alpha_\ast \in \mathcal{B}_n} |\frac{\p^2 \mathcal{T}_n}{\p \alpha^2}|$.

We thus define the \emph{error at the next time scale} created by this solution as $E_n := u_{\alpha_{n}}(s_{n+1})$, which the new choice of $\alpha_{n+1}$ must rectify in order to achieve the condition $u_{\alpha_{n+1}}(s_{n+1}) = 0$. The first main estimate in the scheme is thus careful control of this error, $E_n$, throughout the iteration. Specifically, using backwards integration from $s_n$, we may obtain the decay estimate 
\begin{align*}
|E_n| = |\mathcal{T}_n(\alpha_n)| \lesssim e^{- \gamma s_n}\,. 
\end{align*}

Lower bounds on $\frac{\p \mathcal{T}_n}{\p \alpha}$ are achieved by differentiating the forward integration formula, \eqref{aloha:1} in $\alpha$, as this formula importantly holds for all $\alpha$. A simple inspection shows that we may expect $\frac{\p \mathcal{T}_n}{\p \alpha} \sim e^{\frac s 2}$. Third, an upper bound of $\sup_{\alpha_\ast \in [\alpha_n, \alpha_{n+1}]} |\frac{\p^2 \mathcal{T}_n}{\p \alpha^2}|$ can also be computed by differentiating twice \eqref{aloha:1} in $\alpha$. 

\subsection{Notational Conventions}

We now discuss some notational conventions that we will be using throughout the analysis. First, for a function $f = f(x, s)$, we use $\| f \|_\infty = \sup_{x} |f(s, x)|$, that is $L^\infty$ refers to in the $x$ variable only. Next, we define the bracket notation $\langle x \rangle := \sqrt{1 + x^2}$. Lastly, we will often use $A \lesssim B$ to mean $A \le CB$, where $C$ is a universal constant independent of $M, \eps, \kappa_0$. We will use $A \lesssim_M B$ to mean $A \le CB$ where $C$ is a constant that can depend on $M$. 
 
 \section{Preliminaries to the analysis}

\subsection{Self-similar variables and derivation of equations} \label{sec:var}

We will employ the notation
\begin{align*}
\beta_{\tau}=\frac{1}{1-\dot\tau},\quad\beta_1=\frac{1}{1+\lambda},\quad\beta_2=\frac{1-\lambda}{1+\lambda},\quad\beta_3=\frac{1-2\lambda}{1+\lambda},\quad\beta_4=\frac{3+2\lambda}{1+\lambda},\quad\beta_5=\frac{\lambda}{2+2\lambda}\,.
\end{align*}

We now introduce the change of coordinates that we work in and the relevant modulation variables. We define our self-similar temporal and spacial variables as
\begin{align} \label{sec3:1}
s =- \log( \tau - t), \qquad x = \frac{\theta - \xi(t)}{(\tau - t)^{\frac 5 4}}\,.
\end{align}
We record the following identities
\begin{align*}
&\frac{\p s}{\p t} = (1 - \dot{\tau}) e^s, \qquad \frac{\p x}{\p t} = \frac{5}{4} (1 - \dot{\tau}) x e^s - \dot{\xi} e^{\frac 5 4 s}, \qquad \frac{\p x}{\p \theta} = e^{\frac 5 4 s}\,.
\end{align*}
We now introduce the new unknowns, $W, Z, A$ which are defined through the following relations
\begin{align} \label{sec3:2}
w(\theta, t) = e^{- \frac{s}{4}} W(x, s) + \kappa(t), \qquad z(\theta, t) = Z(x, s), \qquad a(\theta, t) = A(x, s)\,. 
\end{align}
In order to solve for the three modulation variables  $\kappa$, $\tau$ and $\xi$, we enforce the following constraints
\begin{align}\label{e:constraints}
W(0, s) = 0, \qquad \p_x W(0, s) = - 1, \qquad \p_x^4 W(0, s) = 0\,.
\end{align}

We now record the following calculations 
\begin{align} 
\p_t w= & - \frac{1-\dot{\tau}}{4} e^{\frac{3}{4}s} W + (1-\dot{\tau}) e^{\frac{3s}{4}} \p_s W + \dot{\kappa} \label{BUDS:1}
 + \frac{5}{4}(1 - \dot{\tau}) x e^{\frac{3}{4}s} \p_x W - \dot{\xi} e^{s} \p_x W\,, \\ \label{BUDS:2}
\p_\theta w = & e^{ s} \p_x W\,. 
\end{align}

\noindent Next, we record the calculations 
\begin{align} 
\p_t z 
=  (1 - \dot{\tau}) e^s \p_s Z + (\frac 5 4 (1 - \dot{\tau}) x e^s - \dot{\xi} e^{\frac 5 4 s}) \p_x Z\label{whilk:miskey:1} ,\quad
\p_\theta z=   e^{\frac 5 4 s} \p_x Z\,. 
\end{align}
and similarly, 
\begin{align} 
\p_t a = (1 - \dot{\tau}) e^s \p_s A + (\frac 5 4 (1 - \dot{\tau}) x e^s - \dot{\xi} e^{\frac 5 4 s}) \p_x A,  \label{whilk:miskey:4}\quad
\p_\theta a=  \ e^{\frac 5 4 s} \p_x A\,.
\end{align}

Then in self-similar variables \eqref{eq:w:evo} becomes 
\begin{align} \n
&(\p_s - \frac 1 4) W + \left(\frac 5 4 x - \beta_\tau (\dot{\xi} - \kappa) e^{\frac 1 4 s} +\beta_{\tau}(\beta_2 e^{\frac 1 4 s} Z+ W)\right)\p_x W  \\ 
&\quad= - \beta_\tau e^{- \frac 3 4 s}\dot{\kappa} -  \beta_\tau e^{- \frac 3 4 s} A \Big( \beta_3 Z + \beta_4 (e^{- \frac s 4}W + \kappa) \Big)\,.\label{eq:W:0}
\end{align}
Similarly, we rewrite \eqref{eq:z:evo} as
\begin{align} 
&\p_s Z + \left(\frac 5 4 x+ \beta_{\tau}(e^{\frac 1 4 s}(\beta_2\kappa-\dot{\xi} +Z)+\beta_2 W)\right) \p_x Z    \label{eq:Z:0}
=  -  \beta_\tau e^{-s} A \Big(\beta_3(e^{- \frac s 4}W+ \kappa) + \beta_4Z\Big)\,,
\end{align}
and \eqref{eq:a:evo} as
\begin{align} \n
&\p_s A + \left(\frac 5 4 x+ \beta_{\tau}(e^{\frac 1 4 s}(\beta_1\kappa-\dot{\xi} +\beta_1Z)+\beta_1W)\right) \p_x A \\ 
&\quad=  -  2\beta_\tau\beta_1e^{-s} A^2 + \frac{1}{2}\beta_\tau\beta_1  e^{-s} \Big( e^{- \frac s 4} W + \kappa + Z \Big)^2  \label{eq:A:0}
-  \beta_\tau\beta_5 e^{-s} \Big( e^{- \frac s 4} W + \kappa - Z \Big)^2\,.
\end{align}

We now compactify the above equations by introducing the following transport speeds
\begin{align} \label{gw:def}
&g_W :=\beta_{\tau}W- \beta_\tau (\dot{\xi} - \kappa) e^{\frac 1 4 s} + \beta_{\tau}\beta_2e^{\frac 1 4 s} Z
=:\beta_{\tau}W+G_W\,,\\ \label{gz:def}
&g_Z :=\beta_{\tau}\beta_2W+ \beta_{\tau}e^{\frac 1 4 s}(\beta_2\kappa-\dot{\xi} +Z)
=:\beta_{\tau}\beta_2W+G_Z \,,\\ \label{ga:def}
&g_A := \beta_{\tau}\beta_1W+ \beta_{\tau}e^{\frac 1 4 s}(\beta_1\kappa-\dot{\xi} +\beta_1Z)=:
\beta_{\tau}\beta_1W+G_A \,,  
\end{align}

\noindent and forcing terms
\begin{align} \label{def:FW}
F_W &:=-\beta_\tau e^{- \frac 3 4 s} A \Big( \beta_3 Z +\beta_4 (e^{- \frac s 4}W + \kappa) \Big)\,,\\ \label{def:FZ}
F_Z& := -  \beta_\tau e^{-s} A \Big(\beta_3(e^{- \frac s 4}W+ \kappa) + \beta_4Z\Big)\,,\\ \label{def:FA}
F_A&:=  -  2\beta_\tau\beta_1e^{-s} A^2 + \frac{1}{2}\beta_\tau\beta_1  e^{-s} \Big( e^{- \frac s 4} W + \kappa + Z \Big)^2-  \beta_\tau\beta_5 e^{-s} \Big( e^{- \frac s 4} W + \kappa - Z \Big)^2\,.
\end{align}
We note that the quantities $G_W, G_Z, G_A$ are defined through the second equalities in \eqref{gw:def} - \eqref{ga:def}.  

With these definitions, our equations become 
\begin{align} \label{W:0}
(\p_ s- \frac 1 4) W + (g_W + \frac{5}{4}x ) \p_x W &= -e^{- \frac 3 4 s} \frac{\dot{\kappa}}{1 - \dot{\tau}} +  F_W\,, \\ \label{Z:0}
\p_s Z + (g_Z + \frac 5 4 x ) \p_x Z &= F_Z\,, \\ \label{A:0}
\p_s A + (g_A + \frac 5 4 x) \p_x A &= F_A\,. 
\end{align}

Further, it will be convenient to introduce the notation
\begin{equation*}
\mathcal V_W:=g_W + \frac{5}{4}x,\qquad
\mathcal V_Z:=g_Z + \frac{5}{4}x,\qquad
\mathcal V_A:=g_A + \frac{5}{4}x\,.
\end{equation*}
so that we obtain
\begin{align} \label{basic:w}
(\p_ s- \frac 1 4) W + \mathcal V_W \p_x W &= -e^{- \frac 3 4 s} \frac{\dot{\kappa}}{1 - \dot{\tau}} +  F_W\,, \\ \label{basic:z}
\p_s Z + \mathcal V_Z \p_x Z &= F_Z\,, \\ \label{basic:a}
\p_s A +  \mathcal V_A \p_x A &= F_A\,. 
\end{align}

We define now the combination 
\begin{align} \label{def:mu}
\mu :=- \beta_{\tau} (\dot{\xi} - \kappa)e^{\frac s 4} +\beta_{\tau}\beta_2e^{\frac14 s}Z(0,s) = G_W(s, 0)\,. 
\end{align}

\subsection{An unstable self-similar solution to Burgers' equation} \label{s:burgers}

Here we develop properties of the self-similar Burgers profile, $\bar{W}:= \bar{W}_2$, which solves the equation 
\begin{align} \label{Burger:1}
- \frac 1 4 \bar{W} + (\bar{W} + \frac 5 4 x) \bar{W}_x = 0\,. 
\end{align}

According to \cite{Fontelos}, \eqref{Burger:1} has an implicit solution 
\begin{align}\label{e:implicit:eq}
x = - \bar{W} - \bar{W}^{5}\,.
\end{align} 
Differentiating yields
\begin{align}\label{e:w1:implicit}
 \bar{W}^{(1)}=-\frac{1}{1+5{\bar W}^4}\,.
\end{align} 
Hence $ \bar{W}^{(1)}\leq 0$ and thus since $\bar{W}(0)=0$ we attain that $\bar{W}\leq 0$ for $x\geq 0$. By Young's inequality and applied to \eqref{e:implicit:eq}, we have
\begin{align*}
x \leq - \bar{W}  -\bar{W}^{5} \leq -\frac{\bar{W}^{5} }{5x^4}+-{\bar W}^5+\frac{4x}{5}\,.
\end{align*}
Rearranging, we obtain
\[-{\bar W}^5\geq \frac{x^5}{5(5+x^4)}\,. \]
This lower bound combined with \eqref{e:w1:implicit} yields 
\begin{equation}\label{bound:bar:W:1}
\abs{ \bar{W}^{(1)}}\leq (1+x^4)^{-\frac15}\,.
\end{equation}
Similarly, using Young's inequality and \eqref{e:implicit:eq} we have
\begin{equation*}
 - \bar{W}^{5}\leq 5x+1\,,
\end{equation*}
from which we obtain the estimate
\begin{equation*}
\abs{\bar{W}}\leq \frac32 (1+x^4)^{\frac1{20}}\,.
\end{equation*}
Finally, differentiating \eqref{e:implicit:eq} $5$ times, we obtain
\begin{equation} \label{W5:non}
W^{(5)}(0)=120\,.
\end{equation}

We now define the weight function 
\begin{align} \label{weight:eta}
\eta_\gamma := (1 + x^4)^\gamma, \text{ for any } \gamma \in \mathbb{R}\,.
\end{align}

We now record the following lemma, which summarizes the properties of $\bar{W}$ that we will be using
\begin{lemma} Let $\ell$ be sufficiently small relative to universal constants. For $n = 2, 3, 4$ at $x=0$ we have
\begin{align}\label{fifth:deriv:bar:W}
&\bar{W}(0) = 0\,, \quad \bar{W}^{(1)}(0) = -1\,, \quad \bar{W}^{(n)}(0) = 0\,, \quad \bar{W}^{(5)}(0) = 120 \,.
\end{align}
Furthermore, for $n \ge 2$, $\bar W$ satisfies the estimates
\begin{gather}\label{decay:bar:2}
|\bar{W}| \le \frac 3 2 \eta_{\frac{1}{20}}\,,  \quad|\bar{W}^{(1)}| \le \eta_{- \frac 1 5}\,, \quad  |\bar{W}^{(n)}| \le C_k \eta_{- \frac 1 5 - \frac{n}{4}} \,,\\ \label{truth:1}
-1 + \frac{l^7}{50} \le \bar{W}^{(1)} \le 0 \quad\text{ for } |x| \ge \ell\,.
\end{gather}
\end{lemma}

\subsection{Higher order $x$ derivatives}

In this section we list the higher order derivatives of $(W,Z,A)$. It will be convenient to introduce the notation: 
\begin{align*}
f^{(n)}(s, x) := \p_x^n f(s, x)\,. 
\end{align*}

We will derive now up to eight derivatives of the above system. 
\begin{align} \label{W:n}
\Big( \p_s + \frac 1 4 (- 1 + 5n) + \beta_{\tau}(n+1_{n> 1}) W^{(1)}  \Big) W^{(n)} +\mathcal V_W \p_x W^{(n)} &= F_{W,n}\,, \\ \label{Z:n}
(\p_s + \frac{5n}{4} +  n\beta_{\tau}\beta_2 W^{(1)}) Z^{(n)} + \mathcal V_Z \p_x Z^{(n)} &= F_{Z, n}\,, \\ \label{A:n}
(\p_s + \frac 5 4 n + n\beta_{\tau}\beta_1 W^{(1)} ) A^{(n)} +\mathcal V_A\p_x A^{(n)} &= F_{A,n}\,,
\end{align}
 where the forcings are defined by
\begin{align} \label{F.W.n}
&F_{W,n} := F_W^{(n)} - 1_{n \ge 3}\beta_\tau \sum_{j = 2}^{n-1} \binom{n}{j} W^{(j)} W^{(n+1 - j)} - \sum_{j = 1}^n \binom{n}{j} G_W^{(j)} W^{(n+1-j)}\,, \\ \label{F.Z.n}
&F_{Z, n} := F_Z^{(n)} - 1_{n \ge 2}\beta_\tau \beta_2 \sum_{j = 2}^{n} \binom{n}{j}  W^{(j)} Z^{(n+1 - j)} - \sum_{j = 1}^n \binom{n}{j} G_Z^{(j)} Z^{(n+1-j)}\,, \\ \label{F.A.n}
&F_{A,n} := F_A^{(n)} -1_{n \ge 2} \beta_\tau \beta_1\sum_{j = 2}^{n} \binom{n}{j}  W^{(j)} A^{(n+1 - j)}- \sum_{j = 1}^n \binom{n}{j} G_A^{(j)} A^{(n+1-j)}\,.
\end{align}

For repeated future reference, we record here the following expressions which are obtained by differentiating \eqref{def:FW} (for $n \ge 1$) 
\begin{align} \label{FW:to:the:n}
F_W^{(n)} = &- \beta_\tau e^{- \frac 3 4 s} \sum_{j = 0}^n \binom{n}{j} A^{(j)} \Big( \beta_3 Z^{(n-j)} + \beta_4 (e^{- \frac s 4} W + \kappa)^{(n-j)} \Big)\,, \\ 
F_{Z}^{(n)} = &- \beta_\tau e^{-s} \sum_{j = 0}^{n} \binom{n}{j}  A^{(j)} \Big(\beta_3 (e^{- \frac s 4}W+ \kappa)^{(n-j)} + \beta_4 Z^{(n-j)}\Big)  \label{def:FZn}
\,, \\ \n
F_A^{(n)} = & - 2 \beta_\tau \beta_1 e^{-s} \sum_{j = 0}^n \binom{n}{j} A^{(j)} A^{(n-j)} \\ \n
& + \frac 1 2 \beta_\tau \beta_1 e^{-s} \sum_{j = 0}^n \binom{n}{j} (e^{- \frac s 4}W + \kappa + Z)^{(j)} (e^{- \frac s 4}W + \kappa + Z)^{(n-j)} \\ \label{okey:1}
& - \beta_\tau \beta_1 e^{-s} \sum_{j = 0}^n \binom{n}{j} (e^{- \frac s 4}W + \kappa - Z)^{(j)} (e^{- \frac s 4}W + \kappa - Z)^{(n-j)}\,.
\end{align}

\noindent By combining \eqref{F.W.n} with \eqref{FW:to:the:n}, we obtain the expression 
\begin{align} \n
F_{W,n} = & - \beta_\tau e^{- \frac 3 4 s} \sum_{j = 0}^n \binom{n}{j} A^{(j)} \Big( \beta_3 Z^{(n-j)} + \beta_4 (e^{- \frac s 4} W + \kappa)^{(n-j)} \Big) \\ \label{F.W.n.bot}
& - 1_{n \ge 3}\beta_\tau \sum_{j = 2}^{n-1} \binom{n}{j} W^{(j)} W^{(n+1 - j)} - \sum_{j = 1}^n \binom{n}{j} G_W^{(j)} W^{(n+1-j)}\,.
\end{align}

\noindent By combining \eqref{F.Z.n} with \eqref{def:FZn}, we obtain the final expression
\begin{align} \n
F_{Z,n} = & - \beta_\tau e^{-s} \sum_{j = 0}^{n} \binom{n}{j}  A^{(j)} \Big(\beta_3 (e^{- \frac s 4}W+ \kappa)^{(n-j)} + \beta_4 Z^{(n-j)}\Big) \\ 
& - 1_{n \ge 2}\beta_\tau \beta_2 \sum_{j = 2}^{n} \binom{n}{j}  W^{(j)} Z^{(n+1 - j)}  \label{F.Z.n.bot}
-  \sum_{j = 1}^n \binom{n}{j} G_Z^{(j)} Z^{(n+1-j)}\,.
\end{align}

We now derive the first five constrained ODEs. First, we introduce an important piece of notation to describe the purely $s$-dependent quantities at $x = 0$, 
\begin{align} \label{q:def:q}
q^{(n)}(s) := W^{(n)}(0, s)\,. 
\end{align}
From the equations \eqref{W:0} and \eqref{W:n} , evaluating $W^{(n)}$, for $n=0,\dots,4$ at $x = 0$ and using the constraints \eqref{e:constraints}, we obtain the following system of five ODEs in the $s$ variable
\begin{align} \label{eq:ODE:1}
&- \frac{\mu}{\beta_\tau} + e^{- \frac 3 4 s} \dot{\kappa} = \frac{1}{\beta_\tau} F_W(0, s)\,, \\  \label{eq:ODE:2}
&\dot{\tau} -\frac{1}{\beta_\tau} G^{(1)}_W(0, s) + \frac{\mu}{\beta_\tau} q^{(2)}(s) =\frac{1}{\beta_\tau} F_{W}^{(1)}(0, s)\,, \\ \label{eq:ODE:3}
&(\p_s + \frac 9 4 ) q^{(2)} - 3 \beta_\tau q^{(2)} + \mu q^{(3)} + 2 G_W^{(1)}(0, s) q^{(2)} = F_W^{(2)}(0, s) + G_W^{(2)}(0, s)\,, \\  \label{eq:ODE:4}
&(\p_s + \frac{14}{4}) q^{(3)} - 4 \beta_\tau q^{(3)} + 3 G_W^{(1)}(0, s) q^{(3)} + 3 \beta_\tau |q^{(2)}|^2+ \sum_{j = 2}^3 \binom{3}{j} G_W^{(j)}(0, s) q^{(4-j)}= F_W^{(3)}(0, s)\,, \\  \label{eq:ODE:5}
&q^{(5)} \mu + 10 \beta_\tau q^{(2)} q^{(3)} + \sum_{j = 2}^4 \binom{4}{j} G_W^{(j)}(0, s) q^{(5-j)} = F_W^{(4)}(0, s)\,.
\end{align}
In addition, we will need the evolution equation of $W^{(5)}$ at $x=0$, given by
\begin{align}
\label{eq:ODE:6}
&\partial_s  q^{(5)}= - \mu q^{(6)} + (1 - \beta_\tau) q^{(5)} - 10 |q^{(3)}|^2  - \sum_{j = 1}^{5} \binom{5}{j} G_W^{(j)}(0, s) q^{(6-j)} + F_W^{(5)}(0, s)\,. 
\end{align} 

We also derive the following equation for the difference $\tilde W:=W-\bar W$:
\begin{align}
&(\p_ s- \frac 1 4+\beta_\tau   \bar W^{(1)}) \tilde W +\mathcal V_W \p_x \tilde W = -\beta_\tau e^{- \frac 3 4 s} \dot{\kappa} +  F_{ W}+((\beta_\tau -1)\bar W-G_W)\p_x \bar W:=\tilde F_W\label{diff:eq0}\,.
\end{align}
The equation for the higher order derivatives $W^{(n)}$ is given by
\begin{align} \n
&\p_s \tilde{W}^{(n)} + \Big( \frac 1 4 (-1 + 5n) + \beta_{\tau}\left(\bar W^{(1)}+ nW^{(1)} \right)  \Big) \tilde{W}^{(n)} + \mathcal V_W \p_x \tilde{W}^{(n)}   \\ \n
 &\qquad= F_{W}^{(n)} - 1_{n \ge 2}\beta_\tau \sum_{j = 2}^{n-1} \binom{n}{j}  W^{(j)} \tilde W^{(n+1 - j)}- \sum_{j = 1}^n \binom{n}{j} \left(\beta_\tau \bar W^{(j+1)}\tilde W ^{(n-j)}+G_W^{(j)}\tilde W ^{(n+1-j)}\right) \\ \n
& \qquad \qquad + (\beta_\tau - 1) \sum_{j = 0}^n \binom{n}{j} \bar{W}^{(j)} \bar{W}^{(n+1-j)} - \sum_{j = 0}^n \binom{n}{j} G_W^{(j)} \bar{W}^{(n+1-j)} \\ \label{diff:eq}
& \qquad =: \tilde{F}_{W,n}\,.
\end{align}

\subsection{$\nabla_{\alpha, \beta}$ derivatives}

We introduce the following notation to compactify the forthcoming equations
\begin{align} \label{c:deriv}
f_{c} := \p_{c} f, \qquad c \in \{ \alpha, \beta \}\,,
\end{align}
for any function $f$.

\subsubsection{$\nabla_{\alpha, \beta}$ derivatives of $Z$}

We first take $\p_c$ of equation \eqref{Z:0} which produces 
\begin{align} \label{every:time:1}
\p_s Z_c + \mathcal{V}_Z \p_x Z_c = \p_c F_Z - Z^{(1)} \Big( \dot{\tau}_c \beta_\tau^2 \beta_2 W + \beta_\tau \beta_2 W_{c} + \p_c G_Z \Big) =: F_{Z,0}^{c}\,.
\end{align} 
We now use \eqref{def:FZ} to evaluate the $\p_c F_Z$ term appearing above via 
\begin{align} \label{pcFz:0}
\p_c F_Z = \dot{\tau}_c \beta_\tau F_Z - \beta_\tau e^{-s} A_c ( \beta_3 (e^{- \frac s 4}W + \kappa) + \beta_4 Z) - \beta_\tau e^{-s} A(\beta_3 (e^{- \frac s 4} W_c + \kappa_c) + \beta_4 Z_c)
\end{align}

We next compute $\p_x^n$ of equation \eqref{every:time:1} to obtain 
\begin{align} \n
&(\p_s + \frac 5 4 n + n \beta_\tau \beta_2 W^{(1)}) Z_c^{(n)} + \mathcal{V}_Z \p_x Z_c^{(n)} \\  \n
&\quad= \p_c F_Z^{(n)}  - \sum_{j = 0}^n \binom{n}{j} \dot{\tau}_c \beta_\tau^2 \beta_2 Z^{(j+1)} W^{(n-j)} - \sum_{j = 0}^n \binom{n}{j} \beta_\tau \beta_2 Z^{(j+1)} W_c^{(n-j)}\\  \n
&\qquad - \sum_{j = 0}^n \binom{n}{j} Z^{(1+j)} \p_c G_Z^{(n-j)} - 1_{n \ge 1} \sum_{j = 1}^n \binom{n}{j} G_Z^{(j)} Z_c^{(n+1-j)}  \\ \label{Midterms:1}
&\qquad - 1_{n \ge 2} \sum_{j = 2}^n \binom{n}{j} \beta_\tau \beta_2 W^{(j)} Z^{(n-j+1)}_c =: F_{Z,n}^{c}\,. 
\end{align}

\noindent We now compute the expression for $\p_c F_Z^{(n)}$ by computing $\p_x^n$ of \eqref{pcFz:0} which yields
\begin{align} \n
\p_c F_Z^{(n)} = &  \dot{\tau}_c \beta_\tau F_Z^{(n)} - \beta_\tau e^{-s} \sum_{j = 0}^n \binom{n}{j} A_c^{(j)} \Big( \beta_3 ( e^{- \frac s 4} W + \kappa)^{(n-j)} + \beta_4 Z^{(n-j)} \Big) \\ \label{Midterms:2}
& - \beta_\tau e^{-s} \sum_{j =0}^n \binom{n}{j} A^{(j)} \Big(  \beta_3 ( e^{- \frac s 4} W_c + \kappa_c)^{(n-j)} + \beta_4 Z_c^{(n-j)} \Big)\,.
\end{align} 

\subsubsection{$\nabla_{\alpha, \beta}$ derivatives of $A$}

We compute $\p_c$ of the basic equation for $A$, \eqref{basic:a}, which yields 
\begin{align} \label{socialite:1}
\p_s A_c + \mathcal{V}_A \p_x A_c = \p_c F_A - \Big( \dot{\tau}_{c} \beta_\tau^2 \beta_1 W + \beta_\tau \beta_1 W_c + \p_c G_A \Big) A^{(1)} =: F_{A,0}^{c}\,. 
\end{align} 

\noindent Computing now the expression $\p_c F_A$ by differentiating \eqref{def:FA}, we obtain 
\begin{align} \n
\p_c F_A = & \dot{\tau}_{c} \beta_\tau F_A + \beta_\tau \beta_1 e^{-s} \Big( e^{- \frac s 4}W + \kappa + Z \Big) \Big( e^{- \frac s 4} W_c + \kappa_c + Z_c \Big) \\ \label{pcFa}
&- 2 \beta_\tau \beta_5 e^{-s} \Big( e^{- \frac s 4} W + \kappa - Z \Big) \Big( e^{- \frac s 4} W_c + \kappa_c - Z_c \Big)\,.
\end{align}

\noindent We now compute $\p_x^n$ of equation \eqref{socialite:1} which produces 
\begin{align} \n
&(\p_s + \frac{5n}{4} + n \beta_\tau \beta_1 W^{(1)} ) A_c^{(n)} + \mathcal{V}_{A} \p_x A_c^{(n)} \\ \n
= & \p_c F_A^{(n)}  - 1_{n \ge 1} \sum_{j = 1}^n \binom{n}{j} G_A^{(j)} A_c^{(n+1-j)} - 1_{n \ge 2} \sum_{j = 2}^n \binom{n}{j} \beta_\tau \beta_1 W^{(j)} A_c^{(n+1-j)} \\ \n
& - \sum_{j = 0}^n \binom{n}{j} \dot{\tau}_c \beta_\tau^2 \beta_1 W^{(j)} A^{(n+1-j)} - \sum_{j = 0}^n \binom{n}{j} \beta_\tau \beta_1 W_c^{(j)} A^{(n+1-j)} \\ \label{Midterms:3}
& - \sum_{j = 0}^n \binom{n}{j} \p_c G_A^{(j)} A^{(n+1-j)} =:  F_{A,n}^{c}\,. 
\end{align}

\noindent We now compute $\p_x^n$ of the expression for $\p_c F_A$ in \eqref{pcFa} which yields
\begin{align} \n
\p_c F_A^{(n)} =& \dot{\tau}_c \beta_\tau F_A^{(n)} + \beta_\tau \beta_1 e^{-s}\sum_{j = 0}^n \binom{n}{j} \Big( e^{- \frac s 4} W + \kappa + Z \Big)^{(j)} \Big( e^{- \frac s 4} W_c + \kappa_c + Z_c\Big)^{(n-j)} \\ \label{Midterms:4}
& - 2 \beta_\tau \beta_5 e^{-s} \sum_{j = 0}^n \binom{n}{j} \Big( e^{- \frac s 4} W + \kappa - Z \Big)^{(j)} \Big( e^{- \frac s 4} W_c+ \kappa_c - Z_c \Big)^{(n-j)}\,.
\end{align}

\subsubsection{$W$ Quantities}

 For the $W$ equations, we separately write down the $n = 0$ system. Differentiating \eqref{W:0} in $c$ yields 
 \begin{align} \n
 &(\p_s - \frac 1 4 + \beta_\tau W^{(1)}) \p_c W + \mathcal{V}_W \p_x \p_c W \\ 
 &\qquad= - e^{- \frac 3 4 s} \beta_\tau \p_c \dot{\kappa} - e^{- \frac 3 4 s} \dot{\kappa} \p_c \dot{\tau} \beta_\tau^2 - \p_c G_W W^{(1)} - W^{(1)} \dot{\tau}_{c} \beta_\tau^2 W + \p_c F_W\,. \label{eq.dcw.0}
 \end{align} 
 
 \noindent By differentiating \eqref{def:FW} in $\p_c$, we obtain 
 \begin{align} \n
 \p_c F_W = &- \p_c \dot{\tau} \beta_\tau^2 e^{- \frac 3 4 s} A \Big( \beta_3 Z + \beta_4 (e^{- \frac s 4}W + \kappa) \Big) - \beta_\tau e^{- \frac 3 4 s} \p_c A \Big( \beta_3 Z + \beta_4 (e^{- \frac s 4} W + \kappa) \Big) \\ \n
 & - \beta_\tau e^{- \frac 3 4 s} A \Big( \beta_3 \p_c Z + \beta_4 (e^{- \frac s 4} \p_c W + \p_c \kappa) \Big) \\ \label{dc.FW}
 = &  \dot{\tau}_c \beta_\tau F_W - \beta_\tau e^{- \frac 3 4 s} A_c \Big( \beta_3 Z + \beta_4 (e^{- \frac s 4}W + \kappa) \Big) - \beta_\tau e^{- \frac 3 4 s} A \Big( \beta_3 Z_c + \beta_4 (e^{- \frac s 4}W_c + \kappa_c) \Big).
 \end{align}
 
 \noindent We combine with \eqref{eq.dcw.0} to obtain 
\begin{align} \label{eq:C:W}
 (\p_s - \frac 1 4 + \beta_\tau W^{(1)}) \p_c W + \mathcal{V}_W \p_x \p_c W = F_{W,0}^{c}\,, 
\end{align}

\noindent where the forcing is given by
\begin{align} \n
 F_{W,0}^{c} := &  \dot{\tau}_c \beta_\tau F_W - \beta_\tau e^{- \frac 3 4 s} A_c \Big( \beta_3 Z + \beta_4 (e^{- \frac s 4} W + \kappa) \Big)  - \p_c G_W W^{(1)} - W^{(1)} \dot{\tau}_{c} \beta_\tau^2 W \\ \label{dc.FW:0}
 & - \beta_\tau e^{- \frac 3 4 s} A \Big( \beta_3 Z_c + \beta_4 (e^{- \frac s 4} W_c + \kappa_c) \Big)  - e^{- \frac 3 4 s} \beta_\tau \p_c \dot{\kappa}   - e^{- \frac 3 4 s} \dot{\kappa} \p_c \dot{\tau} \beta_\tau^2  \,.
 \end{align}
 
 We now take $\p_x^n$ of equation \eqref{eq.dcw.0}. This produces, for $n \ge 1$, 
 \begin{align} \n
 &(\p_s + \frac{5n-1}{4} +(n+1) \beta_\tau W^{(1)}) \p_c W^{(n)} + \mathcal{V}_W \p_x \p_c W^{(n)} \\ \n
 =& - 1_{n \ge 1} \sum_{j = 1}^n \binom{n}{j} \beta_\tau W^{(1+j)} \p_c W^{(n-j)} - 1_{n \ge 2} \sum_{j = 0}^{n-2} \binom{n}{j} \beta_\tau W^{(n-j)} \p_c W^{(j+1)} \\ \n
 & - 1_{n \ge 1} \sum_{j = 0}^{n-1} \binom{n}{j} G_W^{(n-j)} \p_c W^{(j+1)} - \sum_{j = 0}^n \binom{n}{j} \p_c G_W^{(j)} W^{(n-j+1)} \\ \label{trek:mix}
 & - \dot{\tau}_c \beta_\tau^2 \sum_{j = 0}^n \binom{n}{j} W^{(1+j)} W^{(n-j)} +  \p_c \p_x^n F_W  =: F_{W, n}^{c}\,.
 \end{align}

We now use the expression \eqref{dc.FW} compute 
\begin{align} \n
 \p_c F_W^{(n)} = & \dot{\tau}_c \beta_\tau F_W^{(n)} - \sum_{j = 0}^n \binom{n}{j} \beta_\tau e^{- \frac 3 4 s} \p_c A^{(j)} \Big( \beta_3 Z^{(n-j)} + \beta_4 (e^{- \frac s 4} W + \kappa)^{(n-j)} \Big) \\ \label{dcbcFwn}
& - \sum_{j = 0}^n \binom{n}{j} \beta_\tau e^{- \frac 3 4 s} A^{(j)} \Big( \beta_3 \p_c Z^{(n-j)} + \beta_4 (e^{- \frac s 4} \p_c W + \p_c \kappa)^{(n-j)} \Big)\,.
\end{align}

\noindent Combining now with the expression \eqref{trek:mix}, we obtain 
\begin{align} \n
F_{W,n}^{c} := &  \dot{\tau}_c \beta_\tau F_W^{(n)} - \sum_{j = 0}^n \binom{n}{j} \beta_\tau e^{- \frac 3 4 s} \p_c A^{(j)} \Big( \beta_3 Z^{(n-j)} + \beta_4 (e^{- \frac s 4} W + \kappa)^{(n-j)} \Big) \\ \n
& - \sum_{j = 0}^n \binom{n}{j} \beta_\tau e^{- \frac 3 4 s} A^{(j)} \Big( \beta_3 \p_c Z^{(n-j)} + \beta_4 (e^{- \frac s 4} \p_c W + \p_c \kappa)^{(n-j)} \Big) \\ \n
& - \sum_{j = 1}^n \binom{n}{j} \beta_\tau W^{(1+j)} \p_c W^{(n-j)} - 1_{n \ge 2} \sum_{j = 0}^{n-2} \binom{n}{j} \beta_\tau W^{(n-j)} \p_c W^{(j+1)} \\ \n
& - 1_{n \ge 1} \sum_{j = 0}^{n-1} \binom{n}{j} G_W^{(n-j)} \p_c W^{(j+1)} -  \sum_{j = 0}^n \binom{n}{j} \p_c G_W^{(j)} W^{(j+1)} \\ \label{Fncw.fin.2} 
&  - \dot{\tau}_c \beta_\tau^2 \sum_{j = 0}^n \binom{n}{j} W^{(1+j)} W^{(n-j)}\,.
\end{align}

\subsection{$\nabla_{\alpha, \beta}^2$ derivatives}

\subsubsection{$\nabla_{\alpha, \beta}^2$ derivatives of $W$}

We compute $\p_{c_2}$ of \eqref{eq:C:W} which results in 
\begin{align} \n
&(\p_s - \frac 1 4 + \beta_\tau W^{(1)}) W_{c_1 c_2} + \mathcal{V}_W \p_x W_{c_1 c_2} \\ \n
 =&  \p_{c_1 c_2} F_W - \beta_\tau W^{(1)}_{c_2} W_{c_1} - \beta_\tau^2 \dot{\tau}_{c_2} W^{(1)} W_{c_1} - \Big( \beta_\tau^2 \dot{\tau}_{c_2} W + \beta_\tau W_{c_2} + \p_{c_2} G_W \Big) W^{(1)}_{c_1} \\ \n
 & - \dot{\tau}_{c_1} \beta_\tau^2 W W^{(1)}_{c_2} - \dot{\tau}_{c_1 c_2} \beta_\tau^2 W W^{(1)} - 2 \beta_\tau^2 \dot{\tau}_{c_1} \dot{\tau}_{c_2} W W^{(1)} - \dot{\tau}_{c_1} \beta_\tau^2 W^{(1)} W_{c_2} - \mathcal{M}^{c_1, c_2} \\ \label{HAIM:1}
=:& F_{W,0}^{c_1, c_2}\,,
\end{align}
where the modulation terms have been grouped into 
\begin{align} \label{Mod:Mod:1}
\mathcal{M}^{c_1, c_2} := e^{- \frac 3 4 s}\Big( \beta_\tau \dot{\kappa}_{c_1 c_2} + \beta_\tau^2 (\dot{\tau}_{c_2} \dot{\kappa}_{c_1} + \dot{\kappa}_{c_2} \dot{\tau}_{c_1}) + \dot{\kappa} \dot{\tau}_{c_1 c_2} \beta_\tau^2 + 2 \beta_\tau^2 \dot{\kappa} \dot{\tau}_{c_1} \dot{\tau}_{c_2} \Big)\,.
\end{align}

\noindent Similarly we compute $\p_{x}^n$ of \eqref{HAIM:1} which results in the following system for $n \ge 1$
\begin{align} \n
&\Big( \p_s + \frac{5n-1}{4} + (n+1) \beta_\tau W^{(1)} \Big) W_{c_1, c_2}^{(n)} + \mathcal{V}_W \p_x W^{(n)}_{c_1, c_2} \\ \n
= & \p_{c_1 c_2} F_W^{(n)}   - \sum_{i \in \{1, 2\}} \sum_{j = 0}^n \binom{n}{j} \beta_\tau^2 \dot{\tau}_{c_i}  W^{(1+j)} W^{(n-j)}_{c_{i'}} - \sum_{j = 0}^n \binom{n}{j} \beta_\tau W^{(j)}_{c_1} W^{(n+1-j)}_{c_2} \\ \n
& - 1_{n \ge 1} \sum_{j = 1}^n \binom{n}{j} \beta_\tau W^{(1+j)} W^{(n-j)}_{c_1 c_2}  - \sum_{i = \{1, 2\}} \sum_{j = 0}^{n} \binom{n}{j} \beta_\tau^2 \dot{\tau}_{c_{i'}} W^{(j)} W^{(n+1-j)}_{c_i} \\  \n
&-  \sum_{j = 0}^{n} \binom{n}{j} \beta_\tau W^{(n+1-j)}_{c_1} W^{(j)}_{c_2}  - 1_{n \ge 2} \sum_{j = 2}^{n} \binom{n}{j} \beta_\tau W^{(j)} W^{(n+1-j)}_{c_1 c_2} \\ \n
& -  \sum_{j = 0}^{n} \binom{n}{j} \p_{c_2}G_W^{(j)} W^{(n+1-j)}_{c_1} - 1_{n \ge 1} \sum_{j = 1}^{n} \binom{n}{j} G_W^{(j)} W_{c_1 c_2}^{(n-j+1)} \\ 
& - \sum_{j = 0}^n \binom{n}{j} \left(\dot{\tau}_{c_1 c_2}+2\dot{\tau}_{c_1} \dot{\tau}_{c_2} \right) \beta_\tau^2 W^{(j)} W^{(n+1-j)}  =: F_{W,n}^{c_1, c_2}\,.  \label{Bernie:1}
\end{align} 

\noindent We shall now compute the following identity by differentiating \eqref{dc.FW}
\begin{align} \n
\p_{c_1 c_2} F_W = &- \beta_\tau e^{- \frac 3 4 s} \Big( A_{c_1 c_2} ( \beta_3 Z + \beta_4 (e^{- \frac s 4}W + \kappa) ) + A_{c_1} ( \beta_3 Z_{c_2} + \beta_4 (e^{- \frac s4} W_{c_2} + \kappa_{c_2}) ) \Big) \\  \n
& - \beta_\tau e^{- \frac 3 4 s} \Big( A_{c_2} (\beta_3 Z_{c_1} + \beta_4 (e^{- \frac s 4} W_{c_1} + \kappa_{c_1})) + A (\beta_3 Z_{c_1 c_2} + \beta_4 ( e^{- \frac s 4}W_{c_1 c_2} + \kappa_{c_1 c_2})) \Big) \\ \label{find:1}
& + \dot{\tau}_{c_2} \beta_\tau \p_{c_1} F_W + \dot{\tau}_{c_1 c_2} \beta_\tau F_W + \dot{\tau}_{c_1} \beta_\tau \p_{c_2} F_W\,. 
\end{align}

\noindent Similarly, computing $\p_x^n$ of the above expression, we record for $n \ge 1$, 
\begin{align} \n
\p_{c_1 c_2} F_W^{(n)} = & - \beta_\tau e^{- \frac 3 4 s} \sum_{j = 0}^n \binom{n}{j} \Big( A_{c_1 c_2}^{(j)} (\beta_3 Z^{(n-j)} + \beta_4 (e^{- \frac s 4}W+ \kappa)^{(n-j)}) \\ \n
& \qquad \qquad \qquad \qquad \qquad + A_{c_1}^{(j)} (\beta_3 Z_{c_2}^{(n-j)} + \beta_4 (e^{- \frac s 4} W_{c_2} + \kappa_{c_2}) ^{(n-j)}) \Big) \\ \n
& - \beta_\tau e^{- \frac 3 4 s} \sum_{j = 0}^n \binom{n}{j} \Big( A_{c_2}^{(j)} (\beta_3 Z_{c_1}^{(n-j)} + \beta_4 (e^{- \frac s 4}W_{c_1} + \kappa_{c_1})^{(n-j)}) \\  \n
& \qquad \qquad \qquad \qquad \qquad + A^{(j)} (\beta_3 Z_{c_1 c_2}^{(n-j)} + \beta_4 (e^{- \frac s 4} W_{c_1 c_2} + \kappa_{c_1 c_2})^{(n-j)} ) \Big) \\ \label{Fwn:c1:c2:exp}
& + \dot{\tau}_{c_2} \beta_\tau \p_{c_1} F_W^{(n)} + \dot{\tau}_{c_1 c_2} \beta_\tau F_W^{(n)} + \dot{\tau}_{c_1} \beta_\tau \p_{c_2} F_W^{(n)}\,.
\end{align}
 
\subsubsection{$\nabla_{\alpha, \beta}^2$ derivatives of $Z$}

A calculation of $\p_{c_2}$ of equation \eqref{every:time:1} results in 
\begin{align} \n
\p_s Z_{c_1 c_2} + \mathcal{V}_Z \p_x Z_{c_1 c_2} = & \p_{c_1 c_2} F_Z - \sum_{i \in \{1, 2 \}} Z_{c_i}^{(1)} \Big( \dot{\tau}_{c_{i'}} \beta_\tau^2 \beta_2 W + \beta_\tau \beta_2 W_{c_{i'}} + \p_{c_{i'}} G_Z \Big) \\ \n
& - Z^{(1)} \Big( \dot{\tau}_{c_1 c_2} \beta_\tau^2 \beta_2 W + 2 \dot{\tau}_{c_1} \dot{\tau}_{c_2} \beta_\tau^2 \beta_2 W + \sum_{i \in \{1, 2\}}\dot{\tau}_{c_i} \beta_\tau^2 \beta_2 W_{c_{i'}} \\ 
& + \beta_\tau \beta_2 W_{c_1 c_2} + \p_{c_1 c_2} G_Z \Big) =: F_{Z,0}^{c_1, c_2}\,.\label{Lauv:1}
\end{align}

\noindent Computing $\p_x^n$ we obtain 
\begin{align} \n
&\Big( \p_s + \frac 5 4 n + n \beta_\tau \beta_2 W^{(1)} \Big) Z_{c_1 c_2}^{(n)} + \mathcal{V}_Z \p_x Z_{c_1 c_2}^{(n)} \\ \n
= & - 1_{n \ge 2} \sum_{j = 2}^n \binom{n}{j} \beta_\tau \beta_2 W^{(j)} Z_{c_1 c_2}^{(n-j+1)} - 1_{n \ge 1} \sum_{j = 1}^n \binom{n}{j} G_Z^{(j)} Z_{c_1 c_2}^{(n-j+1)} \\ \n
& - \sum_{j = 0}^n \sum_{i \in \{1, 2 \}} \binom{n}{j} Z_{c_i}^{(j+1)} \Big( \dot{\tau}_{c_{i'}} \beta_\tau^2 \beta_2 W^{(n-j)} + \beta_\tau \beta_2 W_{c_{i'}}^{(n-j)} + \p_{c_{i'}}G_Z^{(n-j)} \Big) \\ \n
& - \sum_{j = 0}^n \binom{n}{j} Z^{(j+1)} \Big( \dot{\tau}_{c_1 c_2} \beta_\tau^2 \beta_2 W^{(n-j)} + 2 \dot{\tau}_{c_1} \dot{\tau}_{c_2} \beta_\tau^2 \beta_2 W^{(n-j)} + \beta_\tau \beta_2 W_{c_1 c_2}^{(n-j)} \\ \n
& \qquad \qquad \qquad \qquad + \sum_{i \in \{1, 2\}} \dot{\tau}_{c_i} \beta_\tau^2 \beta_2 W_{c_{i'}}^{(n-j)} + \p_{c_1 c_2}G_Z^{(n-j)} \Big) \\  
& + \p_{c_1 c_2} F_Z^{(n)} =: F_{Z,n}^{c_1, c_2}\,. \label{Lauv:2}
\end{align}

We now record the expression for 
\begin{align} \n
\p_{c_2 c_1} F_Z = &- \beta_\tau e^{-s}\Big( A ( \beta_3 ( e^{- \frac s 4} W_{c_1 c_2} + \kappa_{c_1 c_2} ) + \beta_4 Z_{c_1 c_2} ) + A_{c_1 c_2} (\beta_3 (e^{- \frac s 4}W + \kappa) + \beta_4 Z) \Big) \\ \n
& - \beta_\tau e^{-s} \sum_{i \in \{1, 2 \}} A_{c_i} \Big( \beta_3 (e^{- \frac s 4} W_{c_{i'}} + \kappa_{c_{i'}}) + \beta_4 Z_{c_{i'}} \Big) + \dot{\tau}_{c_1} \beta_\tau \p_{c_2} F_Z  + \dot{\tau}_{c_2} \beta_\tau \p_{c_1} F_Z \\ \label{Lauv:3}
& + \dot{\tau}_{c_1 c_2} \beta_\tau F_Z\,. 
\end{align}

\noindent Next, we compute $\p_x^n$ of the above expression to obtain 
\begin{align} \n
\p_{c_2 c_1} F_Z^{(n)} = & - \beta_\tau e^{-s} \sum_{j = 0}^n \binom{n}{j} A^{(j)} \Big(  \beta_3 (e^{- \frac s 4} W_{c_1 c_2} + \kappa_{c_1 c_2})^{(n-j)} + \beta_4 Z_{c_1 c_2}^{(n-j)} \Big) \\ \n
& - \beta_\tau e^{-s} \sum_{j = 0}^n \binom{n}{j} A_{c_1 c_2}^{(j)} \Big( \beta_3 (e^{- \frac s 4} W + \kappa)^{(n-j)} + \beta_4 Z^{(n-j)} \Big) \\ \n
& - \beta_\tau e^{-s} \sum_{j = 0}^n \sum_{i \in \{1, 2\}} \binom{n}{j} A_{c_i}^{(j)} \Big( \beta_3( e^{- \frac s 4} W_{c_{i'}} + \kappa_{c_{i'}})^{(n-j)} + \beta_4 Z_{c_{i'}}^{(n-j)}  \Big) \\ \label{Lauv:4}
&+ \dot{\tau}_{c_1} \beta_\tau \p_{c_2} F_Z^{(n)}  + \dot{\tau}_{c_2} \beta_\tau \p_{c_1} F_Z^{(n)}  + \dot{\tau}_{c_1 c_2} \beta_\tau F_Z^{(n)}\,.
\end{align}

\subsubsection{$\nabla_{\alpha, \beta}^2$ derivatives of $A$}

We compute $\p_{c_2}$ of equation \eqref{socialite:1} to obtain the equation to obtain 
\begin{align} \n
\p_s A_{c_1 c_2} + \mathcal{V}_A \p_x A_{c_1 c_2} = & \p_{c_1 c_2} F_A - \sum_{i = \{1, 2\}} A^{(1)}_{c_{i'}} \Big( \dot{\tau}_{c_i} \beta_\tau^2 \beta_1 W + \beta_\tau \beta_1 W_{c_i} + \p_{c_i} G_A \Big) \\ \n
& - A^{(1)} \Big( \dot{\tau}_{c_1 c_2} \beta_\tau^2 \beta_1 W + 2 \dot{\tau}_{c_1} \dot{\tau}_{c_2} \beta_1 \beta_\tau^3 W + \beta_\tau \beta_1 W_{c_1 c_2} + \p_{c_1 c_2}G_A \\ \n
& \qquad \qquad + \sum_{i = \{1, 2\}} \beta_\tau^2 \beta_1 \dot{\tau}_{c_i} W_{c_{i'}} \Big)\\ \label{Lauv:5}
 =: & F_{A,0}^{c_1, c_2}\,. 
\end{align}
By computing $\p_x^n$ of the above equation, we obtain 
\begin{align} \n
&\Big( \p_s + \frac 5 4 n + n \beta_\tau \beta_1 W^{(1)} \Big) A_{c_1 c_2}^{(n)} + \mathcal{V}_A \p_x A_{c_1 c_2}^{(n)} \\ \n
&\quad=  - 1_{n \ge 2} \sum_{j = 2}^n \binom{n}{j} \beta_\tau \beta_1 W^{(j)} A_{c_1 c_2}^{(n-j+1)} - 1_{n \ge 1} \sum_{j = 1}^n \binom{n}{j} G_A^{(j)} A_{c_1 c_2}^{(n-j+1)} \\ \n
& \qquad- \sum_{i = \{1, 2\}} \sum_{j = 0}^n \binom{n}{j} A^{(j+1)}_{c_{i'}} \Big( \dot{\tau}_{c_i} \beta_\tau^2 \beta_1 W^{(n-j)} + \beta_\tau \beta_1 W_{c_{i}}^{(n-j)} + \p_{c_i} G_A^{(n-j)} \Big) \\ \n
&\qquad - \sum_{i = \{1, 2 \}} \sum_{j = 0}^n \binom{n}{j} A^{(j+1)} \Big( \dot{\tau}_{c_1 c_2} \beta_\tau^2 \beta_1 W^{(n-j)} + 2 \dot{\tau}_{c_1} \dot{\tau}_{c_2} \beta_1 \beta_\tau^3 W^{(n-j)} + \beta_\tau \beta_1 W_{c_1 c_2}^{(n-j)} \\ \label{Lauv:6}
& \qquad \qquad \qquad \qquad + \sum_{i \in \{1, 2 \}} \beta_\tau^2 \beta_1 \dot{\tau}_{c_i} W_{c_{i'}}^{(n-j)} + \p_{c_1 c_2} G_A^{(n-j)}\Big) + \p_{c_1 c_2} F_A^{(n)} =: F_{A,n}^{(c_1, c_2)}\,. 
\end{align}

We next differentiate equation \eqref{pcFa} to obtain
\begin{align} \n
\p_{c_1 c_2} F_A = & \beta_\tau \beta_1 e^{-s} \Big( e^{- \frac s 4} W_{c_2} + \kappa_{c_2} + Z_{c_2} \Big) \Big( e^{- \frac s 4} W_{c_1} + \kappa_{c_1} + Z_{c_1} \Big) \\ \n
& + \beta_\tau \beta_1 e^{-s} \Big( e^{- \frac s 4} W + \kappa + Z \Big) \Big( e^{- \frac s 4} W_{c_1 c_2} + \kappa_{c_1 c_2} + Z_{c_1 c_2} \Big) \\ \n
& - 2 \beta_\tau \beta_5 e^{-s} \Big( e^{- \frac s 4} W_{c_2} + \kappa_{c_2} - Z_{c_2} \Big) \Big( e^{- \frac  s4} W_{c_1} + \kappa_{c_1} - Z_{c_1} \Big) \\ \n
& - 2 \beta_\tau \beta_5 e^{-s} \Big( e^{- \frac s 4} W + \kappa - Z \Big) \Big( e^{- \frac s 4} W_{c_1 c_2} + \kappa_{c_1 c_2} - Z_{c_1 c_2} \Big) \\ \label{Lauv:7}
&+  \dot{\tau}_{c_1 c_2} \beta_\tau F_A + \dot{\tau}_{c_2} \beta_\tau \p_{c_1}F_A + \dot{\tau}_{c_1} \beta_\tau \p_{c_2}F_A\,.
\end{align}

\noindent By computing $\p_x^n$ of the above, we obtain 
\begin{align} \n
\p_{c_1 c_2} F_A^{(n)} = & \beta_\tau \beta_1 e^{-s} \sum_{j = 0}^n \binom{n}{j} \Big( e^{- \frac s 4} W_{c_2} + \kappa_{c_2} + Z_{c_2} \Big)^{(j)} \Big( e^{- \frac s 4} W_{c_1} + \kappa_{c_1} + Z_{c_1} \Big)^{(n-j)} \\ \n
+& \beta_\tau \beta_1 e^{-s} \sum_{j = 0}^n \binom{n}{j} \Big( e^{- \frac s 4} W + \kappa + Z \Big)^{(j)} \Big( e^{- \frac s 4} W_{c_1 c_2} + \kappa_{c_1 c_2} + Z_{c_1 c_2} \Big)^{(n-j)} \\ \n
-& 2 \beta_\tau \beta_5 e^{-s} \sum_{j = 0}^n \binom{n}{j} \Big( e^{- \frac s 4} W_{c_2} + \kappa_{c_2} - Z_{c_2} \Big)^{(j)} \Big( e^{- \frac s 4} W_{c_1} + \kappa_{c_1} - Z_{c_1} \Big)^{(n-j)} \\ \n
-& 2 \beta_\tau \beta_5 e^{-s} \sum_{j = 0}^n \binom{n}{j} \Big( e^{- \frac s 4} W + \kappa - Z \Big)^{(j)} \Big( e^{- \frac s 4} W_{c_1 c_2} + \kappa_{c_1 c_2} - Z_{c_1 c_2} \Big)^{(n-j)}  \\ \label{Lauv:8}
 + &  \dot{\tau}_{c_1 c_2} \beta_\tau F_A^{(n)} + \dot{\tau}_{c_2} \beta_\tau \p_{c_1}F_A^{(n)} + \dot{\tau}_{c_1} \beta_\tau \p_{c_2} F_A^{(n)}\,. 
\end{align}

\section{Initial data}

We assume the data is of the form
\begin{equation} \label{guitar:or:band}
W_0=\bar{W} \chi(\eps^{\frac 1 4} x) + \hat W_0+\alpha x^2 \chi(x)+\beta x^3 \chi(x)\,,
\end{equation}
where $\chi$ is a smooth cut-off function satisfying $\chi(x)=1$ for $\abs{x}\leq 1$ and with support contained in the ball of radius $2$.

On the perturbation $\hat{W}_0$, we shall assume 
\begin{align} \label{assume:1}
\abs{\eta_{\frac 1 5} \hat{W}_0^{(n)}(x)} &\le \eps \,,&&\mbox{ for } \abs{x}\leq \eps^{-\frac14} \mbox{ and } n = 0,...,8\,, \\
|\hat{W}_0^{(n)}(0)| &\le \eps\,, &&\text{ for } n = 2, 3\,, \\
\hat{W}_0^{(n)}(0) &= 0 \,,&&\text{ for } n = 0,1, 4, 5\,. 
\end{align}

For $Z_0(x) := Z(s_0, x)$, and $A_0(x) = A(s_0, x)$, we assume 
\begin{align}
\| Z_0^{(n)} \|_\infty &\le \eps^{\frac 3 2}\,,\\
\| A_0^{(n)} \|_\infty &\le \eps^{\frac 3 2} \,.
\end{align}
for $n=0,\dots, 8$.

Furthermore, we will assume the following support assumption on the initial data $(W_0,Z_0,A_0)$
\begin{equation}\label{blue:grass}
\supp({W}_0)\cup \supp({Z}_0)\cup \supp({A}_0) \subset [- \frac M 2 \eps^{- \frac 1 4}, \frac M 2 \eps^{-\frac 1 4}] \,.
\end{equation}

We will now describe the iteration. 
\begin{definition} The quantities $W_{\alpha, \beta}, Z_{\alpha, \beta}, A_{\alpha, \beta}$ solve the system \eqref{W:0} - \eqref{A:0} with initial data $W_{0}$ given by \eqref{guitar:or:band} for $W_{\alpha, \beta}$.  
\end{definition}

We now describe the inductive hypotheses. First, we define the time step via 
\begin{align} \label{time:step}
s_N := - \log(\eps) + N, \qquad N  \in \mathbb{N}\,.  
\end{align}

\noindent The inductive hypotheses we make are the following: 
\begin{align} \label{induct:1}
W^{(2)}_{\alpha_N, \beta_N}(s_N) = 0, \qquad  W_{\alpha_N, \beta_N}^{(3)}(s_N) = 0\,,  
\end{align}
To initialize the induction, we take 
\begin{align} \label{zero:param}
\alpha_0 = - \frac 1 2 \hat{W}_0^{(2)}(0), \qquad \beta_0 = - \frac 1 6 \hat{W}_0^{(3)}(0)\,. 
\end{align}
Note that \eqref{induct:1} is satisfied for $N = 0$, which is the first step of the iteration, according to \eqref{zero:param}, due to \eqref{guitar:or:band} which implies that 
\begin{align*}
W^{(2)}_{0, 0}(0, s_0) &= \bar{W}^{(2)}(0) + \hat{W}_0^{(2)}(0) -  \hat{W}_0^{(2)}(0)  = 0 \,,\\
W^{(3)}_{0,0}(0, s_0) &=  \bar{W}^{(3)}(0) + \hat{W}_0^{(3)}(0) - \hat{W}_0^{(3)}(0) = 0\,. 
\end{align*}

\section{Bootstrap assumptions} \label{section:Bootstraps}

In this section we delineate all of our bootstrap assumptions. First, recall the weight function $\eta_\gamma$ defined in \eqref{weight:eta}. Let us also specify the hierarchy of three small parameters, where $\eps$ is significantly smaller than any power of $M^{-1}$, and in turn $M^{-1}$ is significantly smaller than any power of $\ell$. For the sake of precision, we make the following selections 
\begin{align} \label{choice:M}
\ell^{-1} = \log \log(M)\,. 
\end{align}

\subsection{Parameter assumptions}

We will first specify bootstrap assumptions on the parameters, $(\alpha, \beta)$, appearing in the specification of the initial data in \eqref{guitar:or:band}. Throughout the analysis, our parameters $(\alpha, \beta)$ will be contained in the rectangle set $\mathcal{B}_N$, which is defined via   
\begin{align} \label{def:BN}
\mathcal{B}_N  =  \left\{(\alpha, \beta) \in \mathbb{R}^2 : |\alpha -\alpha_{N} | \leq M^{30} \eps^{-\frac34}e^{-\frac74s_{N}}  + \eps^{-\frac 3 {10}} e^{- \frac32s_{N}} , 
\abs{\beta - \beta_{N}} \leq M^{30} \eps^{-\frac12}e^{- \frac32 s_{N}}  \right\}\,.
\end{align}
In particular, since $s_0=-\log\eps$ we have
\begin{align} \label{apple:1}
|\alpha | \leq 2 M^{30}\eps , \qquad \abs{\beta} \leq 2 M^{30}\eps\,. 
\end{align}

\noindent Note that the bootstrap in this parameter region will be verified in \eqref{crickets:2} - \eqref{crickets:3}. Moreover, notice that due to \eqref{est:hatW:in}, \eqref{apple:1} is valid for the initial choice of $(\alpha, \beta) = (\alpha_0, \beta_0)$, defined in \eqref{zero:param}.

\begin{remark}[Notation] We will now drop the subscript $W_{\alpha, \beta}$ as it is understood that $\alpha, \beta$ are fixed, and arbitrary elements of the set $\mathcal{B}_N(\alpha_N, \beta_N)$. 
\end{remark}
Note that we only assume (and therefore prove) the below bootstraps on the time interval $-\log \eps \le s \le s_{N+1}$. We now state the main inductive proposition we will be proving using these bootstrap estimates. The proof of this proposition will take place in Subsection \ref{subsection:proof}.

\begin{proposition} \label{induct:prop} Fix $N \in \mathbb{N}$, the parameters $(\eps, M, \ell)$ through \eqref{choice:M}. Let $s_N$ be given by \eqref{time:step}. Assume $(\alpha_N, \beta_N)$ are given so that \eqref{induct:1} is valid for choice of data \eqref{guitar:or:band}, satisfying conditions \eqref{assume:1} - \eqref{blue:grass}. Then there exists $(\alpha_{N+1}, \beta_{N+1})$ so that \eqref{induct:1} is valid for $s_{N+1}$ for data given again by \eqref{guitar:or:band}.  
\end{proposition}

\subsection{Bootstrap estimates on $(W^{(n)},Z^{(n)},A^{(n)})$ and modulation variables} \label{subsection:base}

We will assume the following bootstraps on the support of the solutions:
\begin{equation}\label{e:support}
\supp W(s) \cup \supp Z(s) \cup \supp A(s) \subset B(M \eps e^{\frac54 s} ) =: B_f\,,
\end{equation}
where $B(r)$ is the ball centered at the origin of radius $r$. We give the name $B_f$ to the above ball to compactify notation, as we will frequently write indicator functions on this ball. 

We will assume the following global in $x$ bootstrap assumptions on $W$:
\begin{align}
\abs{W}&\leq  \ell \log M \eta_{\frac1{20}}  \label{W:boot:0}\,,\\
|W^{(1)}| &\le \ell \log M \eta_{- \frac 15} \label{e:uniform:W1}\,,
 \\ \label{weds:1}
|W^{(n)}|&\leq M^{n^2} \eta_{- \frac 1 5} \quad\text{ for } n = 2,\dots,8 \,,
\end{align}
As a consequence of $\eqref{W:boot:0}$ and \eqref{e:support}, we have that 
\begin{align*}
|W| \le \ell \log(M) \eta_{\frac{1}{20}} \lesssim \ell \log(M) \langle x \rangle^{\frac 1 5} \lesssim \ell \log(M) \langle M \eps e^{\frac 5 4 s} \rangle^{\frac 1 5} \lesssim \ell \log(M) (1 + M^{\frac 1 5} \eps^{\frac 15} e^{\frac s 4})\,, 
\end{align*}
and thus, 
\begin{align} \label{est:W:good}
e^{- \frac s 4} |W| \le 1\,,
\end{align}
which we shall use repeatedly. 

On $Z$ and $A$ we will assume the following bootstraps:
\begin{align} \label{Z:boot:0}
 \| Z \|_\infty &\le \eps^{\frac 5 4}\,, &\| Z^{(n)} \|_\infty &\le M^{2n^2} e^{- \frac 5 4 s} \,, \\ \label{A:boot:9}
 \| A \|_\infty &\le M\eps\,, &\| A^{(n)} \|_\infty &\le M^{2n^2} e^{- \frac 5 4 s}\,, 
\end{align}
for $n=1,\dots 8$.

\noindent For the difference, $\tilde{W}$, we make the following bootstrap assumptions on $\tilde{W}$ and $\tilde{W}^{(1)}$ in the region $\abs{x}\leq \eps^{-\frac14} $
\begin{align} \label{thurs:1}
 |\tilde{W}| &\le \eps^{\frac{3}{20}} \eta_{\frac{1}{20}}\,,\\
 |\tilde{W}^{(1)}| &\le \eps^{\frac{1}{20}} \eta_{- \frac 1 5}\,.  \label{e:Wtilde:1:bootstrap} 
 \end{align}
 For the higher order derivatives of $\tilde{W}$, we will assume the following local in $x$ bootstraps in the region $|x| \le \ell$
 \begin{align}
  |\tilde{W}^{(n)}| &\le \abs{x}^{6-n}\eps^{\frac{1}{5}}+\eps^{\frac12}\leq 2\abs{\ell}^{6-n}\eps^{\frac{1}{5}},\quad \mbox{for } 0\leq k \leq 5  \label{e:Wtilde:bootstrap} \\
   |\tilde{W}^{(6)}| &\le \eps^{\frac{1}{5}},\label{e:W6:bootstrap}\\ \label{gerrard:1}
  |\tilde{W}^{(7)}| & \le M \eps^{\frac 1 5} \,, \\ \label{gerrard:2}
  |\tilde{W}^{(8)}| &\le M^3 \eps^{\frac 1 5} 
\end{align}

We now make the following crucial bootstrap assumptions, which display decay in $s$ for the unconstrained quantities $q^{(2)}, q^{(3)}$ (recall the notation defined in \eqref{q:def:q}),  
\begin{align} \label{boot:decay}
|q^{(2)}| \le  \eps^{\frac{1}{10}} e^{- \frac 3 4  s} , \qquad |q^{(3)}| \le M^{40} e^{-  s}\,,  
\end{align}
and the following smallness estimate
\begin{align} \label{boot:W:5:0}
|\tilde{W}^{(5)}(0,s)| \le \eps^{\frac12} \text{ for } -\log \eps \le s \le s_{n+1}\,,  
\end{align}
which in particular, when coupled with \eqref{W5:non}, ensures that 
\begin{align} \label{moon}
|q^{(5)}| \ge 120 - \eps^{\frac 1 2} \ge 100\,. 
\end{align}

We also have crucially the following estimate 
\begin{equation}
\abs{W^{(1)}}\leq 1+ e^{-\frac 3 4 s}\label{eq:W1:bnd:1}\,.
\end{equation}

Finally, we have the bootstraps on the modulation variables:
\begin{align}\label{e:GW0}
&|\mu| \le \eps^{\frac 1 6} e^{-\frac 3 4 s}\,, && |\dot{\tau}| \le \eps^{\frac 1 6} e^{- \frac 3 4 s}\,, && |\dot{\kappa}| \le \eps^{\frac 1 8}\,, \\ \label{mod:sub}
&|\kappa- \kappa_0| \le \eps \,, && |\dot{\xi}| \le 3 \kappa_0\,.
\end{align}
As a consequence we have
\begin{equation}
\abs{1-\beta_{\tau}}\leq 2 \eps^{\frac 1 6} e^{- \frac 3 4 s}\,,\label{e:1beta:bnd}
\end{equation}
which will be employed repeatedly in the forthcoming estimates.

\subsection{$\nabla_{\alpha, \beta}$ bootstraps}

We now provide the bootstrap assumptions we make on the $(\alpha, \beta)$ derivatives of the quantities appearing in Subsection \ref{subsection:base}. The first bootstraps we provide are for the modulation variables, for which we notably do not distinguish between $\alpha$ and $\beta$ derivative (recall $\p_c \in \{\p_\alpha, \p_\beta \}$ from \eqref{c:deriv}). 
\begin{align} \label{mako:1}
&|\p_c \mu| \le M^{33} \eps^{\frac 1 2} e^{- \frac s 4}\,, && |\p_c \dot{\tau}| \le \eps^{\frac 1 2}\,, && |\p_c \dot{\kappa}| \le \eps^{\frac 1 4} e^{\frac 1 2 s}\,, \\ \label{Baba:O:1}
&|\p_c \kappa| \le \eps^{\frac 1 2}\,, && |\p_c \dot{\xi}| \le M \eps^{\frac 1 2}  \,.
\end{align}

\noindent Next, we provide the bootstrap assumptions on $\p_\alpha Z, \p_\beta Z, \p_\alpha A, \p_\beta A$, and higher derivatives thereof. We again note that we do not distinguish between $\alpha$ and $\beta$ derivatives for these quantities. 
\begin{align} \label{mako:2}
\| \p_c Z \|_\infty &\le \eps^{\frac 1 2} \,, & \| \p_c A \|_\infty &\le \eps^{\frac 1 2}\,, \\ \label{mako:3}
\| \p_c Z^{(n)} \|_\infty &\le M^{2n^2} \eps^{\frac 1 2} e^{-\frac12 s}\,, & \|\p_c A^{(n)}\|_\infty &\le M^{2n^2}  \eps^{\frac 1 2} e^{-\frac12 s}\,,
\end{align}
for $n=1,\dots,7$.

\noindent Next, we provide the bootstrap assumptions for the elements of the $2 \times 2$ $s$-dependent matrix 
\begin{align*}
\begin{pmatrix} \p_\alpha q^{(2)}(s) && \p_\beta q^{(2)}(s) \\ \p_\alpha q^{(3)}(s) && \p_\beta q^{(3)}(s) \end{pmatrix}\,.
\end{align*}
For these quantities, we need to distinguish between $\alpha$ and $\beta$ derivatives carefully, which we do via 
\begin{align} \label{mako:4}
&\frac 1 2 \eps^{\frac34}e^{\frac 3 4 s} \le \p_\alpha q^{(2)} \le 4 \eps^{\frac34}e^{\frac 3 4 s }\,,  &&|\p_\alpha q^{(3)}| \le \eps e^{\frac s 2 }\,, \\ \label{mako:5}
&|\p_\beta q^{(2)}| \le \eps e^{\frac 3 4 s}\,,  &&\frac 1 2 \eps^{\frac12} e^{\frac 1 2 s} \le \p_\beta q^{(3)} \le 4 \eps^{\frac12}e^{\frac 1 2 s}\,.
\end{align}

In addition, we will need the enhanced constrained bootstrap 
\begin{align} \label{hotel:motel}
|\tilde{q}^{(5)}_c(s)| \le \eps^{\frac 3 8} e^{\frac 1 8 s}\,. 
\end{align}

Next, we will assume the following bootstrap bounds on $\p_c W$ and higher derivatives thereof. 
\begin{align} \label{pc:W0}
\|\p_c W  \|_\infty &\le M^{4} \eps^{\frac34} e^{\frac 3 4 s}\,, \\ \label{mako:6}
\|\p_c W^{(n)} \eta_{\frac{1}{20}} \|_\infty &\le M^{(n+2)^2}  \eps^{\frac34}e^{\frac 3 4 s} \,.
\end{align} 
for $n=1,\dots,7$.
Finally, we assume the following localized bounds on the region  $|x| \le \ell$ which are stronger than \eqref{pc:W0} - \eqref{mako:6}
\begin{align} \label{warrior:1}
|W_c^{(n)}| &\le \ell^{\frac 1 2} M  \eps^{\frac34}e^{\frac 3 4 s} \qquad\text{ for }0 \le n \le 6 \,,\\  \label{warrior:2}
|W_c^{(7)}| &\le M  \eps^{\frac34}e^{\frac 3 4 s}\,.
\end{align}


\subsection{$\nabla_{\alpha, \beta}^2$ bootstraps}

We now provide the bootstrap assumptions on two parameter ($\alpha, \beta$) derivatives of the quantities in Subsection \ref{subsection:base}. For these highest order bootstraps, we do not need to distinguish between $\alpha$ and $\beta$ derivatives. Recall that $\p_{c_1c_2}$ means $c_i \in \{\alpha, \beta\}$.  We impose the following bootstrap assumptions for $0\le n\le 6$
\begin{align} \label{posty:1}
\| \p_{c_1 c_2} Z^{(n)} \|_\infty& \le M^{2j^2} \eps^{\frac 5 8} e^{\frac s 4}  \,, \\ \label{posty:2}
\| \p_{c_1 c_2} A^{(n)} \|_\infty &\le M^{2j^2} \eps^{\frac 5 8} e^{\frac s 4}   \,, \\ \label{buddy:1}
\| \p_{c_1 c_2} W^{(n)} \|_\infty &\le M^{(k+5)^2} \eps^{\frac32 }e^{\frac 3 2 s}  \,.
\end{align}

%

\noindent We will also need bootstraps on the second derivative of the modulation variables
\begin{align} \label{group:1}
&|\mu_{c_1 c_2}| \le M \eps^{\frac54}e^{\frac 5 4 s}\,, && | \dot{\kappa}_{c_1 c_2}| \le M^2 \eps^{\frac54}e^{2s}\,, && |\dot{\tau}_{c_1 c_2}| \le  \eps e^{\frac 3 4 s}\,, \\ \label{group:2}
&|\kappa_{c_1 c_2}| \le M^3 \eps^{\frac 5 4} e^s\,, && |\dot{\xi}_{c_1 c_2}| \le M^4 \eps^{\frac 5 4} e^s\,.  
\end{align}

\section{Preliminary estimates}

In order to analyze the equations \eqref{basic:w} - \eqref{basic:a} and their higher order spatial derivative counterparts, \eqref{W:n} - \eqref{A:n}, as well as their higher order parameter derivative counterparts, we first provide estimates on the forcing terms appearing in \eqref{basic:w} - \eqref{basic:a}. These are performed in Subsection \ref{subsection:Forcing}. Controlling these forcing terms requires in turn controlling the transport speeds, $G_W, G_Z, G_A$, which is achieved in Subsection \ref{GW:control}. The final subsection in this section, Subsection \ref{subsect:trajectory}, collects estimates on the trajectories associated to the transport structure of equations \eqref{W:n} - \eqref{A:n}.

\subsection{Transport speed estimates} \label{GW:control}

We now provide estimates on the transport speeds, which are defined in \eqref{gw:def} - \eqref{ga:def}. We begin with the following estimates. 
\begin{lemma}\label{l:GW} Let $-1 \leq r \leq 0$, and $n \ge 1$. Then the following estimates are valid on the transport speeds, \eqref{gw:def} - \eqref{ga:def}. 
\begin{align} \label{GW:transport:est}
\| G_W \eta_{\frac r 4} \|_{\infty}& \lesssim  \eps^{\frac 1 6} e^{- \frac 3 4 s} + M^{3+r}\eps^{(1+r)} e^{\frac{1+5r}{4}s},  && \| G_W^{(n)} \|_\infty \les M^{2n^2} e^{-s}\,, \\ \label{Z:transport:1}
\| G_Z +(1-\beta_2)e^{\frac s4}\kappa_0\|_\infty &\lesssim e^{\frac s 4}\,, && \| G_Z^{(n)} \|_\infty \lesssim M^{2n^2} e^{-s}\,, && \\
\| G_A  +(1-\beta_1)e^{\frac s4}\kappa_0 \|_\infty &\lesssim e^{\frac s 4}\,, && \| G_A^{(n)} \|_\infty \lesssim M^{2n^2} e^{-s}\label{A:transport:1}\,.
\end{align}
\end{lemma}
\begin{proof} We record the following identity: 
\begin{align}\label{def:GW:e}
G_W(x, s) = \mu(s) + G_{W,e}(x, s), \qquad G_{W,e}(x, s) := \beta_\tau \beta_2 e^{\frac s 4} \int_0^x Z^{(1)}(x', s) \ud x'\,,  
\end{align}
where we have invoked definition \eqref{gw:def} for $G_W$ and subsequently \eqref{def:mu} for the quantity $\mu(s)$. 
We estimate for $j \ge 1$, 
\begin{align} \label{Gw:j:est}
\| G_W^{(j)} \|_\infty = \| \beta_\tau \beta_2 e^{\frac 1 4 s} Z^{(j)} \|_\infty \le 2 e^{\frac 1 4 s} M^{2j^2} e^{- \frac 5 4 s}\,. 
\end{align}

\noindent Using \eqref{def:GW:e}, we estimate 
\begin{align} \n
\| G_W\eta_{\frac r 4} \|_\infty \lesssim &|\mu|  + \| G_{W,e} \eta_{\frac r 4} \|_\infty \\ \n
\lesssim &  \eps^{\frac 1 6} e^{- \frac 3 4 s}+ \| \langle x \rangle^r \int_0^{x} \p_x G_W(x') \ud x' \|_\infty \\ \n
\lesssim &  \eps^{\frac 1 6} e^{- \frac 3 4 s}+ \| \langle x \rangle^r \int_0^{x} \langle x' \rangle^{-1-r} \p_x G_W(x') \langle x' \rangle^{1+ r} \ud x' \|_\infty \\ \n
\lesssim & \eps^{\frac 1 6} e^{- \frac 3 4 s} + \| \p_x G_W \langle x \rangle^{1+r} \|_\infty \\\n\lesssim&  \eps^{\frac 1 6} e^{- \frac 3 4 s} + e^{\frac 1 4 s} \| Z^{(1)}  \langle x \rangle^{1+r} \|_\infty \\ \label{hardy:1}
\lesssim &  \eps^{\frac 1 6} e^{- \frac 3 4 s} + M^{3+r}\eps^{(1+r)} e^{\frac{1+5r}{4}s}\,.
\end{align}
\noindent Above, we have invoked estimate \eqref{e:GW0} for the estimate on $\mu$, the definition \eqref{gw:def} to calculate $\p_x G_W$, estimate \eqref{Z:boot:0} on $Z^{(1)}$, and the estimate \eqref{e:support} to translate spatial weights to growth in $s$.

The above calculation, \eqref{hardy:1}, works when $r < 0$, but at $r = 0$ does not quite work due to having to integrate $\langle x \rangle^{-1}$. However, in that case, we may estimate via 
\begin{align*}
\| G_W \|_\infty \lesssim |\mu| + \| G_{W,e} \|_\infty \lesssim \eps^{\frac 1 6} e^{- \frac 3 4 s} + \| \langle x \rangle G_W^{(1)} \|_\infty \lesssim \eps^{\frac 1 6} e^{- \frac 3 4 s} + M^2 e^{-s} (M \eps e^{\frac 5 4 s})\,, 
\end{align*}
where we have invoked \eqref{e:GW0} for the estimate on $\mu$, \eqref{Gw:j:est} with $j = 1$, and the estimate \eqref{e:support} on the support.  

We now move to the transport speed $G_Z$. First, for the lowest order quantity, we use the definition \eqref{gz:def} and the bootstrap assumptions \eqref{mod:sub} to estimate 
\begin{align*}
\| G_Z + (1-\beta_2)e^{\frac s4}\kappa_0 \|_\infty  \lesssim e^{\frac s 4} (1 + \eps + \eps^{\frac 5 4}) \lesssim \eps^{\frac s 4}\,. 
\end{align*}

 According to the definition \eqref{gz:def}, we estimate 
\begin{align*}
\| G_Z^{(n)} \|_\infty \lesssim e^{\frac 1 4 s} \| Z^{(n)} \|_\infty \lesssim M^{2n^2} e^{- s}\,, 
\end{align*}

\noindent where we have invoked the bootstrap, \eqref{Z:boot:0}. For the transport speed $G_A$, we invoke the definition \eqref{ga:def} to perform the exact same calculation.

\end{proof}

\begin{lemma} Let $c \in \{ \alpha, \beta \}$. For $0 < r \le 1$ and $1\leq n\leq 7$, the following estimates are valid on the transport speeds, \eqref{gw:def} - \eqref{ga:def} 
\begin{align} \label{r:est:1}
\| \p_c G_W \eta_{- \frac r 4} \|_\infty &\lesssim \eps^{\frac 1 2} +  M^{3-r} \eps^{\frac 3 2 - r} e^{\frac{4 - 5r}{4} s}\,, && \| \p_c G_W^{(n)} \|_\infty \lesssim M^{2n^2} \eps^{\frac 1 2} e^{- \frac s 4}\,, \\ \label{r:est:2}
\| \p_c G_Z \|_\infty &\lesssim \eps^{\frac 1 4} e^{\frac s 4}\,, && \| \p_c G_Z^{(n)} \|_\infty \lesssim M^{2n^2} \eps^{\frac 1 2} e^{- \frac s 4}\,,\\ \label{r:est:3}
\| \p_c G_A \|_\infty &\lesssim \eps^{\frac 1 4} e^{\frac s 4}\,,  && \| \p_c G_A^{(n)} \|_\infty \lesssim M^{2n^2} \eps^{\frac 1 2} e^{- \frac s 4} \,.
\end{align}
\end{lemma}
\begin{proof} 

We differentiate \eqref{def:GW:e} in $c$ to yield 
\begin{align} \label{coolie}
\p_c G_W = \p_c \mu + \p_c G_{W,e} = \p_c \mu + \p_c \dot{\tau} \beta_\tau^2 \beta_2 e^{\frac s 4} \int_0^x Z^{(1)}(x', s) \ud x' + \beta_\tau \beta_2 e^{\frac s 4} \int_0^x \p_c Z^{(1)} \ud x'\,.  
\end{align}

\noindent Multiplying now by a weight of $\eta_{- \frac r 4}$, we obtain for every $r > 0$, 
\begin{align} \n
\| \p_c G_W \eta_{- \frac r 4} \|_\infty \lesssim & |\p_c \mu| + |\p_c \dot{\tau}| e^{\frac s 4} \| Z^{(1)} \eta_{\frac{1-r}{4}} \|_\infty + e^{\frac s 4} \| \p_c Z^{(1)} \eta_{\frac{1-r}{4}} \|_\infty \\
\lesssim & M^{33} \eps^{\frac 1 2}e^{-\frac s4} + \eps^{\frac 1 2} e^{\frac s 4} (M^2 e^{- \frac 5 4 s}) (M\eps e^{\frac 5 4 s})^{1-r} + e^{\frac s 4} (M^2 \eps^{\frac 1 2} e^{- \frac s 2}) (M\eps e^{\frac 5 4 s} )^{1 - r}\n\\
\lesssim& \eps^{\frac 1 2} + M^{3-r} \eps^{\frac 3 2 - r} e^{\frac{4 - 5r}{4} s}\n\,,
\end{align}
where we have invoked \eqref{mako:1} for the modulation variables, \eqref{Z:boot:0} and \eqref{mako:3} for the $Z$ quantities, and \eqref{e:support} to estimate $\eta_{\frac{1-r}{4}}$ in the support of $Z^{(1)}$ and hence $\p_c Z^{(1)}$.   

We first differentiate $G_W$ to order $n\geq 1$  in $x$ via \eqref{gw:def} and then take $\p_c$ of the result to produce 
\begin{align*}
\p_c G_W^{(n)} = \p_c \dot{\tau} \beta_\tau^2 \beta_2 e^{\frac s 4} Z^{(n)} + \beta_2 \beta_\tau e^{\frac s 4} \p_c Z^{(n)}\,,
\end{align*} 

\noindent which upon estimating yields 
\begin{align} \n
\| \p_c G_W^{(n)} \|_\infty \lesssim &|\p_c \dot{\tau}| e^{\frac s 4} \| Z^{(n)} \|_\infty + e^{\frac s 4} \| \p_c Z^{(n)} \|_\infty  \\ 
\lesssim & M^{2n^2} \eps^{\frac 1 2}  e^{- s} + e^{\frac s 4} M^{2n^2} \eps^{\frac 12} e^{- \frac s2} \lesssim M^{2n^2} \eps^{\frac 1 2} e^{- \frac s 4}\,,\n
\end{align}
where we have invoked \eqref{mako:1} for the modulation variables, \eqref{Z:boot:0} and \eqref{mako:3} for the $Z$ quantities.

By differentiating \eqref{gz:def} in $\p_c$, we obtain the identities 
\begin{align} \label{Sam:smith:1}
\p_c G_Z = & \dot{\tau}_c \beta_\tau G_Z + \beta_\tau e^{\frac s 4} (\beta_2 \kappa_c - \dot{\xi}_c + Z_c)  \\ \n
 = & \p_c \dot{\tau} \beta_\tau^2 e^{\frac s 4} (\beta_2 \kappa - \dot{\xi} + Z) + \beta_\tau e^{\frac s 4} (\beta_2 \p_c \kappa - \p_c \dot{\xi} + \p_c Z) \\ \label{Sam:smith:2}
\p_c G_Z^{(n)} = & \dot{\tau}_c \beta_\tau G_Z^{(n)} + \beta_\tau e^{\frac s 4} Z_c^{(n)}\,. 
\end{align}

\noindent By estimating \eqref{Sam:smith:1} we obtain 
\begin{align} \n
\| \p_c G_Z \|_\infty \lesssim & |\p_c \dot{\tau}| e^{\frac s 4} (|\kappa| + |\dot{\xi}| + \| Z \|_\infty) + e^{\frac s 4} (|\p_c \kappa| + |\p_c \dot{\xi}| + \| \p_c Z \|_\infty) \\ \n
\lesssim & \eps^{\frac 1 2} e^{\frac s 4} (1 + \eps^{\frac 5 4}) + e^{\frac s 4} (\eps^{\frac 1 4} + \eps^{\frac 1 4} + \eps^{\frac 1 2}) \lesssim \eps^{\frac 1 4} e^{\frac s 4},  
\end{align}
where we have invoked both \eqref{mako:1} - \eqref{Baba:O:1} for the $\p_c$ of the modulation variables, \eqref{e:GW0} - \eqref{mod:sub} for the modulation variables themselves, and finally \eqref{Z:boot:0} and \eqref{mako:3} for the $Z$ quantities, with $j \ge 1$.  

 By estimating \eqref{Sam:smith:2} we obtain for $1 \le n \le 7$,  
\begin{align*} \n
\| \p_c G_Z^{(n)} \|_\infty \lesssim & |\p_c \dot{\tau}| e^{\frac s 4} \| Z^{(n)} \|_\infty + e^{\frac s 4} \| \p_c Z^{(n)} \|_\infty \\ 
\lesssim & e^{\frac s 4} \eps^{\frac 1 2} M^{2 n^2} e^{- \frac 5 4 s} + e^{\frac s 4} M^{2n^2} \eps^{\frac 1 2} e^{- \frac s 2} \lesssim M^{2n^2} \eps^{\frac 1 2} e^{- \frac s 4},  
\end{align*}
where we have invoked \eqref{mako:1} for the $\p_c \dot{\tau}$ term, and then \eqref{Z:boot:0} and \eqref{mako:3} for $Z^{(n)}$ and $\p_c Z^{(n)}$, respectively. 

\noindent For $\p_c G_A$, we perform essentially the same estimate as for $\p_c G_Z$. 

\end{proof}

\begin{lemma}[Transport Estimates] Let $c_i \in \{ \alpha, \beta \}$ for $i = 1, 2$, and fix any $0 < r \le 1$. Then the following estimates are valid for the transport speeds
\begin{align} \label{spin:1}
\| \p_{c_1 c_2} G_W \eta_{- \frac r 4} \|_\infty& \lesssim  M \eps^{\frac54}e^{\frac 5 4 s} +M^{3-r} \eps^{\frac{13}{8}-r } e^{\frac {7-5r} 4 s}\,, &&\| \p_{c_1 c_2} G_W^{(n)} \|_\infty \lesssim M^{2n^2} \eps^{\frac 5 8} e^{\frac s 2}\,, \\ \label{spin:2}
\| \p_{c_1 c_2} G_Z \|_\infty &\lesssim M^4  \eps^{\frac54}e^{\frac 5 4 s} \,, &&\| \p_{c_1 c_2} G_Z^{(n)} \|_\infty \lesssim M^{2n^2} \eps^{\frac 5 8} e^{\frac s 2}\,, \\ \label{spin:3}
\| \p_{c_1 c_2} G_A \|_\infty &\lesssim M^4  \eps^{\frac54}e^{\frac 5 4 s}  \,, &&\| \p_{c_1 c_2} G_A^{(n)} \|_\infty \lesssim M^{2n^2} \eps^{\frac 5 8} e^{\frac s 2}\,, 
\end{align}
for $1\leq n\leq 7$.
\end{lemma}
\begin{proof} We differentiate \eqref{coolie} in $\p_{c_2}$ which generates the identities 
\begin{align} 
&\p_{c_1 c_2} G_W = \mu_{c_1 c_2} + \beta_\tau \beta_2 e^{\frac s 4} \int_0^x Z^{(1)}_{c_1 c_2} + \beta_\tau^2 \beta_2 \dot{\tau}_{c_i} e^{\frac s 4} \int_0^x Z^{(1)}_{c_{i'}} \n\\\label{myoo}
& \qquad \qquad \qquad + (\dot{\tau}_{c_1 c_2}  + 2 \beta_\tau \dot{\tau}_{c_1} \dot{\tau}_{c_2})  \beta_\tau^2 \beta_2 e^{\frac s 4} \int_0^x Z^{(1)}\,, \\ \label{myoo:2}
&\p_{c_1 c_2} G_W^{(n)} = \beta_\tau \beta_2 e^{\frac s 4} Z^{(n)}_{c_1 c_2} + \beta_\tau^2 \beta_2 \dot{\tau}_{c_i} e^{\frac s 4} Z^{(n)}_{c_{i'}} + (\dot{\tau}_{c_1 c_2} + 2 \beta_\tau \dot{\tau}_{c_1} \dot{\tau}_{c_2}) \beta_\tau^2 \beta_2 e^{\frac s 4} Z^{(n)}\,,
\end{align}
for $ n \ge 1$.
Estimating the right-hand side of \eqref{myoo} yields 
\begin{align} \n
\| \p_{c_1 c_2} G_W \eta_{- \frac r 4} \|_\infty \lesssim & |\mu_{c_1 c_2}| + e^{\frac s 4} \| Z_{c_1 c_2}^{(1)} \eta_{\frac{1-r}{4}} \|_\infty + |\dot{\tau}_{c_i}| e^{\frac s 4} \| Z^{(1)}_{c_{i'}} \eta_{\frac{1-r}{4}} \|_\infty + (|\dot{\tau}_{c_1 c_2}| + |\dot{\tau}_{c}|^2) e^{\frac s 4} \| Z^{(1)} \eta_{\frac{1-r}{4}} \|_\infty   \\\n
\lesssim & M \eps^{\frac 54}e^{\frac 5 4 s} + M^{3-r}\eps^{\frac{13}{8}-r} e^{\frac {7-5r} 4 s}   + \eps^{2-r} e^{\frac s 4} M^{3-r} e^{(1+\frac 5 4 r)s}   \\ \n
&+ ( \eps e^{\frac 3 4 s} + \eps)  M^{3-r} \eps^{1-r} e^{-\frac {1+5r} 4s}  \,.
\end{align}
Above, we have used \eqref{group:1} for the $\mu_{c_1 c_2}, \dot{\tau}_{c_1 c_2}$ terms, \eqref{posty:1} for the $Z^{(1)}_{c_1 c_2}$ term, \eqref{mako:3} for the $Z^{(1)}_c$ term, \eqref{Z:boot:0} for the $Z^{(1)}$ term, \eqref{mako:1} for the $\dot{\tau}_c$ terms, and finally \eqref{e:support} for the estimation of $\eta$ in the presence of $Z$.   

Estimating the right-hand side of \eqref{myoo:2} yields for $j \ge 1$, 
\begin{align*}
\| \p_{c_1 c_2} G_W^{(n)} \|_\infty \lesssim e^{\frac s 4} \| Z_{c_1 c_2}^{(n)} \|_\infty + |\dot{\tau}_c| e^{\frac s 4} \| Z^{(n)}_c \|_\infty + (|\dot{\tau}_{c_1 c_2}| + |\dot{\tau}_c|^2) e^{\frac s 4} \| Z^{(n)} \|_\infty \lesssim M^{2n^2} \eps^{\frac 5 8} e^{\frac s 2}\,. 
\end{align*}
We have invoked \eqref{posty:1} for the $Z_{c_1 c_2}^{(j)}$ term, \eqref{mako:1} for the $\dot{\tau}_c$ term, \eqref{mako:3} for the $Z_c$ term, \eqref{Z:boot:0} for the $Z^{(j)}$ term, and \eqref{group:1} for the $\dot{\tau}_{c_1c_2}$ term. 

Next, we differentiate \eqref{Sam:smith:1} - \eqref{Sam:smith:2} in $\p_{c_2}$ to arrive at
\begin{align} \label{lose:1}
&\p_{c_1 c_2} G_Z = \dot{\tau}_{c_1 c_2} \beta_\tau G_Z + \dot{\tau}_{c_i} \beta_\tau \p_{c_{i'}}G_Z + \beta_\tau e^{\frac s 4} \Big( \beta_2 \kappa_{c_1 c_2} -\dot{\xi}_{c_1 c_2} + Z_{c_1 c_2} \Big)\,, \\ \label{lose:2}
&\p_{c_1 c_2} G_Z^{(n)} = \dot{\tau}_{c_1 c_2} \beta_\tau G_Z^{(n)} + \dot{\tau}_{c_i} \beta_\tau \p_{c_{i'}} G_Z^{(n)} + \beta_\tau e^{\frac s 4} Z_{c_1 c_2}^{(n)}\,. 
\end{align} 
Estimating the right-hand side gives \eqref{spin:2} via 
\begin{align} \n
\| \p_{c_1 c_2} G_Z \|_\infty \lesssim & |\dot{\tau}_{c_1 c_2}| \| G_Z \|_\infty + |\dot{\tau}| \| \p_c G_Z \|_\infty + e^{\frac s 4} \Big( |\kappa_{c_1 c_2}| + |\dot{\xi}_{c_1 c_2}| + \| Z_{c_1 c_2} \|_\infty \Big) \\
\lesssim & \eps e^{s} +  \eps^{\frac 3 4} e^{\frac s 4} + e^{\frac s 4} \Big( M^3 \eps^{\frac 5 4} e^{s} + M^4 \eps^{\frac 5 4} e^{s} + \eps^{\frac 5 8} e^{\frac s 4} \Big)\n
\end{align} 
Above we have invoked \eqref{Z:transport:1} and \eqref{r:est:2} for the $G_Z$ and $\p_c G_Z$ terms, respectively. We have also invoked \eqref{group:1} - \eqref{group:2} for the second derivatives of the modulation variables and \eqref{posty:1} for the $Z_{c_1 c_2}$ term. 

For the right-most estimate in \eqref{spin:2}, we estimate the right-hand side of \eqref{lose:2}, 
\begin{align*}
\| \p_{c_1 c_2} G_Z^{(n)} \|_\infty \lesssim & |\dot{\tau}_{c_1 c_2}| \| G_Z^{(n)} \|_\infty + |\dot{\tau}_c| \| \p_c G_Z^{(n)} \|_\infty + e^{\frac s 4} \| Z^{(n)}_{c_1 c_2} \|_\infty \\
\lesssim & M^{2n^2} \eps e^{-\frac{s}{4}} + M^{2n^2} \eps e^{- \frac s 4} +  M^{2n^2}\eps^{\frac 5 8} e^{\frac s 2} \lesssim M^{2n^2} \eps^{\frac 5 8} e^{\frac s 2}\,, 
\end{align*}
where we have invoked \eqref{Z:transport:1} and \eqref{r:est:2} for the $G_Z^{(n)}$ and $\p_c G_Z^{(n)}$ terms, respectively. 

A nearly identical estimate is valid for \eqref{spin:3}. 

\end{proof}

\subsection{Forcing estimates} \label{subsection:Forcing}

In this subsection, we will provide pointwise estimates on the forcing terms $F_W, F_Z, F_A$, defined in \eqref{def:FW} - \eqref{def:FA} as well as their various derivatives (spatial and parameter).

\subsubsection{Forcing estimates for $(W,Z,A)$ and its derivatives}

We now provide estimates on the forcing of $(W,Z,A)$ and their spatial derivatives. 

\begin{lemma} For the forcing quantities defined in \eqref{def:FW} - \eqref{def:FA} and \eqref{diff:eq}, the following estimates are valid 
\begin{align}\label{FW:est:1}
\| F_W \|_\infty &\le \eps^{\frac 3 4} e^{- \frac 3 4 s},  &\| F_W^{(n)} \|_\infty &\le \eps^{\frac 3 4} e^{-s} \qquad \text{ for } 1\leq n \leq 8 \\ \label{FW:est:2}
\| \tilde{F}_W \|_\infty &\le  e^{- \frac 3 4 s}\,, &\| \tilde{F}_{W,1} \eta_{\frac 1 4} \|_\infty &\le \eps^{\frac{1}{10}}\,, \\ \label{FW:est:3}
\| F_{W,n} \eta_{\frac 1 5} \|_\infty &\les M^{n^2 - 1} \quad \text{ for } 2 \le n \le 8\,, & \| F_{W}^{(1)} \eta_{\frac 1 4} \|_\infty &\le e^{- \frac 1 2 s}\,, \\ \label{Motor:1}
\| F_{W,1} \eta_{\frac 1 5} \|_\infty &\lesssim e^{- \frac 1 2 s}  
\end{align}
\end{lemma}
\begin{proof} We use definition \eqref{def:FW} to estimate 
\begin{align} \n
\| F_W \|_\infty \lesssim &e^{- \frac 3 4 s} \| A \|_\infty (\| Z \|_\infty + \| e^{- \frac s 4} W + \kappa \|_\infty) \lesssim  e^{- \frac 3 4 s} M \eps (\eps^{\frac 5 4} +M) \lesssim_M \eps e^{- \frac 3 4 s}\,,  
\end{align}
which establishes the first inequality in \eqref{FW:est:1}. 

We now want to estimate $\tilde{F}_W$, for which we use definition \eqref{diff:eq0} to bound 
\begin{align} \n
\| \tilde{F}_W \|_\infty \le & |\beta_\tau| e^{- \frac 3 4 s}|\dot{\kappa}| + \| F_W \|_\infty + |\beta_\tau - 1| \| \bar{W} \p_x \bar{W} \|_\infty + \| G_W \eta_{- \frac 1 5} \|_\infty \| \p_x \bar{W} \eta_{\frac 1 5} \|_\infty \\ \n
\lesssim & e^{- \frac 3 4 s} \eps^{\frac 1 8} + \eps^{\frac 3 4} e^{- \frac 3 4 s} + \eps^{\frac 1 6} e^{- \frac 3 4 s} + M^{\frac{11}{5}} \eps^{\frac1 5} e^{- \frac 3 4 s} \\ \n
\lesssim & \eps^{\frac 1 8} e^{- \frac 3 4 s}\,, 
\end{align}

\noindent which establishes the first inequality in \eqref{FW:est:2}. Above, we have invoked estimate \eqref{e:GW0} for the $\dot{\kappa}$ term, the previously established estimate on $\| F_W \|_\infty$ in \eqref{FW:est:1}, \eqref{e:1beta:bnd} for the $\beta_\tau - 1$ quantity, and estimate \eqref{GW:transport:est} for the $G_W$ term, with $r = - \frac 4 5$. 

Estimating the expression \eqref{FW:to:the:n}, we obtain 
\begin{align} \n
\| F_W^{(n)} \|_\infty \lesssim &e^{- \frac 3 4 s} \sum_{j = 1}^{n -1} \| A^{(j)} \|_\infty \Big( \| Z^{(n-j)} \|_\infty + e^{- \frac s 4} \|W^{(n-j)} \|_\infty  \Big) \\ \n
& + e^{- \frac 3 4 s} \|A^{(n)} \|_\infty \| e^{- \frac s 4}W \mathbbm{1}_{B_f} + \kappa \|_\infty + e^{- \frac 3 4 s} \| A \|_\infty (\| Z^{(n)} \|_\infty + e^{- \frac s 4} \| W^{(n)} \|_\infty) \\   \label{Fwn:est:s:dec}
\lesssim & M e^{- 2 s} (e^{- \frac 5 4 s} + e^{- \frac s 4}) +  e^{- 2 s} +  \eps e^{- \frac 3 4 s}  (e^{- \frac 5 4 s} + e^{- \frac s 4}) \les_M \eps e^{-s} \,,
\end{align}

\noindent which establishes the second inequality in \eqref{FW:est:1}. To estimate \eqref{Fwn:est:s:dec}, we have invoked \eqref{e:support} and estimates \eqref{Z:boot:0} - \eqref{A:boot:9}. 

We now turn to the second inequality in \eqref{FW:est:3}. For this, we appeal to the definition \eqref{FW:to:the:n}
\begin{align} \n
\| F_W^{(1)} \eta_{\frac 1 4} \|_\infty \lesssim &e^{- \frac 3 4 s} \| A^{(1)} \eta_{\frac 1 4} \|_\infty \Big( \| Z \|_\infty + \| e^{- \frac s 4} W \mathbbm{1}_{B_f} + \kappa \|_\infty \Big) + e^{- \frac 3 4 s} \|A  \|_\infty \| Z^{(1)} \eta_{\frac 1 4 } \|_\infty \\ \n
& + e^{- \frac 3 4 s} \| A \eta_{\frac{1}{20}} \|_\infty \Big( \| \bar{W}^{(1)} \eta_{\frac 1 5} \|_\infty + \| \tilde{W}^{(1)} \eta_{\frac 1 5} \|_\infty  \Big) \\ \n
\lesssim & M^2 e^{-2s} (M \eps e^{\frac 5 4 s}) ( \eps^{\frac 5 4} + M ) + M^4 \eps^2 e^{- \frac 3 4 s}+ M e^{- \frac 3 4 s} \eps^{\frac 5 4} (M\eps e^{\frac 5 4 s})^{\frac 1 5} \ell \log M \\ \label{spoke:2}
\lesssim & \eps^{\frac 1 8 } e^{- \frac 1 2 s}\,, 
\end{align}
where above we have used the inequality $\eta_{\frac r 4} \lesssim (M\eps e^{\frac 5 4 s})^{r}$ in the support of $A, Z$, as well as estimates \eqref{Z:boot:0} - \eqref{A:boot:9} and \eqref{e:Wtilde:1:bootstrap} and \eqref{decay:bar:2} for the spatial decay property of $\bar{W}^{(1)}$.  

We now arrive at the second estimate in \eqref{FW:est:2}. An appeal to \eqref{diff:eq} gives
\begin{align} \n
\|\tilde{F}_{W,1} \eta_{\frac 1 4}\|_\infty \lesssim & \| F_W^{(1)} \eta_{\frac 1 4} \|_\infty + \| \bar{W}^{(2)} \tilde{W} \eta_{\frac 1 4}\|_\infty + \| G_W^{(1)} \tilde{W}^{(1)} \eta_{\frac 1 4} \|  + | \beta_\tau - 1| \Big(\| \bar{W} ~\bar{W}^{(2)} \eta_{\frac 1 4} \|_\infty \\ \n
& \qquad + \| \bar{W}^{(1)} \eta_{\frac 1 8}| \| \bar{W}^{(1)} \eta_{\frac 1 8} \|_\infty\Big) + \| G_W \bar{W}^{(2)} \eta_{\frac 1 4} \|_\infty + \| G_W^{(1)} \bar{W}^{(1)} \eta_{\frac 1 4} \|_\infty\\ \n
\lesssim & \| F_W^{(1)} \eta_{\frac 1 4} \|_\infty + \| \bar{W}^{(2)} \eta_{\frac{9}{20}} \|_\infty \| \tilde{W} \eta_{- \frac{1}{20}} \|_\infty + (M\eps e^{\frac 5 4 s})^{\frac 1 5} \| G^{(1)}_W \|_\infty \| \tilde{W}^{(1)} \eta_{\frac 1 5} \|_\infty  \\ \n
& + \eps^{\frac 1 6} e^{- \frac 3 4 s} + \| \bar{W}^{(2)} \eta_{\frac{9}{20}} \|_\infty \| G_W \eta_{- \frac 1 5} \|_\infty + (M\eps e^{\frac 5 4 s})^{\frac 1 5} \| G_W^{(1)} \|_\infty \| \bar{W}^{(1)} \eta_{\frac 1 5} \|_\infty \\ \n
\lesssim & \eps^{\frac 18}e^{-\frac s 2}+ \eps^{\frac{3}{20}} + \eps^{\frac 1 6} e^{- \frac 3 4 s} + \eps^{\frac 1 4} e^{- \frac 3 4 s} + \eps^{\frac 1 6} e^{- \frac 3 4 s}  +M^{\frac {11} 5}  \eps^{\frac 1 5} e^{- \frac 3 4 s} 
\lesssim  \eps^{\frac{3}{20}}\,. 
\end{align}
Above, we have used the bootstrap estimates \eqref{thurs:1} and \eqref{e:Wtilde:1:bootstrap} on $\tilde W$, the bound \eqref{decay:bar:2} regarding the decay of $\bar{W}^{(2)}$, as well as estimate \eqref{FW:est:3}, which has already been established. 
We have moreover invoked the previously established estimates \eqref{GW:transport:est} on the $G_W$ quantity with $r = - \frac 4 5$ and the $G_W^{(1)}$ quantity. 

To prove \eqref{Motor:1}, we first recall the definition \eqref{F.W.n.bot}, according to which if we pair with estimate \eqref{spoke:2} yields 
\begin{align} \n
\| F_{W,1} \eta_{\frac 1 5} \|_\infty \le &\| F_{W}^{(1)} \eta_{\frac 1 5} \|_\infty + \| G_W^{(1)} \|_\infty \| W^{(1)} \eta_{\frac 1 5} \|_\infty \les \eps^{\frac18}e^{- \frac 1 2 s} + \ell M^3\log M e^{-s}\les \eps^{\frac18}e^{- \frac 1 2 s} \, , 
\end{align}
where we have also invoked estimate \eqref{GW:transport:est}, and the bootstrap \eqref{e:uniform:W1}. 

We now appeal to the definition \eqref{F.W.n} to perform the third estimate, \eqref{FW:est:3}. We estimate also with the help of \eqref{Fwn:est:s:dec}
\begin{align*}
\| F_W^{(n)} \eta_{\frac 1 5} \|_\infty \lesssim& \eps^{\frac 34} e^{-s} (M\eps e^{\frac 5 4s})^{\frac 4 5} = M^{\frac 4 5} \eps^{\frac 7 4}\,, \\
 \Big\| 1_{n \ge 3}\beta_\tau \sum_{j = 2}^{n-1} \binom{n}{j} W^{(j)} W^{(n+1 - j)} \eta_{\frac 1 5}\Big\|_\infty \lesssim & \sum_{j = 2}^{n-1} M^{j^2} M^{(n+1-j)^2} \lesssim M^{n^2 - 1} \\
 \Big\| \sum_{j = 1}^n \binom{n}{j} G_W^{(j)} W^{(n+1-j)} \eta_{\frac 1 5}\Big\|_\infty \lesssim &\sum_{j = 1}^n M^{2j^2} e^{-s} M^{(n+1-j)^2} \le \eps^{\frac 1 2}\,. 
\end{align*}
Above we have invoked the elementary inequality $j^2 + (n+1 - j)^2 \le - 1 + n^2$ for $n \ge 3$, and $2 \le j \le n-1$, as well as the estimates on $G_W^{(j)}$ in \eqref{GW:transport:est}, and estimates \eqref{weds:1} on $W^{(n)}$.

\end{proof}

We now state a lemma regarding localized estimates, on $|x| \le \ell$, which have an enhanced scaling. 
\begin{lemma} The following estimates are valid: 
\begin{align}\label{e:FW6:decay}
\sup_{|x| \le \ell} |\tilde{F}_{W,6}| \les {\ell \eps^{\frac 1 5}}\,, && \sup_{|x| \le \ell} |\tilde{F}_{W,7}| \le { \eps^{\frac 1 5}}\,, &&& \sup_{|x| \le \ell} |\tilde{F}_{W,8}| \les {M \eps^{\frac 1 5}}\,. 
\end{align}
\end{lemma}
\begin{proof} We use the definition \eqref{diff:eq} to estimate via 
\begin{align} \n
\sup_{|x| \le \ell } | \tilde{F}_{W,6} | \lesssim &\|F_W^{(6)} \|_\infty + \sum_{j = 2}^{5} \sup_{|x| \le \ell} | \tilde{W}^{(7-j)} | + \sum_{j = 1}^6 \sup_{|x| \le \ell} | \tilde{W}^{(6-j)} | + \eps^{\frac 1 2} \eps^{\frac 1 5} \\ \n
& + \eps^{\frac 1 6} e^{- \frac 3 4 s} + \sum_{j = 1}^6 M^{2j^2} e^{-s} + \eps^{\frac 1 2} \\ \n
\lesssim & \eps^{\frac 3 4} e^{-s} + \ell \eps^{\frac 1 5} +  \eps^{\frac 1 2} \eps^{\frac 1 5}  + \eps^{\frac 1 6} e^{- \frac 3 4 s} + \eps^{\frac 1 2} \les \ell \eps^{\frac 1 5}\,, 
\end{align}
where we have invoked estimates \eqref{FW:est:1} with $n = 6$, and \eqref{e:W6:bootstrap}. 

The identical argument applies to the estimate of $\tilde{F}_{W,7}$ and $\tilde{F}_{W,8}$. 
\end{proof}

\begin{lemma} For $F_Z$ defined in \eqref{def:FZ}, the following estimates are valid 
\begin{align}\label{Z:0:order}
\| F_Z \|_\infty \le & \eps^{\frac 3 4} e^{-s}\,,  \\ \label{FZ:est:1}
\|  F_{Z,1}  \| \le &  \eps^{\frac 1 2} e^{- \frac 5 4 s} I(s) + e^{- \frac 3 2 s}\,, \\ \label{my:generation}
\| F_Z^{(n)} \|_\infty \le & \eps^{\frac 3 4} e^{- \frac 5 4 s}  \,,\\ \label{gen:n:FZ:est}
\| F_{Z, n} \|_\infty \les & {M^{2n^2-1}} e^{- \frac 5 4 s}\,.
\end{align}
for $2 \le n \le 8$,
\noindent where $I(s)$ is an integrable function of $s$ satisfying the bound $\int_{s_0}^s |I(s')| ds' < 1$.  
\end{lemma}
\begin{proof} For estimate \eqref{Z:0:order}, we use the definition \eqref{def:FZ} to estimate 
\begin{align*}
\|F_Z \|_\infty \lesssim e^{-s} \|A \|_\infty \Big( \| e^{- \frac s 4} W + \kappa \|_\infty + \| Z \|_\infty \Big) \lesssim M^2\eps e^{-s}\,, 
\end{align*}
where we have invoked \eqref{A:boot:9}, as well as \eqref{est:W:good}. 

To estimate $F_Z^{(n)}$, we recall definition \eqref{def:FZn}, which requires us to estimate the following four types of terms 
\begin{align*}
\sum_{j = 1}^n \| \beta_\tau e^{-s} A^{(j)} (e^{- \frac s 4} W + \kappa)^{(n-j)} \|_\infty &\lesssim  e^{-s} \| A^{(j)} \|_\infty \| e^{- \frac s 4} W \mathbbm{1}_{B_f} + \kappa \| \lesssim  M e^{-s} e^{- \frac 5 4 s}\,, \\
\sum_{j = 1}^n \| \beta_\tau e^{-s} A^{(j)} Z^{(n-j)} \|_\infty &\lesssim  e^{-s} \| A^{(j)} \|_\infty \| Z \| \lesssim M \eps^{\frac 5 4} e^{- \frac 9 4s}\,, \\
\| \beta_\tau e^{-s} A \beta_3 e^{- \frac s 4} W^{(n)} \|_\infty & \lesssim_M  \eps e^{- \frac 5 4 s}\,, \\
\| \beta_\tau e^{-s} A \beta_4 Z^{(n)} \|_\infty & \lesssim_M  \eps e^{- \frac 9 4 s}\,.
\end{align*}
Again, we have used estimates \eqref{Z:boot:0} - \eqref{A:boot:9}, as well as estimates \eqref{weds:1} for derivatives of $W$. 

 We now provide the estimate \eqref{gen:n:FZ:est}. Recall the definition \eqref{F.Z.n.bot}. For this, when coupled with \eqref{my:generation}, we need to estimate further the following two terms
\begin{align*}
\| 1_{n \ge 2} \beta_\tau \beta_2 \sum_{j = 2}^{n} \binom{n}{j} W^{(j)} Z^{(n+1-j)} \|_\infty &\lesssim M^{2 n^2 -  1} 1_{n \ge 2} e^{- \frac 5 4 s}\,, \\
\| \sum_{j = 1}^n G_Z^{(j)} Z^{(n+1-j)} \|_\infty &\lesssim_M  e^{- \frac 9 4 s}\,. 
\end{align*} 
Above, we have invoked estimates \eqref{weds:1} for derivatives of $W$, \eqref{Z:boot:0} for $Z$, as well as \eqref{Z:transport:1} for the $G_Z^{(j)}$ terms. 

 For estimate \eqref{FZ:est:1}, we estimate all of the terms above by $e^{- \frac 3 2 s}$ with the exception of 
\begin{align}
|\beta_\tau \beta_3 e^{- \frac 54 s} A W^{(1)} \circ \Phi_Z| \le 10 \eps^{\frac 5 4} e^{- \frac 5 4 s} |\eta_{- \frac 1 5} \circ \Phi_Z| \le \eps e^{- \frac 5 4 s} I(s)\,,\n
\end{align}
where we have invoked the trajectory estimate \eqref{traj:large:z}.

\end{proof}

\begin{lemma} \label{phone:1}   For $F_A$ defined in \eqref{def:FA}, the following estimates are valid 
\begin{align}\label{A:0:order}
\| F_A \|_\infty &\les M^{\frac 1 2} e^{-s}\,,\\ \label{FA:est:1}
\|  F_{A,1}  \|_\infty &\le   e^{- \frac 5 4 s} I(s)\,, \\ \label{myA:generation}
\| F_A^{(n)} \|_\infty &\lesssim M^{n^2} e^{- \frac 5 4 s}, \text{ for } 2 \le n \le 8\,, \\  \label{gen:n:FA:est}
\| F_{A, n} \|_\infty &\les {M^{2n^2-1}} e^{- \frac 5 4 s} \text{ for } 2 \le n \le 8\,,  
\end{align}
where $I(s)$ is an integrable function of $s$ satisfying the bound $\int_{s_0}^s |I(s')| ds' < M$.  
\end{lemma}
\begin{proof} First, we estimate $F_A$ via the definition in \eqref{def:FA} 
\begin{align}
\|F_A \|_\infty \lesssim e^{-s} \|A\|_\infty^2 + e^{-s} \| e^{- \frac s 4} W  + \kappa + Z \|_\infty^2 + e^{-s} \| e^{- \frac s 4} W + \kappa - Z \|_\infty^2 \les M^{\frac12 }e^{-s}\,,
\end{align}
where we have used estimate \eqref{Z:boot:0}, \eqref{A:boot:9}, \eqref{est:W:good}, and \eqref{mod:sub}, coupled with the fact that $M$ is large relative to $\kappa_0$.  

We now turn to \eqref{myA:generation}, for $n \ge 1$, for which we consider \eqref{okey:1}. 
\begin{align*}
\| F_A^{(n)} \|_\infty \lesssim & e^{-s} \sum_{j = 0}^n \|A^{(j)} \|_\infty \| A^{(n-j)} \|_\infty + e^{-s} \sum_{j = 0}^n \Big( \| (e^{- \frac s 4} W  +  \kappa)^{(j)}\|_\infty + \| Z^{(j)} \|_\infty \Big) \times \\
&  \Big( ( \| (e^{- \frac s 4} W + \kappa)^{(n-j)}\|_\infty + \| Z^{(n - j)} \|_\infty \Big) \\
\lesssim & M^{n^2} e^{- \frac 5 4 s}\,. 
\end{align*}
Above, we have invoked \eqref{Z:boot:0} - \eqref{A:boot:9} as well as \eqref{weds:1} and \eqref{est:W:good}. 

The remaining two estimates, \eqref{FA:est:1} and \eqref{gen:n:FA:est}, follow in the same manner as \eqref{FZ:est:1} - \eqref{gen:n:FZ:est}. 
\end{proof}

\subsubsection{$\nabla_{a,b}$ forcing estimates}

We now develop estimates regarding the $\p_\alpha$ and $\p_\beta$ derivatives of the forcing terms $F_W, F_Z, F_A$. We start with the quantities $\p_\alpha F_W$ and $\p_\beta F_W$ in the following lemma. 

\begin{lemma} \label{L:N:F} Let $n \ge 1$. Then, 
\begin{align} \label{Fwc.est.ult:0}
\| \p_c F_W \|_\infty &\lesssim M \eps^{\frac 3 4} e^{- \frac s 4} \,, &  \| F_{W,0}^{c} \|_\infty &\lesssim \eps^{\frac 1 8} \,,\\  \label{Fwc.est.ult}
\| \p_c F_W^{(n)} \|_\infty &\lesssim  \eps^{\frac 3 4} e^{-\frac s 4}\,, & \| F_{W,n}^{c} \eta_{\frac{1}{20}} \|_\infty &\lesssim M^{-1} M^{(n+2)^2} \eps^{\frac34} e^{\frac 3 4 s}\,. 
\end{align}
\end{lemma}
\begin{proof} First, we use equation \eqref{dc.FW} to estimate the first quantity in \eqref{Fwc.est.ult}. We proceed in order, starting with  
\begin{align} \n
\| \p_c F_W \|_\infty \lesssim &  \| \p_c \dot{\tau} \beta_\tau^2 e^{- \frac 3 4 s} A \Big( \beta_3 Z + \beta_4 (e^{- \frac s 4}W + \kappa) \Big) \|_\infty + \| \beta_\tau e^{- \frac 3 4 s} \p_c A \Big( \beta_3 Z + \beta_4 (e^{- \frac s 4} W + \kappa) \Big) \|_\infty \\ \n
& + \| \beta_\tau e^{- \frac 3 4 s} A \Big( \beta_3 \p_c Z + \beta_4 (e^{- \frac s 4} \p_c W + \p_c \kappa) \Big) \|_\infty \\ \n
\lesssim & |\p_c \dot{\tau}| e^{- \frac 3 4 s} \| A \|_\infty \Big( \| Z \|_\infty + \| e^{- \frac s 4}W \mathbbm{1}_{B_f} + \kappa \|_\infty \Big) + e^{- \frac 3 4 s} \| \p_c A \|_\infty \Big( \| Z \|_\infty + \| e^{- \frac s 4}W \mathbbm{1}_{B_f}+ \kappa \|_\infty \Big) \\ \n
& + e^{- \frac 3 4 s} \| A \|_\infty \Big( \| \p_c Z \|_\infty + \| e^{- \frac s 4} \p_c W \|_\infty + |\p_c \kappa| \Big) \\
\lesssim &M \eps^{\frac 3 2} e^{- \frac 3 4 s}  \Big( \eps^{\frac 5 4} + M \Big) + e^{- \frac 3 4 s} \eps^{\frac 1 2} \Big( \eps^{\frac 5 4} + M \Big) + M\eps e^{- \frac 3 4 s} \Big( \eps^{\frac 1 2} +\eps e^{- \frac 3 4 s}  + \eps^{\frac 3 8}  \Big)\,.\n
\end{align}
Above, we have invoked repeatedly estimates \eqref{Z:boot:0} - \eqref{A:boot:9}, as well as \eqref{mako:1} - \eqref{mako:3}. 

Next, we use equation \eqref{dc.FW:0} to estimate the second quantity in \eqref{Fwc.est.ult:0} via 
\begin{align*}
\| e^{- \frac 3 4 s} \beta_\tau \p_c \dot{\kappa}+ e^{- \frac 3 4 s} \dot{\kappa} \p_c \dot{\tau} \beta_\tau^2 \|_\infty &\lesssim e^{- \frac 3 4 s} |\p_c \dot{\kappa}| + e^{- \frac 3 4 s} |\dot{\kappa}| |\p_c \dot{\tau}| \lesssim \eps^{\frac 1 4 } e^{-\frac s 4 } + e^{- \frac 3 4 s} \eps^{\frac 5 8}\,, \\
\| \p_c G_W W^{(1)} \|_\infty &\le \| \p_c G_W \eta_{- \frac 1 5} \|_\infty \| W^{(1)} \eta_{\frac 1 5} \|_\infty \lesssim_M \eps^{\frac 1 2}\,, \\
\| W^{(1)} \dot{\tau}_c W \|_\infty &\le |\dot{\tau}_c| \| W^{(1)} \eta_{\frac 1 5} \|_{\infty} \| W \eta_{- \frac{1}{20}} \|_\infty \lesssim \eps^{\frac 1 2} \,,
\end{align*} 
where we have invoked the bootstrap bounds \eqref{e:GW0}, \eqref{mako:1}, and for the second line above we have invoked \eqref{r:est:1} with $r = \frac 4 5$. 

Next, we use equation \eqref{dcbcFwn} to estimate the first quantity in \eqref{Fwc.est.ult}. Specifically, 
\begin{align*}
\| \p_c F_W^{(n)} \|_\infty \lesssim & \Big\| \sum_{j = 0}^n \binom{n}{j} \p_c \dot{\tau} \beta_\tau^2 e^{- \frac 3 4 s} A^{(j)} \Big( \beta_3 Z^{(n-j)} + \beta_4 (e^{- \frac s 4} W + \kappa)^{(n-j)} \Big) \Big\|_\infty \\ 
& + \Big\| \sum_{j = 0}^n \binom{n}{j} \beta_\tau e^{- \frac 3 4 s} \p_c A^{(j)} \Big( \beta_3 Z^{(n-j)} + \beta_4 (e^{- \frac s 4} W + \kappa)^{(n-j)} \Big) \Big\|_\infty \\ 
& + \Big\| \sum_{j = 0}^n \binom{n}{j} \beta_\tau e^{- \frac 3 4 s} A^{(j)} \Big( \beta_3 \p_c Z^{(n-j)} + \beta_4 (e^{- \frac s 4} \p_c W + \p_c \kappa)^{(n-j)} \Big) \Big\|_\infty \\
=: & \mathcal{O}_1 + \mathcal{O}_2 + \mathcal{O}_3\,. 
\end{align*}

Bounding $ \mathcal{O}_1$, we obtain
\begin{align*}
\mathcal{O}_1 &\lesssim  \sum_{j = 1}^{n-1} |\p_c \dot{\tau}| e^{- \frac 3 4 s} \| A^{(j)} \|_\infty \Big( \| Z^{(n-j)} \|_\infty + \| e^{- \frac s 4} W^{(n-j)}  \|_\infty \Big) \\
 & \qquad  + |\p_c \dot{\tau}| e^{- \frac 3 4 s} \|A \|_\infty \Big( \| Z^{(n)} \|_\infty + \| e^{- \frac s 4} W^{(n)} \|_\infty \Big) + |\p_c \dot{\tau}| e^{- \frac 3 4 s} \| A^{(n)} \|_\infty \\
 & \qquad \times \Big( \| Z \|_\infty + \| e^{- \frac s 4} W \mathbbm{1}_{B_f}+ \kappa \|_\infty \Big) \\
& \lesssim_M   \eps^{\frac 1 2} e^{- 2 s} ( e^{- \frac 5 4 s} + e^{- \frac s 4} ) + \eps^{\frac 3 2} e^{- \frac 3 4 s} ( e^{- \frac 5 4 s} + e^{- \frac s4} ) + \eps^{\frac 1 2} e^{- 2 s} ( \eps^{\frac 5 4} + 1 ) \\
 &\lesssim_M  \eps^{\frac 5 4} e^{-s},
\end{align*}
where we have invoked estimates \eqref{Z:boot:0} - \eqref{A:boot:9}, as well as \eqref{mako:1}. 

We now bound $ \mathcal{O}_2$
\begin{align*}
\mathcal{O}_2 &\lesssim  \sum_{j = 1}^{n-1} e^{- \frac 3 4 s} \| \p_c A^{(j)} \|_\infty \Big( \| Z^{(n-j)} \|_\infty + e^{- \frac s 4} \| W^{(n-j)} \|_\infty \Big) \\
&\qquad + e^{- \frac 3 4 s} \| \p_c A \|_\infty \Big( \| Z^{(n)} \|_\infty + e^{- \frac s 4} \| W^{(n)} \|_\infty \Big) + e^{- \frac 3 4 s} \| \p_c A^{(n)} \|_\infty \\
& \qquad \times \Big( \| Z \|_\infty + \| e^{- \frac s 4} W \mathbbm{1}_{B_f}+ \kappa \|_\infty \Big) \\
& \lesssim_M  \eps^{\frac 1 2} e^{- \frac 5 4 s}  ( e^{- \frac 5 4 s} + e^{- \frac s 4} ) + e^{- \frac 3 4 s} \eps^{\frac 1 2} ( e^{- \frac 5 4 s} + e^{- \frac s 4} ) + \eps^{\frac 1 2} e^{- \frac 5 4 s}  ( \eps^{\frac 5 4} + 1) \\
&\lesssim_M  \eps^{\frac 3 4} e^{-s}.
\end{align*}
We have invoked estimates \eqref{Z:boot:0} - \eqref{A:boot:9}, as well as \eqref{mako:3}.

Finally, we estimate  $\mathcal{O}_3$
\begin{align*}
\mathcal{O}_3 &\lesssim  \sum_{j = 1}^{n-1} e^{- \frac 3 4 s} \|A^{(j)} \|_\infty \Big( \| \p_c Z^{(n-j)} \|_\infty + e^{- \frac s 4} \| \p_c W^{(n-j)} \|_\infty \Big) \\
& \qquad + e^{- \frac 3 4 s} \| A \|_\infty \Big( \| \p_c Z^{(n)} \|_\infty + e^{- \frac s 4} \| \p_c W^{(n)} \|_\infty \Big) + e^{- \frac 3 4 s} \| A^{(n)} \|_\infty \\
& \qquad \times  \Big( \| \p_c Z \|_\infty + e^{- \frac s 4} \| \p_c W \|_\infty + | \p_c \kappa | \Big) \\
&\lesssim_M  e^{- 2 s} ( \eps^{\frac 1 2} e^{- \frac 1 2 s} + \eps^{\frac 3 4 }e^{ \frac s 2}  ) + e^{- \frac 3 4 s} \eps ( \eps^{\frac 1 2} e^{- \frac s 2} + \eps^{\frac 3 4} e^{\frac s 2} ) + e^{- 2 s} ( \eps^{\frac 1 2} +\eps^{\frac34}e^{\frac 3 4 s} + \eps^{\frac 3 8} ) \\
&\lesssim_M  \eps e^{- \frac s 4}.
\end{align*}
We have used the bootstrap bounds \eqref{Z:boot:0} - \eqref{A:boot:9}, as well as \eqref{Baba:O:1} and \eqref{mako:2} - \eqref{mako:3}. 

\noindent We now remark that, according to \eqref{e:support}, 
\begin{align} \label{born:2}
\| \p_c F_W^{(n)} \eta_{\frac{1}{20}} \|_\infty \lesssim \eps^{\frac 3 4} e^{- \frac s 4} (M\eps e^{\frac 5 4 s})^{\frac{1}{5}} = M^{\frac 1 5} \eps^{\frac {19} {20}}\,. 
\end{align}

Finally, we use equation \eqref{trek:mix} to estimate the second quantity in \eqref{Fwc.est.ult}. In addition to estimate \eqref{born:2}, we need to estimate the following two quadratic quantities in $W$  
\begin{align} \label{truck:0}
 \|  \sum_{j = 1}^n \binom{n}{j} \beta_\tau W^{(1+j)} \p_c W^{(n-j)} \eta_{\frac{1}{20}} \|_\infty &\lesssim M^{(1+j)^2} M^{(n-j+2)^2} \eps^{\frac34}e^{\frac 3 4 s}  \lesssim M^{-1} M^{(n+2)^2} \eps^{\frac34}e^{\frac 3 4 s}, 
 \end{align}
and similarly 
 \begin{align} \n
 \| 1_{n \ge 2} \sum_{j = 0}^{n-2} \binom{n}{j} \beta_\tau W^{(n-j)} \p_c W^{(j+1)} \eta_{\frac{1}{20}} \|_\infty & \lesssim 1_{n \ge 2} \sum_{j = 0}^{n-2} \| W^{(n-j)}\eta_{\frac{1}{20}} \|_\infty \| \p_c W^{(j+1)} \|_\infty \\ \label{truck:-1}
&  \lesssim M^{(n-j)^2} M^{(j+3)^2} \eps^{\frac34}e^{\frac 3 4 s}\lesssim M^{-1} M^{(n+2)^2} \eps^{\frac34}e^{\frac 3 4 s}\,.
\end{align}
For both of the above estimates, \eqref{truck:0} and \eqref{truck:-1}, we have invoked \eqref{weds:1} and \eqref{pc:W0} - \eqref{mako:6}. 
 
Next, using again \eqref{trek:mix}, we need to estimate the following two quantities 
\begin{align} \label{truck:1}
 \| 1_{n \ge 1} \sum_{j = 0}^{n-1} \binom{n}{j} G_W^{(n-j)} \p_c W^{(j+1)}\eta_{\frac{1}{20}} \|_\infty &\lesssim 2 M^{2(n-j)^2} e^{-s} M^{(j+2)^2} \eps^{\frac34}e^{\frac 3 4 s} \les_M \eps^{\frac 3 4} e^{- \frac s 4}\,, \\ \label{truck:2}
\| \sum_{j = 0}^n \binom{n}{j} \p_c G_W^{(j)} W^{(n-j+1)}\eta_{\frac{1}{20}}  \|_\infty &\lesssim \sum_{j = 0}^n \| \p_c G_W^{(j)} \eta_{- \frac{3}{20}} \|_\infty \| W^{(n-j+1)} \eta_{\frac 1 5} \|_\infty \lesssim_M \eps^{\frac 9 {10}} e^{\frac s 4}\,.
\end{align}
Above, we have appealed to estimates \eqref{GW:transport:est} with $r = \frac 3 5$ and \eqref{r:est:1} on $G_W$ and $\p_c G_W$.

Finally, according to \eqref{trek:mix}, we need to estimate  
\begin{align} \label{backer:2}
\|\sum_{j = 0}^n \binom{n}{j} \dot{\tau}_c \beta_\tau^2 \sum_{j = 0}^n \binom{n}{j} W^{(1+j)} W^{(n-j)}\eta_{\frac{1}{20}} \|_\infty \lesssim_M \eps^{\frac 1 2}\,. 
\end{align}
\noindent Above, we have used the elementary inequality 
\begin{align*}
(1 + j)^2 + (n - j+2)^2 \le -1 + (n+2)^2 \text{ for } n \ge 1, 1 \le j \le n\,, 
\end{align*}
 and we have invoked estimates \eqref{pc:W0}, \eqref{mako:6}, \eqref{W:boot:0}. Combining \eqref{born:2} - \eqref{backer:2}, we obtain the right-most estimate in \eqref{Fwc.est.ult}. 

\end{proof}

We now establish enhanced localized estimates for the bottom order derivatives. 
\begin{lemma}  The following estimates are valid 
\begin{align} \label{better:1}
\sup_{|x| \le \ell} |F_{W,7}^{c}| \le  {M \ell^{\frac 1 5} \eps^{\frac34}e^{\frac 3 4 s}}\,. 
\end{align}
\end{lemma}
\begin{proof} An inspection of the proof Lemma \ref{L:N:F} shows that only terms \eqref{truck:0} and \eqref{truck:-1} need to be estimated, with $n = 7$. Accordingly, we estimate 
\begin{align} \n
& \|  \sum_{j = 1}^7 \binom{7}{j} \beta_\tau W^{(1+j)} \p_c W^{(7-j)} \eta_{\frac{1}{20}} \|_\infty +  \| \sum_{j = 0}^{5} \binom{n}{j} \beta_\tau W^{(7-j)} \p_c W^{(j+1)} \eta_{\frac{1}{20}} \|_\infty \lesssim  \ell^{\frac 1 2} M \eps^{\frac34}e^{\frac 3 4 s}\,, 
\end{align}
upon invoking the localized bootstraps \eqref{e:Wtilde:bootstrap} and \eqref{warrior:1}. 
\end{proof}

\begin{lemma} \label{lemma:j:1} The following estimates are valid  
\begin{align} \label{Force:Z:pc:1}
\| \p_c F_Z \|_\infty &\le \eps^{\frac 3 4} e^{- \frac s 2}\,, &\| F^{c}_{Z,0} \|_\infty &\le \eps^{\frac 1 2} e^{- \frac s 2} \,,\\ \label{Force:Z:pc:n}
\| \p_c F_Z^{(n)} \|_\infty &\lesssim \eps^{\frac 3 4} e^{- \frac s 2}\,, &\| F^{c}_{Z,n} \|_\infty &\les {M^{2n^2-1} }\eps^{\frac 1 2} e^{- \frac s 2}\,,
\end{align}
for $1\leq n\leq 7$
\end{lemma}
\begin{proof} First, we use expression \eqref{pcFz:0} to estimate 
\begin{align} \n
\| \p_c F_Z \|_\infty &\lesssim  |\dot{\tau}_c| \| F_Z \|_\infty + e^{-s} \| A_c \|_\infty \Big( \| e^{- \frac s 4}W + \kappa \|_\infty + \| Z \|_\infty \Big) + e^{-s} \| A \|_\infty \Big( e^{- \frac s 4} \| W_c \|_\infty + | \kappa_c| + \| Z_c \|_\infty \Big) \\ \label{teq:1}
&\lesssim_M  \eps^{\frac 5 4} e^{-s} +  \eps^{\frac 1 2} e^{- \frac 3 2s} \Big(1 + \eps^{\frac 5 4} \Big) + \eps e^{-s}  \Big( \eps^{\frac 3 4} e^{\frac s 2} + \eps^{\frac 3 8} + \eps^{\frac 12} \Big) \lesssim_M \eps^{\frac 7 8} e^{- \frac s 2}\,. 
\end{align}
where above we have also invoked estimate \eqref{Z:0:order} for the $F_Z$ term together with the bootstrap estimates \eqref{W:boot:0}, \eqref{Z:boot:0}, \eqref{A:boot:9}, \eqref{mod:sub}, \eqref{Baba:O:1}, \eqref{mako:3} and \eqref{mako:6}. The first estimate in \eqref{Force:Z:pc:1} follows from \eqref{teq:1} upon bringing $\eps$ small relative to $M$. 
 
Next, we use the identity \eqref{every:time:1} to estimate 
\begin{align} \n
\| F^{c}_{Z,0} \|_\infty &\lesssim_M \| Z^{(1)} \|_\infty \Big( |\p_c \dot{\tau}| \| W \|_\infty + \| \p_c W \|_\infty + \| \p_c G_Z \|_\infty \Big) + \| \p_c F_Z \|_\infty \\ \n
&\lesssim_M  e^{- \frac 5 4 s} \Big( \eps^{\frac 1 2} e^{\frac s 4} +  \eps^{\frac 3 4} e^{\frac 3 4 s}  + \eps^{\frac 1 2} e^{\frac s 4} \Big) + \eps^{\frac 3 4} e^{- \frac s 2}  \lesssim_M \eps^{\frac 3 4 } e^{- \frac s 2}\,,  
\end{align}
from which the second estimate in \eqref{Force:Z:pc:1} follows again by bringing $\eps$ small relative to $M$. 

We now use expression \eqref{Midterms:2} to estimate the first quantity in \eqref{Force:Z:pc:n} via 
\begin{align*}
\| \p_c F_Z^{(n)} \|_\infty \lesssim &| \dot{\tau}_c | | \beta_\tau | \| F_Z^{(n)} \|_\infty + | \beta_\tau| e^{-s} \sum_{j = 0}^n \| A_c^{(j)} \Big( \beta_3 ( e^{- \frac s 4} W + \kappa)^{(n-j)} + \beta_4 Z^{(n-j)} \Big) \|_\infty \\
&+ | \beta_\tau | e^{-s} \sum_{j =0}^n  \| A^{(j)} \Big(  \beta_3 ( e^{- \frac s 4} W_c + \kappa_c)^{(n-j)} + \beta_4 Z_c^{(n-j)} \Big) \|_\infty \\
\lesssim & \eps^{\frac 3 2} e^{- \frac 5 4 s} + \eps^{\frac 1 4} e^{-s} + \eps e^{- \frac s 2}\,, 
\end{align*}
where above we have invoked the forcing estimate, \eqref{my:generation}. 

Next, in order to complete the estimate of the quantity $\| F_{Z,n}^{c} \|_\infty$, we need to estimate the remaining five terms in \eqref{Midterms:1}. The second, third, and sixth terms from the right-side of \eqref{Midterms:1} are estimated via 
\begin{align*}
\sum_{j = 0}^n |\dot{\tau}_c| \| Z^{(j+1)} \|_\infty \| W^{(n-j)} \|_\infty &\lesssim \sum_{j = 0}^n M^{2(j+1)^2} \eps^{\frac 12} e^{- \frac 5 4 s} M^{(n-j)^2} e^{\frac s 4}\,, \\
\sum_{j = 0}^n \| Z^{(j+1)} \|_\infty \| W_c^{(n-j)} \|_\infty &\lesssim \sum_{j = 0}^n M^{2(j+1)^2} e^{- \frac 5 4 s}  M^{(n-j)^2} \eps^{\frac34}e^{\frac 3 4 s}\,, \\
\sum_{j = 2}^n \| W^{(j)} \|_\infty \| Z_c^{(n-j+1)} \|_\infty &\lesssim M^{j^2} \eps^{\frac 1 2} M^{2 (n-j+1)^2} e^{- \frac s 2} \lesssim M^{-1 + 2n^2} \eps^{\frac 1 2} e^{- \frac s 2}\,.
\end{align*}
Above, we have invoked \eqref{mako:1}, \eqref{mako:3}, and \eqref{mako:6}. 

The fourth and fifth terms from the right-side of \eqref{Midterms:1} are estimated via 
\begin{align*}
\sum_{j = 0}^n \| Z^{(j+1)} \|_\infty \| \p_c G_Z^{(n-j)} \|_\infty & \lesssim \sum_{j = 0}^n M^{2(j+1)^2}  \eps^{\frac 1 2} e^{-s} \\
\sum_{j = 1}^n \| G_Z^{(j)} \|_\infty \| Z_c^{(n+1-j)} \|_\infty &\lesssim M^{2j^2}  M^{(n+1-j)^2} \eps e^{- \frac 3 4 s},
\end{align*}
where we have invoked the estimates on $G_Z$ and $\p_c G_Z$ from \eqref{Z:transport:1} and \eqref{r:est:2}. Above we have also used the elementary inequality
\begin{align} \label{elem:2}
j^2 + 2 (n+1-j)^2 \le -1 + 2 n^2 \quad\text{ for } n \ge 2 \mbox{ and } 2 \le j \le n\,. 
\end{align}

\end{proof}

\begin{lemma} The following estimates are valid  
\begin{align} \label{Force:A:pc:1}
\| \p_c F_A  \|_\infty &\le \eps^{\frac 1 2} e^{- \frac s 2}\,, &&\| F^{c}_{A,0}  \|_\infty \le \eps^{\frac 1 2 } e^{- \frac s2}\,, \\ \label{Force:A:pc:n}
\| \p_c F_A^{(n)}  \|_\infty &\le \eps^{\frac 12} e^{- \frac s 2} ,&&\| F^{c}_{A,n} \|_\infty {\les  M^{2n^2-1} \eps^{\frac 1 2} e^{- \frac s 2} }\,,
\end{align}
for $ 1 \le n \le 7$
\end{lemma}
\begin{proof} We appeal to the expression \eqref{pcFa} to estimate 
\begin{align} \n
\| \p_c F_A  \|_\infty &\lesssim  |\dot{\tau}| \| F_A  \|_\infty + e^{-s} \Big( \| e^{- \frac s 4}W  + \kappa \|_\infty + \|Z \|_\infty \Big) \Big( \| e^{- \frac s 4} W_c + \kappa_c \|_\infty + \|Z_c \|_\infty \Big) \\
&\lesssim M^{\frac 1 2}   \eps^{\frac 1 6} e^{- \frac 7 4 s}  + e^{-s} (  M^4 \eps^{\frac34}e^{\frac 1 2 s}+  \eps^{\frac 1 2}  + \eps^{\frac 1 2})\,.\n
\end{align}  
Above, we have invoked \eqref{e:GW0}, \eqref{est:W:good}, \eqref{Z:boot:0}, \eqref{Baba:O:1}, \eqref{mako:1}, \eqref{pc:W0} and finally \eqref{A:0:order} for the $F_A$ contribution. 

Next, we appeal to the expression \eqref{socialite:1} to estimate 
\begin{align} \n
\| F_{A,0}^{c}\circ \Phi_A^{x_0} \|_\infty \lesssim & \| \p_c F_A \circ \Phi_A^{x_0}\|_\infty + \|A^{(1)} \|_\infty \Big( |\dot{\tau}_c| \|W\|_\infty + \| W_c \|_\infty + \| \p_c G_A \|_\infty  \Big) \\
\lesssim & \eps^{\frac 12} e^{- \frac s 2} + M^2 e^{- \frac 5 4 s} \Big( \eps^{\frac 1 2} e^{\frac s 4} +M^4 \eps^{\frac34}e^{\frac 3 4 s} + \eps^{\frac 1 4} e^{\frac s 4} \Big)\,,\n
\end{align}
where we have appealed to estimates \eqref{Force:A:pc:1}, as well as bootstrap assumptions \eqref{A:boot:9}, \eqref{mako:1}, \eqref{est:W:good}, \eqref{pc:W0}, and \eqref{r:est:3} for the $\p_c G_A$ contribution. 

Next, we appeal to the expression of \eqref{Midterms:4} to estimate 
\begin{align} \n
\| \p_c F_A^{(n)} \|_\infty &\lesssim  |\dot{\tau}_c| \| F_A^{(n)} \|_\infty + e^{-s} \sum_{j = 0}^n (\| (e^{- \frac s 4} W^{(j)}+\kappa)^{(j)}\|_{\infty} + \| Z^{(j)} \|_\infty ) \times \\ \n
& \qquad (\|(e^{- \frac s 4}  W_c  +\kappa_c)^{(n-j)}\|_\infty + \| Z_c^{(n-j)} \|_\infty) \\ \n
&\lesssim_M \eps^{\frac 1 2} e^{-s} + e^{-s} (  1 + \eps^{\frac 12} + \eps^{\frac 5 4} ) ( \eps^{\frac34}e^{\frac s 2 }+ \eps^{\frac 1 2} e^{- \frac s 2} ) \lesssim_M  \eps^{\frac 3 4} e^{- \frac s 2}\,,  
\end{align}
where we have invoked estimates \eqref{mako:1}, \eqref{est:W:good}, \eqref{weds:1}, \eqref{e:GW0} - \eqref{mod:sub}, as well as \eqref{mako:3}. 

The final estimate in \eqref{Force:A:pc:n} requires an estimate of the remaining terms in \eqref{Midterms:3}, which is identical to that of Lemma \ref{lemma:j:1}. 

\end{proof}

\subsubsection{$\nabla_{a,b}^2$ forcing estimates}

\begin{lemma} For $1 \le n \le 6$, the following estimates are valid
\begin{align} \label{disclosure:1}
\| \p_{c_1 c_2} F_W \|_\infty &\le \eps^{\frac 1 2} e^{\frac s 2}\,, && \| F_{W,0}^{c_1, c_2} \|_\infty \le M^{14} \eps^{\frac32}e^{\frac 3 2 s}\,, \\ \label{disclosure:2}
\| \p_{c_1 c_2} F_W^{(n)} \|_\infty &\les \eps^{\frac 1 2} e^{\frac s 2}\,,  && \| F_{W,n}^{c_1, c_2} \|_\infty \le M^{(n+5)^2 - 1} \eps^{\frac32}e^{\frac 3 2 s}\,.
\end{align}
\end{lemma}
\begin{proof}[Proof of \eqref{disclosure:1}] For the computation of $\p_{c_1 c_2} F_W$, we recall the definition of \eqref{find:1}, and proceed to estimate systematically 
\begin{align} \n
&\| \beta_\tau e^{- \frac 3 4 s}  A_{c_1 c_2} ( \beta_3 Z + \beta_4 (e^{- \frac s 4}W + \kappa) ) \|_\infty \\
& \qquad \lesssim e^{- \frac 3 4 s} \| A_{c_1 c_2 } \|_\infty (\| Z \|_\infty + \| e^{- \frac s 4}W + \kappa \|_\infty) \lesssim_M \eps^{\frac 5 8} e^{- \frac s 2}  (\eps^{\frac 5 4} + 1) \lesssim_M \eps^{\frac 5 8} e^{- \frac s 2}\,,\n
\end{align}
and next 
\begin{align} \n
&\| \beta_\tau e^{- \frac 3 4 s} A_{c_1} ( \beta_3 Z_{c_2} + \beta_4 (e^{- \frac s4} W_{c_2} + \kappa_{c_2}) ) \|_\infty \\
& \qquad \lesssim e^{- \frac 3 4 s} \| A_{c} \|_\infty (\| Z_{c} \|_\infty + e^{- \frac s 4} \| W_c \|_\infty + | \kappa_c |) \lesssim_M  \eps^{\frac 1 2} e^{- \frac 3 4 s} ( \eps^{\frac 1 2} +  \eps^{\frac 3 4} e^{\frac s 2} + \eps^{\frac 1 4} ) \lesssim_M \eps^{\frac 3 4} e^{- \frac s 4}\,.\n
\end{align}
Above, we have invoked bootstrap assumptions \eqref{posty:1} as well as \eqref{mako:2} - \eqref{mako:3}, and \eqref{Z:boot:0} - \eqref{A:boot:9}. 

The first term on the second line of \eqref{find:1} is estimated in an identical fashion, while the second term is estimated via 
\begin{align} \n
&\| \beta_\tau e^{- \frac 3 4 s} A (\beta_3 Z_{c_1 c_2} + \beta_4 (e^{- \frac s 4}W_{c_1 c_2} + \kappa_{c_1 c_2})) \|_\infty \\ \n
& \qquad \lesssim e^{- \frac 3 4 s} \| A \|_\infty (\| Z_{c_1 c_2} \|_\infty + e^{- \frac s 4} \| W_{c_1 c_2} \|_\infty + |\kappa_{c_1 c_2}|) \\
& \qquad \lesssim_M \eps e^{- \frac 3 4 s}  ( \eps^{\frac 5 8} e^{\frac s 4} + \eps^{\frac 3 2} e^{\frac 5 4 s}  +  \eps^{\frac 5 4} e^s) \lesssim_M \eps e^{\frac s 2}\,,\n 
\end{align}
where again we have invoked bootstrap assumptions \eqref{posty:1} as well as \eqref{mako:2} - \eqref{mako:3}, and \eqref{Z:boot:0} - \eqref{A:boot:9}. 

Finally, the last line of \eqref{find:1} is estimated via 
\begin{align} \n
&\| \dot{\tau}_{c_2} \beta_\tau \p_{c_1} F_W + \dot{\tau}_{c_1 c_2} \beta_\tau F_W + \dot{\tau}_{c_1} \beta_\tau \p_{c_2} F_W \|_\infty \\
& \qquad \lesssim |\dot{\tau}_c| \| \p_c F_W \|_\infty + |\dot{\tau}_{c_1c_2}| \| F_W \|_\infty \lesssim_M  \eps^{\frac 5 4} e^{- \frac s 4} + \eps^{\frac 3 2} \lesssim_M \eps\,,\n
\end{align}
where we have invoked the estimates \eqref{FW:est:1} and \eqref{Fwc.est.ult:0}. 

Next, to estimate the remaining quantity in \eqref{disclosure:1}, we recall definition \eqref{HAIM:1}, according to which we define the following two auxiliary quantities:  
\begin{align*}
&\mathcal{L}_1 := \beta_\tau W^{(1)}_{c_2} W_{c_1} - \beta_\tau^2 \dot{\tau}_{c_2} W^{(1)} W_{c_1} - \Big( \beta_\tau^2 \dot{\tau}_{c_2} W + \beta_\tau W_{c_2} + \p_{c_2} G_W \Big) W^{(1)}_{c_1} \\ 
&\mathcal{L}_2 := - \dot{\tau}_{c_1} \beta_\tau^2 W W^{(1)}_{c_2} - \dot{\tau}_{c_1 c_2} \beta_\tau^2 W W^{(1)} - 2 \beta_\tau^2 \dot{\tau}_{c_1} \dot{\tau}_{c_2} W W^{(1)} - \dot{\tau}_{c_1} \beta_\tau^2 W^{(1)} W_{c_2},
\end{align*}
so that we have the identity 
\begin{align*}
F_{W,0}^{c_1, c_2} = \p_{c_1 c_2} F_W + \mathcal{L}_1 + \mathcal{L}_2 - \mathcal{M}^{c_1, c_2}\,, 
\end{align*}
where $\mathcal{M}^{c_1, c_2}$ has been defined in \eqref{Mod:Mod:1}. 

We first estimate $\mathcal{L}_1$ via 
\begin{align} \n
\| \mathcal{L}_1 \|_\infty \lesssim & (1 + |\dot{\tau}_c|) \| W^{(1)}_c \|_\infty \| W_c \|_\infty + |\dot{\tau}_c| \| W \|_\infty \|W^{(1)}_c \|_\infty + \| \p_c G_W \eta_{- \frac{1}{20}} \|_\infty \| W^{(1)}_c \eta_{\frac{1}{20}} \|_\infty \\ \n
\lesssim & (1 + \eps^{\frac 1 2}) M^{13} \eps^{\frac32}e^{\frac 3 2 s} + \eps^{\frac 1 2} M^4 \eps^{\frac32}e^{\frac 3 2 s} + {M^{12}} \eps^{\frac {41} {20}}e^{\frac 3 2 s} {+ M^9 \eps^{\frac 54} e^{\frac 3 4s} }\\
\lesssim & M^{13} \eps^{\frac32}e^{\frac 3 2 s}\,.\n
\end{align}
Note that for the estimation of the final term above, we have used crucially the spatial decay of $W_c^{(1)}$, as guaranteed by the bootstrap assumption \eqref{mako:6}, and we have also applied estimate \eqref{r:est:1} with $r = \frac{1}{5}$. 

Next, we estimate $\mathcal{L}_2$ via 
\begin{align} \n
\| \mathcal{L}_2 \|_\infty &\lesssim  |\dot{\tau}_c| \| W \|_\infty \| W^{(1)}_c \|_\infty +( |\dot{\tau}_{c_1 c_2} | +  |\dot{\tau}_c|^2 ) \| W \eta_{-\frac{1}{20}} \|_\infty \| W^{(1)} \eta_{\frac{1}{20}} \|_\infty + |\dot{\tau}_c| \| W^{(1)} \|_\infty \| W_c \|_\infty \\
&\lesssim_M  \eps^{\frac 1 2} e^{\frac s 4} + (\eps e^{\frac 3 4 s} + \eps) + \eps^{\frac 5 4} e^{\frac 3 4 s} \lesssim_M \eps e^{\frac 34 s}\,,\n
\end{align}
where we invoke the bootstrap assumptions \eqref{W:boot:0} - \eqref{e:uniform:W1}, \eqref{mako:1} - \eqref{Baba:O:1}, \eqref{mako:6}, and \eqref{group:1}. 

Next, we estimate $\mathcal{M}^{c_1, c_2}$ via 
\begin{align} \n
|\mathcal{M}^{c_1, c_2}| &\lesssim  e^{- \frac 3 4 s}\Big( | \dot{\kappa}_{c_1 c_2}| + |\dot{\kappa}_{c}| |\dot{\tau}_{c}|  + |\dot{\kappa}| |\dot{\tau}_{c_1 c_2}|  +  |\dot{\kappa}| | \dot{\tau}_{c}|^2 \Big) \\
&\lesssim_M e^{- \frac 3 4 s} \Big( \eps^{\frac 54} e^s +  \eps^{\frac 34} + \eps^{\frac 9 8} e^{\frac 3 4 s}  + \eps^{\frac 5 8}\Big) \lesssim_M \eps^{\frac 5 8} e^{\frac s 4}\,,\n
\end{align}
where we have invoked the bootstrap assumptions on the second (parameter) derivatives of the modulation variables, \eqref{group:1} - \eqref{group:2}. 
\end{proof}
\begin{proof}[Proof of \eqref{disclosure:2}] We now move to the $1 \le n \le 6$ estimates, for which we first recall the expression of $\p_{c_1 c_2} F_W^{(n)}$ from \eqref{Fwn:c1:c2:exp}. The estimate of this is identical to the estimate of $\p_{c_1 c_2}F_W$ (the $n = 0$ case above), and so we omit it. We now proceed to estimate all of the remaining terms in \eqref{Bernie:1}. 
\begin{align*}
\| \sum_{j = 1}^n \beta_\tau^2 \dot{\tau}_{c_i} W^{(1+j)} W^{(n-j)}_{c_{i'}} \|_\infty &\lesssim \sum_{j = 1}^n |\dot{\tau}_c| \| W^{(1+j)} \|_\infty \| W^{(n-j)}_c \|_\infty \lesssim_M \eps^{\frac 5 4} e^{\frac 3 4 s} \,,\\
\| \sum_{j = 0}^n \binom{n}{j} \beta_\tau W_{c_1}^{(j)} W_{c_2}^{(n+1-j)} \|_\infty &\lesssim \sum_{j = 0}^n \|W_{c}^{(j)}\|_\infty  \| W_c^{(n+1-j)} \|_\infty \lesssim \sum_{j = 0}^n M^{(j+2)^2} M^{(n-j+3)^2} \eps^{\frac 3 2} e^{\frac 3 2 s}.
\end{align*}
We have invoked the bootstrap assumptions \eqref{weds:1} on derivatives of $W$, \eqref{mako:1}, as well as \eqref{mako:6}. 
We now appeal to the elementary inequality 
\begin{align*}
(j+1)^2 + (n-j+3)^2 \le (n+5)^2 - 1 \text{ for } 0 \le j \le n, \qquad n \ge 1\,. 
\end{align*}

We continue with 
\begin{align*}
\| \sum_{j =1}^n \binom{n}{j} \beta_\tau W^{(1+j)} W^{(n-j)}_{c_1 c_2} \|_\infty \lesssim \sum_{j = 1}^n \| W^{(1+j)} \|_\infty \|W_{c_1 c_2}^{(n-j)} \|_\infty \lesssim \sum_{j = 1}^n M^{(1+j)^2} M^{(n-j+5)^2} \eps^{\frac 3 2} e^{\frac 3 2 s}\,,
\end{align*}
and again appeal to an elementary inequality 
\begin{align*}
(1 + j)^2 + (n-j+5)^2 \le -1 + (n+5)^2 \text{ for } 1 \le j \le n, \qquad n \ge 1\,. 
\end{align*}

The fifth term on the right-hand side of \eqref{Bernie:1} is formally the same as the second term, with the exception of the $j = 0$ case, which we estimate via 
\begin{align*}
\| \beta_\tau^2 \dot{\tau}_{c_{i'}} W W_{c_i}^{(n+1)} \|_\infty \lesssim_M \eps^{\frac 1 2} e^{\frac s 4} e^{\frac 3 4 (s - s_0)} \lesssim_M {\eps^{\frac 5 4} e^s}\,. 
\end{align*}

We now move to the term 
\begin{align*}
\| \sum_{j = 0}^n \binom{n}{j} \p_{c_2} G_W^{(j)} W_{c_1}^{(n+1-j)} \|_\infty &\lesssim  \sum_{j = 0}^n \| \p_{c_2} G_W^{(j)} \eta_{- \frac{1}{20}} \|_\infty \| W_{c_1}^{(n+1-j)} \eta_{\frac{1}{20}} \|_\infty\\ 
 &\lesssim_M  \eps^{\frac {33} {20} } e^{\frac 3 2 s} + {\eps^{\frac 5 4} e^{\frac 3 4s}}\,.
\end{align*}
Above we have invoked \eqref{r:est:1} with $r = \frac 1 5$. 

We now move to the final three terms, which are easily estimated via  
\begin{align*}
\| \sum_{j = 1}^n G_W^{(j)} W_{c_1 c_2}^{(n-j+1)} \|_\infty &\lesssim_M \eps^{\frac 3 2} e^{\frac s 2}, \\
\| \sum_{j = 0}^n \beta_\tau^2 (\dot{\tau}_{c_1 c_2} + 2 \dot{\tau}_{c_1} \dot{\tau}_{c_2}) W^{(j)} W^{(n+1-j)} \|_\infty &\lesssim_M \eps  e^{\frac 3 4 s} + \eps,
\end{align*}
where we have invoked \eqref{weds:1} for derivatives of $W$, \eqref{GW:transport:est} for $j \ge 1$ for the $G_W$ contribution, and \eqref{mako:1}, \eqref{group:1} for $\p_c$ and $\p_{c}^2$ of $\dot{\tau}$.  


\end{proof}

\begin{lemma} \label{Lemma:ka} For $1 \le n \le 6$, the following estimates are valid 
\begin{align} \label{jay:1}
\| \p_{c_1 c_2} F_Z \|_\infty &\le \eps e^{\frac s 4}\,, && \| F_{Z,0}^{c_1, c_2} \|_\infty \le \eps e^{\frac 1 4 s}\,, \\ \label{jay:2}
\| \p_{c_1 c_2} F_Z^{(n)} \|_\infty &\le \eps e^{\frac s 4}\,, && \| F_{Z,n}^{c_1, c_2} \|_\infty \les M^{2n^2 -1}  \eps^{\frac 5 8} e^{\frac 1 4 s}\,.
\end{align}
\end{lemma}
\begin{proof} First, we turn to the estimation of $\p_{c_1 c_2} F_Z^{(n)}$, for which we appeal to the expression given in \eqref{Lauv:4} and estimate term by term via
\begin{align}  \n
&\| \beta_\tau e^{-s} \sum_{j = 0}^n \binom{n}{j} A^{(j)} \Big( \beta_3( e^{- \frac s 4} W_{c_1 c_2} + \kappa_{c_1 c_2})^{(n-j)} + \beta_4 Z_{c_1 c_2}^{(n-j)} \Big) \|_\infty \\ \n
&\qquad \lesssim  e^{-s} \sum_{j = 0}^n \| A^{(j)} \|_\infty \Big(  \| e^{- \frac s 4}W_{c_1 c_2}+ \kappa_{c_1 c_2})^{(n-j)}\|_{\infty} + \| Z^{(n-j)}_{c_1 c_2}\|_\infty \Big) \\ 
&\qquad \lesssim_M \eps e^{-s}  \Big( \eps^{\frac 3 2} e^{\frac 5 4 s}+ \eps^{\frac 5 4} e^s + \eps^{\frac 5 8} e^{\frac s 4}  \Big) \lesssim_M \eps^{\frac 5 4} e^{\frac s 4}\,.\n
\end{align}
Above, we have invoked estimates \eqref{A:boot:9}, \eqref{posty:1}, \eqref{buddy:1}, and \eqref{group:2}.  

Next, the second term from \eqref{Lauv:4} is estimated via 
\begin{align} \n
 &\|  \beta_\tau e^{-s} \sum_{j = 0}^n \binom{n}{j} A_{c_1 c_2}^{(j)} \Big( \beta_3 (e^{- \frac s 4} W + \kappa)^{(n-j)} + \beta_4 Z^{(n-j)} \Big) \|_\infty \\ \n
 & \qquad \lesssim e^{-s} \sum_{j = 0}^n \| A^{(j)}_{c_1 c_2} \|_\infty \Big( e^{- \frac s 4} \|W^{(n-j)} \|_\infty + |\kappa| + \| Z^{(n-j)} \|_\infty \Big) \\ 
 & \qquad \lesssim_M \eps^{\frac 5 8} e^{- \frac 3 4 s} ( 1+ \eps^{\frac 5 4} )\,,\n
\end{align}
where we have invoked \eqref{weds:1}, \eqref{est:W:good}, \eqref{Z:boot:0} \eqref{mod:sub}, and  \eqref{posty:2}.

Next, the third term from \eqref{Lauv:4} is estimated via 
\begin{align} \n
& \| \beta_\tau e^{-s} \sum_{j = 0}^n \sum_{i \in \{1, 2\}} \binom{n}{j} A_{c_i}^{(j)} \Big( \beta_3( e^{- \frac s 4} W_{c_{i'}} + \kappa_{c_{i'}})^{(n-j)} + \beta_4 Z_{c_{i'}}^{(n-j)}  \Big) \|_\infty \\ \n
& \qquad \lesssim e^{-s} \sum_{j = 0}^n \| A_c^{(j)} \|_\infty \Big( e^{- \frac s 4} \| W_c^{(n-j)} \|_\infty + |\kappa_c| + \| Z_c^{(n-j)} \|_\infty  \Big) \\
& \qquad \lesssim_M  \eps^{\frac 1 2}  e^{-s}\Big( \eps^{\frac 3 4} e^{\frac s 2}+ \eps^{\frac 1 4} + \eps^{\frac 1 2} \Big) \lesssim_M \eps e^{- \frac s 2}\,,\n
\end{align}
where we have invoked \eqref{Baba:O:1}, \eqref{mako:2} - \eqref{mako:3}, and \eqref{mako:6}

 We now move to the final terms from \eqref{Lauv:4} which evaluate to 
\begin{align} \n
&\| \dot{\tau}_{c_1} \beta_\tau \p_{c_2} F_Z^{(n)}  + \dot{\tau}_{c_2} \beta_\tau \p_{c_1} F_Z^{(n)}  + \dot{\tau}_{c_1 c_2} \beta_\tau F_Z^{(n)} \|_\infty \\
& \qquad  \lesssim |\dot{\tau}_c| |\p_c F_Z^{(n)}| + |\dot{\tau}_{c_1 c_2}| \| F_Z^{(n)} \|_\infty \lesssim  \eps^{\frac 5 4} e^{- \frac s 2} + \eps^2 e^{- \frac s 4}\,, \n
\end{align}
where we have invoked the estimates \eqref{Force:Z:pc:1} - \eqref{Force:Z:pc:n}, as well as estimates \eqref{Z:0:order} and \eqref{my:generation}. 

We now turn to equation \eqref{Lauv:2} for the form of $F_{Z,n}^{c_1, c_2}$. We will estimate term by term, starting with 
\begin{align*}
\| 1_{n \ge 2} \sum_{j = 2}^n \binom{n}{j} \beta_\tau \beta_2 W^{(j)} Z_{c_1 c_2}^{(n-j+1)} \|_\infty &\lesssim 1_{n \ge 2} M^{j^2} M^{2(n-j+1)^2} \eps^{\frac 5 8} e^{\frac s 4} \lesssim M^{-1 + 2n^2} \eps^{\frac 5 8} e^{\frac s 4}\,, \\
\| 1_{n \ge 1} \sum_{j = 1}^n \binom{n}{j} G_Z^{(j)} Z_{c_1 c_2}^{(n-j+1)} \|_\infty& \lesssim_M  \eps^{\frac 5 8} e^{- \frac 3 4 s}\,, 
\end{align*}
where for the first estimate above we have invoked the elementary inequality \eqref{elem:2}, and for the second estimate we have invoked \eqref{Z:transport:1}. 

Next, we continue by estimating 
\begin{align} \n
&\| \sum_{j = 0}^n \sum_{i \in \{1, 2 \}} \binom{n}{j} Z_{c_i}^{(j+1)} \Big( \dot{\tau}_{c_{i'}} \beta_\tau^2 \beta_2 W^{(n-j)} + \beta_\tau \beta_2 W_{c_{i'}}^{(n-j)} + \p_{c_{i'}}G_Z^{(n-j)} \Big) \|_\infty \\ \n
& \qquad \lesssim \sum_{j = 0}^n \| Z^{(j+1)}_c \|_\infty \Big(|\dot{\tau}_c| \| W^{(n-j)} \|_\infty + \| W_c^{(n-j)} \|_\infty + \| \p_c G_Z^{(n-j)} \|_\infty \Big) \\ 
& \qquad \lesssim_M \eps^{\frac 1 2} e^{- \frac s 2} \Big( \eps^{\frac 1 2} e^{\frac s 4} + \eps^{\frac 3 4} e^{\frac 3 4} + \eps^{\frac 1 4} e^{\frac s 4} \Big) \lesssim_M \eps^{\frac 5 4} e^{\frac s 4},\n
\end{align}
where we have invoked the bootstrap assumptions \eqref{weds:1}, \eqref{mako:1}, \eqref{mako:2} - \eqref{mako:3}, \eqref{mako:6}, as well as \eqref{r:est:2} on the $\p_c G_Z$ term. 

We return to \eqref{Lauv:2}, and address the third and fourth lines by estimating 
\begin{align} \n
& \|  \sum_{j = 0}^n \binom{n}{j} Z^{(j+1)} \Big( \dot{\tau}_{c_1 c_2} \beta_\tau^2 \beta_2 W^{(n-j)} + 2 \dot{\tau}_{c_1} \dot{\tau}_{c_2} \beta_\tau^2 \beta_2 W^{(n-j)} + \beta_\tau \beta_2 W_{c_1 c_2}^{(n-j)} \\ \n
& \qquad \qquad \qquad \qquad + \sum_{i \in \{1, 2\}} \dot{\tau}_{c_i} \beta_\tau^2 \beta_2 W_{c_{i'}}^{(n-j)} + \p_{c_1 c_2}G_Z^{(n-j)} \Big) \|_{\infty} \\ \n
& \lesssim \sum_{j = 0}^n \| Z^{(j+1)} \|_\infty  \Big( |\dot{\tau}_{c_1 c_2}| \| W^{(n-j)} \|_\infty + |\dot{\tau}_c|^2 \| W^{(n-j)} \|_\infty + \| W_{c_1 c_2}^{(n-j)} \|_\infty \\ \n
& \qquad \qquad \qquad \qquad + |\dot{\tau}_c| \| W_c^{(n-j)} \|_\infty + \| \p_{c_1 c_2} G_Z^{(n-j)} \|_\infty \Big) \\
& \lesssim_M e^{- \frac 5 4 s} \Big( \eps e^s + \eps e^{\frac s 4} + \eps^{\frac 3 2} e^{\frac 3 2 s} + \eps^{\frac 5 4} e^{\frac 3 4 s} + \eps^{\frac 5 4} e^{\frac 5 4 s}  \Big) \lesssim_M \eps^{\frac 5 4}  e^{\frac s 4}\,.\n
\end{align}
Above, we have invoked \eqref{weds:1}, \eqref{Z:boot:0}, \eqref{mako:1}, \eqref{mako:6}, \eqref{buddy:1}, \eqref{group:1}, as well as \eqref{spin:3} for the $\p_{c_1 c_2} G_Z$ contribution. 

This concludes the treatment of the terms from \eqref{Lauv:2} and hence the proof of the lemma. 

\end{proof}

\begin{lemma} For $1 \le n \le 6$, the following estimates are valid 
\begin{align} \label{found:me:1}
\| \p_{c_1 c_2} F_A  \|_\infty &\le \eps e^{\frac s 4}\,, && \| F_{A,0}^{c_1, c_2} \|_\infty \le \eps e^{\frac 1 4 s}\,, \\ \label{found:me:2}
\| \p_{c_1 c_2} F_A^{(n)} \|_\infty &\le \eps e^{\frac s 4}\,, && \| F_{A,n}^{c_1, c_2} \|_\infty \les M^{2n^2 -1} \eps^{\frac 5 8} e^{\frac 1 4 s}\,.
\end{align}
\end{lemma}
\begin{proof} First, we use expression \eqref{Lauv:7} to produce the estimates 
\begin{align} \n
\|  \beta_\tau \beta_1 e^{-s} \Big( e^{- \frac s 4} W_{c_2} + \kappa_{c_2} + Z_{c_2} \Big)\Big( e^{- \frac s 4} W_{c_1} + \kappa_{c_1} + Z_{c_1} \Big) \|_\infty &\le \eps^{\frac 3 2} e^{\frac s 4}\,, \\
\| \beta_\tau \beta_1 e^{-s}\Big( e^{- \frac s 4} W + \kappa + Z \Big) \Big( e^{- \frac s 4} W_{c_1 c_2} + \kappa_{c_1 c_2} + Z_{c_1 c_2} \Big) \|_{\infty} &\le \eps^{\frac 5 4}e^{\frac s 4} , \n
\end{align}
where we have invoked estimates \eqref{est:W:good}, \eqref{Baba:O:1}, \eqref{mako:2}, \eqref{pc:W0}, \eqref{posty:1}, \eqref{buddy:1}, and \eqref{group:1}. 

For the last line from expression \eqref{Lauv:7}, we have 
\begin{align} \n
&\| \dot{\tau}_{c_1 c_2} \beta_\tau F_A + \dot{\tau}_{c_2} \beta_\tau \p_{c_1}F_A + \dot{\tau}_{c_1} \beta_\tau \p_{c_2}F_A \| \lesssim  |\dot{\tau}_{c_1 c_2}| \| F_A \|_\infty + |\dot{\tau}_c| \| \p_c F_A \|_\infty \lesssim M \eps e^{- \frac s 4} + \eps e^{- \frac s 2} \,, \n
\end{align}
where we have invoked the forcing estimates \eqref{myA:generation} and \eqref{Force:A:pc:n}. This contribution is clearly bounded by $\eps^{\frac 3 4} e^{- \frac s 4}$ by bringing $\eps$ small relative to $M$. 

Next, we move to the second estimate in \eqref{found:me:1}, for which we appeal to the expression \eqref{Lauv:5}. However, these estimates are exactly analogous to those of Lemma \ref{Lemma:ka}, estimate \eqref{jay:1}, and so we omit repeating these estimates.  The estimates for general $n$, \eqref{found:me:2} also follow analogously to Lemma \ref{Lemma:ka}. 
%

\end{proof}

\subsection{Trajectory estimates} \label{subsect:trajectory}

In this subsection, we provide estimates on the trajectories associated with the transport structure of the equations \eqref{W:0} - \eqref{A:0}. We now define these trajectories via 
\begin{align*}
&\p_s \Phi^{x_0}_W(s) = \mathcal{V}_W \circ \Phi_{W}^{x_0}\,,  & \Phi^{x_0}_{W}(s_0) &= x_0\,, \\
&\p_s \Phi^{x_0}_Z(s) = \mathcal{V}_Z \circ \Phi_{Z}^{x_0}\,, & \Phi^{x_0}_{Z}(s_0) &= x_0\,, \\
&\p_s \Phi^{x_0}_A(s) = \mathcal{V}_A \circ \Phi_{A}^{x_0}\,, & \Phi^{x_0}_{A}(s_0) &= x_0\,. 
\end{align*}

\begin{lemma} \label{l:support} Let $\Phi(s)$ denote either $\Phi^{x_0}_W$, $\Phi^{x_0}_Z(s)$ or $\Phi^{x_0}_A$, then for $\abs{x_0}\leq \frac M 2 \eps^{- \frac 1 4}$ we have
\begin{equation}
\label{upp:Z:t}
|\Phi^{x_0}(s) | \leq \frac{2M}{3}e^{\frac 5 4 s}\,.
\end{equation}
As a consequence we obtain
\begin{align} \label{assume:cpct:supp:ZA}
\supp W(s)\cup\supp Z(s) \cup \supp A(s) \subset B\left(\frac{3}{4} M \eps e^{\frac 5 4 s}\right)\,,
\end{align}
which verifies the bootstrap assumption \eqref{e:support}
\end{lemma}
\begin{proof}
We restrict to the case $\Phi=\Phi^{x_0}_W$. The cases  $\Phi=\Phi^{x_0}_Z$ and  $\Phi=\Phi^{x_0}_A$ will follow in an analogous fashion. Recall that for  $\Phi=\Phi^{x_0}_W$ we have
\begin{equation*}
\p_s \Phi =   \frac 5 4 \Phi+ \beta_\tau W\circ \Phi+ G_W\circ \Phi\,.
\end{equation*}
As a consequence of \eqref{e:support}, \eqref{W:boot:0} and \eqref{GW:transport:est}, we have
\begin{equation}\label{e:pelican}
\norm{W}_{\infty}+\norm{G_W}_{\infty}\les M^{\frac15}\eps^{\frac15} e^{\frac s4}+e^{\frac s4}\les  e^{\frac s4}\,.
\end{equation}
Thus by Gr\"onwall we obtain we obtain \eqref{upp:Z:t}.

The support bound \eqref{assume:cpct:supp:ZA} follows directly from \eqref{blue:grass}, the defining equations  \eqref{W:0}-\eqref{A:0}, together with \eqref{upp:Z:t}.
\end{proof}

\begin{lemma} Let $\Phi(s)$ denote either $\Phi^{x_0}_Z(s)$ or $\Phi^{x_0}_A$, then for $\abs{x_0}\leq  \frac M 2 \eps^{- \frac 1 4}$ we have
\begin{align}  \label{traj:large:z}
|\Phi^{x_0}(s)| &\ge \text{min} (e^{\frac s 4}, e^{\frac s 4} - e^{\frac{s_\ast}{4}}) \text{ for some } s_\ast \ge s_0\,.  
\end{align}
\end{lemma}
\begin{proof} 
We first show that if $\Phi (s) \leq e^{\frac s 4}$, then we have the inequality
\begin{equation}\label{e:PhiZA:left}
\frac{\p}{\p s} \Phi(s) \le - e^{\frac s 4}\,.
\end{equation}
For notational purposes, we set $(j,GZ)=(2,G_Z)$ or $(j,GZ)=(1,G_A)$ for the cases $\Phi(s)=\Phi^{x_0}_Z(s)$ or $\Phi(s)=\Phi^{x_0}_A$, respectively. We then have the ODE
\begin{align} \n
\p_s \Phi =  & \frac 5 4 \Phi+ \beta_\tau\beta_j W\circ \Phi+ G\circ \Phi\,.
\end{align}
Note that since $\alpha>1$, then $\abs{\beta_j}<1$. Assuming $\eps$ to be sufficiently small (dependent on $\alpha$), then applying \eqref{e:1beta:bnd} yields $\beta_{\tau}\beta_{j}\leq 1$. Then if $\Phi (s) \leq e^{\frac s 4}$, we have from \eqref{W:boot:0}, \eqref{Z:transport:1} and \eqref{A:transport:1}
\begin{align}
\frac{\p}{\p s} \Phi(s) &\le \frac 5 4 e^{\frac s4}+2\eta_{\frac{1}{20}}\circ\Phi(s)-(1-\beta_j)\kappa_0e^{\frac s4}+\eps^{\frac 1 2}e^{\frac s4}\n\\
&\le \frac 5 4 e^{\frac s4}-(1-\beta_j)\kappa_0e^{\frac s4}+\eps^{\frac 1 2}e^{\frac s4}\,,
\notag 
\end{align}
where we used \eqref{traj:large:z}. Since $(1-\beta_j)>0$, then assuming $\kappa_0$ is sufficiently large, dependent of $\alpha$, we obtain \eqref{e:PhiZA:left}.

We now split the proof of \eqref{traj:large:z} into two subcases:
\begin{enumerate}
\item\label{case:Tokaji1} Either $\Phi(s)> e^{\frac s4}$ for all $s\in[s_0,\infty)$, or $x_{0}\leq 0$.
\item\label{case:Tokaji2} We have $x_0>0$ and there exists  a smallest $s_1\in[s_0,\infty)$ such that $0<\Phi(s_1)\leq e^{\frac {s_1}4}$.
\end{enumerate}
Consider first Case \ref{case:Tokaji1}. Note that $\Phi_1(s)> e^{\frac s4}$ directly implies \eqref{traj:large:z}. If $x_0\leq 0$, then  \eqref{e:PhiZA:left} implies that
$
\Phi(s)\leq -e^{\frac s 4}+\eps^{-\frac14} $ and hence \eqref{traj:large:z} is satisfied for $s_{*}=-\log\eps$.

Now consider Case \ref{case:Tokaji2}. The estimate  \eqref{e:PhiZA:left} implies that
$
\frac{d}{ds}\Phi(s)\leq  -e^{\frac s4}$ for all $s\geq s_0$.
Thus by continuity,  there exists a unique $s_*>s_0$ such that $\Phi(s_*)=0$. By continuity, there exists $s_*>s_1$ such that $\Phi(s_*)=0$. Then as a consequence of \eqref{traj:large:z}, by following trajectories forwards and backwards in time from $s_*$ we conclude that 
\[\abs{\Phi(s)}\geq \abs{e^{\frac s4}-e^{\frac{s_*}{4}}}\,,\]
for all $s\in[s_1,\infty)$. For the case $s_1\neq s_0$, then if $s\in [s_0,s_1)$ we have by definition that $\abs{\Phi(s)}\geq e^{\frac s4}$. Thus we have \eqref{e:PhiZA:left}.
\end{proof}

\begin{lemma}
For any $\abs{x_0}\geq \ell$ and $s_0\geq -\log\eps$ we have
\begin{align}\label{e:escape:from:LA}
\Phi_W^{x_0} \ge |x_0|\eps^{\frac15}  e^{\frac{s}{5}}\,.
\end{align}
\end{lemma}
\begin{proof} Using $W(0,s)=0$, \eqref{e:1beta:bnd}, \eqref{eq:W1:bnd:1}, \eqref{e:GW0}  and \eqref{GW:transport:est} we obtain
\begin{align*}
\mathcal V_W x&= \frac{5}{4}x^2+x\beta_{\tau} W+G_W x\\
&\geq x^2\left(\frac{5}{4}-\beta_{\tau}\norm{W^{(1)}}_{\infty}-\abs{G_W ^{(1)}}\right)-\abs{\mu}\\&\geq x^2\left(\frac{1}{4}-2e^{-\frac34 s}\right)-\eps^{\frac16}e^{-\frac34 s}\geq \frac15\,,
\end{align*}
where inequality we used that $\abs{x}\geq\ell\geq \eps^{\frac14}$ and $s_0$ is taken to be sufficiently large.

Thus we obtain
\begin{equation}
\frac{d}{ds} \left(\Phi_W^{x_0}\right)^2=2\mathcal V_W(\Phi_W^{x_0}) \Phi_W^{x_0}\geq \frac{2 (\Phi_W^{x_0})^2}{5}\label{e:escape:from:New:York}\,.
\end{equation}
and hence \eqref{e:escape:from:LA} follows by Gr\"onwall.
\end{proof}
\section{Analysis of modulation variables}
In this section we close all bootstraps related to the modulation variables $\kappa$, $\xi$ and $\tau$, together with the quantity $\mu$.

\subsection{Modulation variables and their time derivatives}

The following lemma verifies the bootstraps \eqref{mod:sub}.
\begin{lemma}  The following estimates are valid 
\begin{align} \label{mod:1}
|\kappa - \kappa_0| \lesssim \eps^{\frac 9 8}, \qquad |\dot{\xi} - \kappa_0| \lesssim \eps\,.
\end{align}
\end{lemma}
\begin{proof} We integrate 
\begin{align*}
|\kappa(t) - \kappa_0| \le \int_{1-\eps}^t |\dot{\kappa}| \ud t' \lesssim \eps^{\frac 9 8}\,, 
\end{align*}
where we have invoked \eqref{e:GW0}. 

For $\dot{\xi}$, we rearrange \eqref{def:mu} to obtain 
\begin{align*}
\beta_\tau \dot{\xi} = \beta_\tau \kappa - e^{- \frac s 4} \mu + \beta_\tau \beta_2 Z(0, s)\,.
\end{align*} 
Estimating the right-hand side and using that $\beta_\tau \ge \frac 1 2$ on the left-hand side yields 
\begin{align*}
|\dot{\xi} - \kappa_0| \lesssim |\kappa - \kappa_0| + e^{- \frac s 4} |\mu| + \| Z \|_\infty \lesssim \eps^{\frac 9 8} +  \eps^{\frac 1 6} e^{-  s} + \eps^{\frac 5 4}\,, 
\end{align*}
where we have invoked the bootstrap bounds \eqref{e:GW0} and \eqref{Z:boot:0}. 
\end{proof}

The following lemma verifies the bootstraps on $\dot{\tau}$, the second estimate of \eqref{e:GW0}. 
\begin{lemma}[$\dot{\tau}$ Estimate]  The following estimates are valid,
\begin{align*}
|\dot{\tau}| \les M^2 e^{-s}\,.
\end{align*}
\end{lemma}
\begin{proof} We rearrange the first ODE equation, \eqref{eq:ODE:1}, to obtain the following estimate 
\begin{align} \n
|\dot{\tau}| \le & |(1 - \dot{\tau})| |G_W^{(1)}(s, 0)| - |(1 - \dot{\tau})| | \mu| | W^{(2)}(s, 0)| + |(1 - \dot{\tau})| | F_W^{(1)}(s, 0)| \\
\les &  M^{2} e^{-s} + \eps^{\frac 4 {15}}   e^{-\frac 32 s} +  \eps^{\frac 3 4} e^{-s}    \,,\n
\end{align}

\noindent where we have invoked the second estimate in \eqref{GW:transport:est}, the bootstrap bounds \eqref{e:GW0}, \eqref{boot:decay}, and the second estimate in \eqref{FW:est:1} to estimate the forcing.  
\end{proof}

The following verifies the bootstraps on $\mu$, the first estimate on \eqref{e:GW0}. 
\begin{lemma}[$\mu$ Estimate] The following estimates are valid,
\begin{align*}
|\mu| \lesssim_M e^{-s}. 
\end{align*}
\end{lemma}
\begin{proof} We rearrange \eqref{eq:ODE:5} for $\mu(s)$, yielding 
\begin{align} \label{gW:eq}
q^{(5)} \mu(s) = &- 10 \beta_\tau q^{(2)} q^{(3)} -  \sum_{j = 2}^4 \binom{4}{j} G_W^{(j)}(0, s) q^{(5-j)} + F_W^{(4)}(s, 0)\,. 
 \end{align}

\noindent We use the bootstrap that $| q^{(5)}(s)| \ge \frac 1 2$, \eqref{moon}, to estimate from below the denominators. We then estimate the right-hand side via  
\begin{align} \n
|\mu| &\lesssim |q^{(2)}| |q^{(3)}| + \sum_{j = 2}^3 \binom{4}{j} |G_W^{(j)}(0, s)| | q^{(5-j)}| + |G_W^{(4)}(0, s)| + |F_W^{(4)}(0, s)| \\
 & \lesssim_M  e^{-  \frac 3 2 s} + e^{-\frac 7 4 s} + e^{-s} + \eps^{\frac 3 4} e^{-s} \lesssim_M e^{-s} \,,\n
\end{align}
where we have invoked \eqref{GW:transport:est} with $j = 4$, and \eqref{FW:est:1} for the $F_W^{(4)}$ term, as well as the decay bootsraps \eqref{boot:decay} on $q^{(2)}, q^{(3)}$. 
\end{proof}

The following lemma verifies the bootstrap \eqref{e:GW0} on $\dot{\kappa}$. 
\begin{lemma}[$\dot{\kappa}$ Estimate]  The following estimates are valid, 
\begin{align*}
|\dot{\kappa}| \le \frac{\eps^{\frac 1 8}}{2}\,. 
\end{align*}
\end{lemma}
\begin{proof} We rearrange equation \eqref{eq:ODE:1} to obtain
\begin{align} \n
|\dot{\kappa}| \le & |(1 - \dot{\tau})| e^{\frac 3 4 s} |\mu| + |(1 - \dot{\tau})| e^{\frac 3 4 s} |F_W(0, s)| \le 2  \eps^{\frac 1 6}  + \eps^{\frac 3 4}  \le M \eps^{\frac 1 6}\,,  
\end{align} 
where we have invoked bootstrap \eqref{e:GW0} for the $\mu$ estimate, and \eqref{FW:est:1} for the estimate on $F_W$.  
\end{proof}

\subsection{$\nabla_{\alpha, \beta}$ derivatives of modulation variables}

The following lemma verifies the bootstraps in \eqref{Baba:O:1}. 
\begin{lemma} Let $c \in \{ \alpha, \beta \}$. Then the following estimates are valid
\begin{align} \label{cor:mod}
 |\kappa_{c}| \lesssim \eps^{\frac 3 4}, \qquad |\dot{\xi}_{c}| \le \frac M 2 \eps^{\frac 1 2}\,.
\end{align}
\end{lemma}
\begin{proof} First, we have for every $- \eps \le t \le 0$, 
\begin{align*}
|\p_c \kappa(t)| = |\int_{ -\eps}^{t} \p_c \dot{\kappa}(t') \ud t' | \le \int_{-\eps}^t \eps^{\frac 1 4} e^{\frac 1 2 s(t')} \ud t' \le \int_{s_0}^\infty \eps^{\frac 1 4} e^{\frac  1 2 s'} e^{-s'} \ud s' \lesssim \eps^{\frac 3 4},
\end{align*}

\noindent where we have used that $\ud s = e^{-s} \ud t$, and the bootstrap assumption on $\dot{\kappa}_c$ in \eqref{mako:1}. 

We now compute $\p_c$ of equation \eqref{def:mu} to obtain the identity 
\begin{align} \label{free:folk:1}
\mu_c = \dot{\tau}_c \beta_\tau \mu - e^{\frac s 4} \dot{\xi}_c \beta_\tau + e^{\frac s 4} \kappa_c \beta_\tau + \beta_\tau \beta_2 e^{\frac s 4} Z_c(s, 0)\,,  
\end{align}
which upon rearranging for $\dot{\xi}_c$, we obtain 
\begin{align}
|\dot{\xi}_c| \lesssim e^{- \frac s 4}|\mu_c| + e^{- \frac s 4} |\dot{\tau}_c| |\mu| + |\kappa_c| + \| Z_c \|_\infty \lesssim {M^{33}} \eps^{\frac 1 2} e^{- \frac s 2} + \eps^{\frac 2 3} e^{- s}+ \eps^{\frac 1 2} + \eps^{\frac 1 2} \le \frac M 2 \eps^{\frac 1 2}\,,\n
\end{align}
where above we have invoked the bootstrap assumptions \eqref{mako:1} for $\p_c$ of the modulation variables, and \eqref{mako:2} for $\p_c Z$.  
\end{proof}

The following verifies the first bootstrap in \eqref{mako:1}. 
\begin{lemma} Let $c \in \{ \alpha, \beta \}$. Then the following estimates are valid
\begin{align*}
|\mu_c| \le \frac{{M^{33}}}{2} \eps^{\frac 1 2} e^{- \frac s 4}.
\end{align*}
\end{lemma}
\begin{proof} We take $\p_c$ of equation \eqref{gW:eq} which produces the identity 

\begin{align} \n
 q^{(5)} \mu_c = &- q_c^{(5)} \mu - 10 \dot{\tau}_c \beta_\tau^2 q^{(2)} q^{(3)} - 10 \beta_\tau (q^{(2)}_c q^{(3)} + q^{(2)} q^{(3)}_c)  \\  \label{deriv:gW}
 &- \sum_{j = 2}^4 \binom{4}{j} (G_W^{(j)}(0, s) q_c^{(5-j)} + \p_c G_W^{(j)}(0, s) q^{(5-j)}) + \p_c F_W^{(4)}(0, s)\,, 
\end{align}
where we recall that $q^{(j)}(s) := W^{(j)}(0, s)$, according to \eqref{q:def:q}. We now estimate each of the terms on the right-hand side above. 
 \begin{align*}
 |\mu_c| &\lesssim | q^{(5)}_c | |\mu| + |\dot{\tau}_c| |q^{(2)}| |q^{(3)}| + |q^{(2)}_c| |q^{(3)}| + |q^{(2)}| |q^{(3)}_c| \\
 &\qquad + \sum_{j = 2}^4 (|G_W^{(j)}(s, 0)| \| q^{(5-j)}_c \|_\infty + \| \p_c G_W^{(j)} \|_\infty |q^{(5-j)}|) + \| \p_cF_W^{(4)} \|_\infty \\
 &\lesssim \eps^{\frac {13} {24}} e^{- \frac 5 8 s} + \eps^{\frac 1 2} e^{- \frac 3 2 s} + M^{40}\eps^{\frac 3 4} e^{- \frac s 4} +  \eps^{\frac35}e^{-\frac{s}{4}} +M^{18} \eps^{\frac 3 4} e^{-\frac 1 4 s}  + M^{32} \eps^{\frac 1 2} e^{- \frac s 4} + \eps^{\frac 3 4} e^{- \frac s 4}\les M^32\eps^{\frac12}e^{-\frac s4}\,, 
 \end{align*}
 where we have invoked estimates \eqref{GW:transport:est}, \eqref{r:est:1} for the $G_W^{(j)}$ contributions, and \eqref{Fwc.est.ult} for the forcing term, \eqref{e:GW0} and \eqref{mako:1} to estimate $\mu$ and $ \dot \tau_c$, as well as the estimates \eqref{boot:decay}, \eqref{mako:4}, \eqref{mako:5}, \eqref{hotel:motel} to bound the terms involving q. We have also invoked bootstrap \eqref{mako:6}. 
\end{proof}

The following verifies the second bootstrap in \eqref{mako:1}. 
\begin{lemma}Let $c \in \{ \alpha, \beta \}$. Then the following estimates are valid
\begin{align*}
|\dot{\tau}_c| \le \frac 1 2 \eps^{\frac 1 2}\,.
\end{align*}
\end{lemma}
\begin{proof} We take $\p_c$ of equation \eqref{eq:ODE:2} to obtain 
\begin{align} \label{Ipso:1}
\beta_\tau (1 + \beta_\tau \dot{\tau}) \dot{\tau}_c = \p_c G^{(1)}_W(0, s) - \mu_c q^{(2)} - \mu q^{(2)}_c + \p_c F^{(1)}_W(0, s)\,. 
\end{align}

\noindent We now estimate the right-hand side above via  
\begin{align*}
| \eqref{Ipso:1} | \lesssim & \| \p_c G_W^{(1)} \|_\infty + |\mu_c| |q^{(2)}| + |\mu| |q^{(2)}_c| + \| \p_c F_W^{(1)} \|_\infty \\
\lesssim & M^{2j^2} \eps^{\frac 1 2} e^{- \frac s 4} +{M^{33}}\eps^{\frac 3 5} e^{-  s } + \eps^{\frac{11}{12}} + \eps^{\frac 3 4} e^{- \frac s 4} \lesssim_M \eps^{\frac 3 4},
\end{align*}
where above, we have invoked estimate \eqref{r:est:2} for the $\p_c G_W^{(1)}$ contribution, bootstraps \eqref{e:GW0}, \eqref{mako:1} for the $\mu, \mu_c$ estimates respectively, bootstraps \eqref{boot:decay}, \eqref{mako:4} for the $q^{(2)}, q^{(2)}_c$ contributions respectively, and finally \eqref{Fwc.est.ult} for the $\p_c F_W^{(1)}$ estimate. 

Finally, to conclude, we estimate the prefactor on the left-hand side of \eqref{Ipso:1} from below 
\begin{align*}
\beta_\tau (1 + \beta_\tau \dot{\tau}) \ge \frac 7 8( 1 - \beta_\tau  |\dot{\tau}|) \ge \frac 3 4\,. 
\end{align*}
\end{proof}

The following verifies the third bootstrap in \eqref{mako:1}. 
\begin{lemma}Let $c \in \{ \alpha, \beta \}$. Then the following estimates are valid
\begin{align*}
|\dot{\kappa}_c| < \frac 12 \eps^{\frac 14} e^{\frac s 2}\,.
\end{align*}
\end{lemma}
\begin{proof} We compute $\p_c$ of equation \eqref{eq:ODE:1} to obtain the identity 
\begin{align} \label{identity:kappa:c1}
\beta_\tau \dot{\kappa}_c = e^{\frac 3 4 s} \mu_c - \dot{\kappa} \beta_\tau^2\dot{\tau}_c + e^{\frac 3 4 s} \p_c F_W(0, s)\,,  
\end{align}
upon which estimating yields 
\begin{align} \n
| \dot{\kappa}_c| \lesssim & e^{\frac 3 4 s} | \mu_c| + |\dot{\kappa}| |\dot{\tau}_c| + e^{\frac 3 4 s} \| \p_c F_W \|_\infty \lesssim {M^{33}}e^{\frac 1 2 s} \eps^{\frac 1 2} + \eps^{\frac 5 8}  + M  \eps^{\frac 3 4} e^{ \frac s 2} \lesssim_M \eps^{\frac 1 2} e^{\frac s 2}, 
\end{align}
where we have invoked the bootstraps on the modulation variables, \eqref{e:GW0}, \eqref{mako:1}, as well as the forcing estimate \eqref{Fwc.est.ult:0}.  

\end{proof}

\subsection{$\nabla_{\alpha, \beta}^2$ derivatives of modulation variables}

The following verifies the bootstraps in \eqref{group:2}. 
\begin{lemma}Let $c_i \in \{ \alpha, \beta \}$ for $i = 1,2$. Then the following estimates are valid
\begin{align*}
|\kappa_{c_1 c_2}| \le M^{\frac 5 2}  \eps^{\frac 5 4} e^s, \qquad |\dot{\xi}_{c_1 c_2}| \le M^{\frac 7 2} \eps^{\frac 5 4} e^{s}.
\end{align*}
\end{lemma}
\begin{proof} We have to integrate 
\begin{align*}
|\kappa_{c_1 c_2}| = | \int_{1-\eps}^{t} \dot{\kappa}_{c_1 c_2}| \lesssim \int_{s_0}^s M^2 \eps^{\frac 5 4} e^{2s'} e^{-s'} \ud s' = M^2 \eps^{\frac 5 4} (e^s - e^{s_0}),
\end{align*}
where above we have invoked the bootstrap assumption \eqref{group:1} on $\dot{\kappa}_{c_1 c_2}$. 

 Next, we want to obtain an expression for $\dot{\xi}_{c_1 c_2}$. For this, we differentiate the expression \eqref{def:mu} which produces the identity 
\begin{align*}
\mu_{c_1 c_2} = \beta_\tau \dot{\tau}_{c_2} \mu_{c_1} +  \beta_\tau \dot{\tau}_{c_1} \mu_{c_2} + \dot{\tau}_{c_1 c_2} \beta_\tau \mu - e^{\frac s 4} \beta_\tau \dot{\xi}_{c_1 c_2} + e^{\frac s 4} \beta_\tau \kappa_{c_1 c_2} + \beta_\tau \beta_2 e^{\frac s 4} Z_{c_1 c_2}(s, 0),
\end{align*}
which rearranging for $\dot{\xi}_{c_1 c_2}$ gives 
\begin{align} \n
|\dot{\xi}_{c_1 c_2}| \lesssim & e^{- \frac s 4}( |\mu_{c_1 c_2}| + |\dot{\tau}_c| |\mu_c| + |\dot{\tau}_{c_1 c_2}| |\mu|) + |\kappa_{c_1 c_2}| + \| Z_{c_1 c_2} \|_\infty \\ \n
\lesssim & e^{- \frac s 4} (M \eps^{\frac 5 4} e^{\frac 5 4 s} + {M^{33}} \eps e^{- \frac s 4} + \eps^{\frac 7 6}) + M^3 \eps^{\frac 5 4} e^s + M^{2j^2} \eps^{\frac 5 8} e^{\frac s 4} \\ \n
\lesssim & M^3 \eps^{\frac 5 4} e^s\,, 
\end{align}
where above we have invoked \eqref{group:1} - \eqref{group:2} for the second derivatives of the modulation variables, \eqref{mako:1} for the $\dot{\tau}_c$ term, \eqref{e:GW0} for the $\mu$ term, and \eqref{posty:1} for the $Z_{c_1 c_2}$ contribution. 
\end{proof}

The following verifies the bootstraps in \eqref{group:1} on $\mu$. 
\begin{lemma} Let $c_i \in \{ \alpha, \beta \}$ for $i = 1,2$. Then the following estimates are valid
\begin{align*}
| \mu_{c_1 c_2}| \le \frac{ M}{2} \eps^{\frac 5 4} e^{\frac 5 4 s}.
\end{align*}
\end{lemma}
\begin{proof} We differentiate equation \eqref{deriv:gW} in $\p_{c_2}$ to get 
\begin{align}  \label{bird:1}
q^{(5)} \mu_{c_1 c_2} = & - q^{(5)}_{c_2} \mu_{c_1} - q^{(5)}_{c_1 c_2} \mu - 10 ( \dot{\tau}_{c_1 c_2} + 2 \dot{\tau}_{c_1} \dot{\tau}_{c_2} ) \beta_\tau^2 q^{(2)} q^{(3)} - 10 \beta_\tau^2 \dot{\tau}_{c_{i'}} (q^{(2)}_{c_i} q^{(3)} + q^{(2)} q^{(3)}_{c_i}) \\ \label{bird:2}
& - 10 \beta_\tau (q^{(3)}_{c_i} q^{(2)}_{c_{i'}} + q^{(2)}_{c_1 c_2} q^{(3)} + q^{(2)} q^{(3)}_{c_1 c_2}) - \sum_{j = 2}^4 \binom{4}{j} (\p_{c_2}G_W^{(j)}(s, 0) q_{c_1}^{(5-j)} \\ \label{bird:3}
& + G_W^{(j)}(s, 0) q_{c_1 c_2}^{(5-j)} + \p_{c_1}G_W^{(j)}(s, 0) q_{c_2}^{(5-j)} + \p_{c_1 c_2} G_W^{(5-j)}(s, 0) q^{(5-j)} ) + \p_{c_1 c_2} F_W^{(4)}(s, 0)\,. 
\end{align}

\noindent We now estimate all of the terms above, line by line, starting with 
\begin{align*}
| \eqref{bird:1} |& \lesssim  \| W^{(5)}_c \|_\infty |\mu_c| + \| W^{(5)}_{c_1 c_2} \|_\infty |\mu| + ( |\dot{\tau}_{c_1 c_2}| + |\dot{\tau}_c|^2) |q^{(2)}| |q^{(3)}| + |\dot{\tau}_c| ( |q^{(2)}_c| |q^{(3)}| + |q^{(2)}| |q^{(3)}_c| ) \\
&\lesssim_M  \eps^{\frac 5 4} e^{\frac s 2}+ \eps^{\frac 5 3} e^{\frac 3 4 s}  + ( \eps e^{\frac 3 4 s} + \eps )\eps^{\frac1{10}} e^{-\frac 7 4s}  +  \eps^{\frac12}(\eps^{\frac34}e^{-\frac s4}+\eps^{\frac 3{5}}e^{-\frac s4 })\les_{M} \eps^{\frac54} e^{\frac34 s}\,.
\end{align*}
Above, we have invoked \eqref{mako:6}, \eqref{buddy:1} for the $W^{(5)}_c, W^{(5)}_{c_1 c_2}$ contributions, respectively, \eqref{e:GW0}, \eqref{mako:1} and \eqref{group:1} for the estimates on the modulation variable, \eqref{boot:decay} for the decay estimates on $q^{(2)}, q^{(3)}$, and finally \eqref{mako:4} and \eqref{mako:5} for $q^{(2)}_c, q^{(3)}_c$ estimates. 

Next, we bound the terms in \eqref{bird:3} 
\begin{align*}
|\eqref{bird:3}| \lesssim & |q_c^{(3)}| |q^{(2)}_c| + |q^{(3)}| |q^{(2)}_{c_1 c_2}| + |q^{(2)}| |q^{(3)}_{c_1 c_2}| + \sum_{j = 2}^4 |\p_c G_W^{(j)}(s, 0)| \| W_c^{(5-j)}\|_{\infty} \\
\lesssim &\eps^{\frac 5 4} e^{\frac 5 4 s} +\eps^{\frac 3 2} e^{\frac s 2 }+\eps^{\frac85 }e^{\frac34s} +  M^{18}\eps^{\frac54}e^{-\frac s2}\les \eps^{\frac 5 4} e^{\frac 5 4 s} 
 \,,
\end{align*}
where above we have invoked estimate \eqref{r:est:1} for the transport term, as well as  the bootstraps \eqref{boot:decay}, \eqref{mako:4}, \eqref{mako:5}, \eqref{buddy:1} for the $q^{(2)}, q^{(3)}$ quantities (and their derivatives in $c$). 

Lastly, we estimate the terms in \eqref{bird:3}
\begin{align*}
\| \eqref{bird:3} \|_\infty &\lesssim \sum_{j = 2}^4 (\| G_W^{(j)} \|_\infty \| q_{c_1 c_2}^{(5-j)} \|_\infty + \| \p_c G_W^{(j)} \|_\infty \| W_c^{(5-j)} \|_\infty + \| \p_{c_1 c_2}G_W^{(j)} \|_\infty ) + \| \p_{c_1 c_2} F_W^{(4)} \|_\infty \\
&\lesssim_M  \eps^{\frac 3 2} e^{\frac s 2} + \eps^{\frac 5 4} e^{\frac s 2}  + \eps^{\frac 5 8} e^{\frac s 2} + \eps e^{\frac s 2},
\end{align*} 
where we have invoked estimates \eqref{GW:transport:est}, \eqref{r:est:1}, \eqref{spin:1}, and \eqref{disclosure:2}.   

\end{proof}

The following verifies the bootstraps \eqref{group:1} on $\dot{\tau}_{c_1 c_2}$. 
\begin{lemma}Let $c_i \in \{ \alpha, \beta \}$ for $i = 1,2$. Then the following estimates are valid
\begin{align*}
|\dot{\tau}_{c_1 c_2}| \le \frac{1}{2} \eps e^{\frac 3 4 s}.
\end{align*}
\end{lemma}
\begin{proof} We take $\p_{c_2}$ of equation \eqref{Ipso:1} to obtain the identity 
\begin{align} \n
\beta_\tau (1 + \beta_\tau \dot{\tau}) \dot{\tau}_{c_1 c_2} = & - \beta_\tau^2 (1 + \beta_\tau \dot{\tau}) \dot{\tau}_{c_1} \dot{\tau}_{c_2} - \beta_\tau^3 \dot{\tau} \dot{\tau}_{c_2} \dot{\tau}_{c_1} - \beta_\tau^2 \dot{\tau}_{c_2} \dot{\tau}_{c_1} - \p_{c_1 c_2} G_W^{(1)}(0, s) \\ \label{bravo:1}
&- \mu_{c_1 c_2} q^{(2)} - \mu_{c_i} q^{(2)}_{c_{i'}} - \mu q^{(2)}_{c_1 c_2} + \p_{c_1 c_2} F_W^{(1)}(0,s)\,. 
\end{align}
We now estimate each of the terms on the right-hand side above via 
\begin{align} \n
|\dot{\tau}_{c_1 c_2}| &\lesssim  |\dot{\tau}_c|^2 + (|\dot{\tau}| + 1) |\dot{\tau}_c|^2 + \| \p_{c_1 c_2} G_W^{(1)} \|_\infty + |q^{(2)}| |\mu_{c_1 c_2}| + |\mu_c| |q^{(2)}_c| \\ \n
&\qquad + |\mu| \| W^{(2)}_{c_1 c_2} \|_\infty + \| \p_{c_1 c_2} F_W^{(1)} \|_\infty \\ 
&\lesssim_M \eps +  \eps^{\frac 5 8} e^{\frac s 2} + \eps^{\frac 5 4} e^{\frac s 2} + \eps^{\frac 5 4} e^{\frac s 2} + \eps^{\frac 5 4} e^{\frac 3 4 s} + \eps e^{\frac s 2}\,,\n
\end{align}
where we have invoked estimates \eqref{spin:1} for the $G_W^{(1)}$ term above and \eqref{disclosure:2} for the $F_W^{(1)}$ term. We have also invoked \eqref{mako:1}, \eqref{group:1} - \eqref{group:2} for the modulation variables, and \eqref{buddy:1}. 

\end{proof}

The following verifies the bootstraps on $\dot{\kappa}_{c_1 c_2}$, the second estimate in \eqref{group:1}. 
\begin{lemma} Let $c_i \in \{ \alpha, \beta \}$ for $i = 1,2$. Then the following estimates are valid
\begin{align*}
|\dot{\kappa}_{c_1 c_2}| \le \frac{M^2}{2} \eps^{\frac 5 4} e^{2s}\,.
\end{align*}
\end{lemma}
\begin{proof} We compute $\p_{c_2}$ of equation \eqref{identity:kappa:c1} to get to 
\begin{align} \n
|\beta_\tau \dot{\kappa}_{c_1 c_2}| =& |- \beta_\tau^2 \dot{\tau}_{c_i} \dot{\kappa}_{c_{i'}} + e^{\frac 3 4 s} \mu_{c_1 c_2} - 2 \dot{\kappa} \beta_\tau^2 \dot{\tau}_{c_1} \dot{\tau}_{c_2} + e^{\frac 3 4s} \p_{c_1 c_2}F_W(0, s)| \\ \n
\lesssim & \eps^{\frac 3 4} e^{\frac s 2} + M \eps^{\frac 5 4} e^{2s} + \eps^{\frac 5 4} e^{\frac s 2} + \eps^{\frac 1 2} e^{\frac 5 4 s} \lesssim M \eps^{\frac 5 4} e^{2s}\,,
\end{align}
where we have invoked estimates \eqref{mako:1} for the first derivative of the modulation variables in $c$, \eqref{group:1} for the $\mu_{c_1 c_2}$ term, and estimate \eqref{disclosure:1} for the $\p_{c_1 c_2} F_W$ term. 

\end{proof}

\section{Analysis of $Z$ and $A$}

For this section, we consider the equations for $Z$ and $A$ given by \eqref{eq:Z:0} and \eqref{eq:A:0}. We begin with the lowest order estimate, for which there is no damping, in which we verify the first bootstrap assumption in \eqref{Z:boot:0}.
\begin{lemma} The quantities $(Z, A)$ satisfy the following bounds 
\begin{align} \label{feel:1}
&\| Z \|_\infty \le \frac 3 4 \eps^{\frac 5 4}\,,  && \| Z^{(n)} \|_{\infty} \le \frac{M^{2n^2}}{2} e^{- \frac 5 4 s} \text{ for } 1 \le n \le 8\,, \\ \label{feel:2}
&\| A \|_\infty \le \frac 3 4 M \eps\,,  && \| A^{(n)} \|_{\infty} \le \frac{M^{2n^2}}{2} e^{- \frac 5 4 s} \text{ for } 1 \le n \le 8\,,  
\end{align}
which thereby verifies the bootstraps \eqref{Z:boot:0} and \eqref{A:boot:9}. 
\end{lemma}
\begin{proof} An application of the Gr\"onwall lemma coupled with estimate \eqref{Z:0:order} yields the estimate  
\begin{align} \n
\| Z(\Phi_Z(s, x), s) \|_\infty \le & \| Z(x, s_0) \|_\infty + \int_{s_0}^s  \| F_Z(\Phi_Z(s', x), s') \|_\infty ds' \\
\le & \frac 1 2 \eps^{\frac 5 4} + \int_{s_0}^s \eps^{\frac 3 4 } e^{- s'} \ud s' \le \frac 3 4 \eps^{\frac 5 4}\,,\n
\end{align}

\noindent which establishes the desired bound upon invoking that $\Phi_Z(\cdot, x)$ is a diffeomorphism for all $s \ge s_0$. 

According to \eqref{Z:n}, we calculate 
\begin{align} \nonumber
e^{-\int_{s_0}^s \Big( \frac{5n}{4} + n \beta_\tau \beta_2 W^{(1)}  \Big) \circ \Phi_Z \ud s'} = & e^{- \frac{5n}{4}(s - s_0)} e^{- \int_{s_0}^s n \beta_\tau \beta_2 W^{(1)} \circ \Phi_Z} \\ \n
\le & e^{- n \beta_\tau \beta_2 \int_{s_0}^s  \eta_{- \frac 1 5} \circ \Phi_Z }e^{- \frac{5n}{4}(s - s_0)} \\
\le &C_n e^{- \frac{5n}{4}(s - s_0)}\,.\n
\end{align}

Using this estimate, coupled with \eqref{gen:n:FZ:est}, the Gr\"onwall lemma, we estimate for $n \ge 2$, 
\begin{align} \n
|Z^{(n)}(\Phi_Z(x, s), s)| \le & C_n |e^{-\frac{10}{4} (s - s_0)} Z^{(n)}(s_0, x)| + C_n \int_{s_0}^s |e^{-\frac{10}{4}(s - s')} F_{Z,n} \circ \Phi_Z |\ud s'   \\ \n
\le & C_n \eps^{\frac 5 4} e^{- \frac 5 4 (s - s_0)} +  C_n \int_{s_0}^s e^{- \frac{10}{4} (s - s')}M^{2n-1} e^{- \frac 5 4 s'} \ud s' \\ \n
 \le & \frac{M^{2n}}{2} e^{- \frac 5 4 s}\,.
\end{align}

We now perform a similar calculation for $n = 1$, using estimate \eqref{FZ:est:1} in place of \eqref{gen:n:FZ:est}. For the $A$ estimates, the identical arguments apply using Lemma \ref{phone:1}. 

\end{proof}

\begin{lemma} For $1 \le n \le 7$, we have the following estimates on $Z$ and $A$
\begin{align*}
&\| \p_c Z \|_\infty \le \frac 1 2 \eps^{\frac 1 2}\,, && \| \p_c Z^{(n)} \|_\infty \le \frac 1 2 M^{2k^2} \eps^{\frac 1 2} e^{- \frac s 2}\,, \\
&\| \p_c A \|_\infty \le \frac 1 2 \eps^{\frac 1 2}\,, && \| \p_c A^{(n)} \|_\infty \le \frac 1 2 M^{2k^2} \eps^{\frac 1 2} e^{- \frac s 2}\,, 
\end{align*}
which thereby verifies the bootstraps \eqref{mako:2} - \eqref{mako:3}. 
\end{lemma}
\begin{proof} This follows immediately from Gr\"onwall, upon invoking the two right-most estimates in \eqref{Force:Z:pc:1} - \eqref{Force:Z:pc:n} for $Z$, and similarly \eqref{Force:A:pc:1} - \eqref{Force:A:pc:n} for $A$. 
\end{proof}

\begin{lemma} For $0 \le n \le 6$, 
\begin{align*}
& \| \p_{c_1 c_2} Z^{(n)} \|_\infty \le \frac 1 2 M^{2n^2} \eps^{\frac 5 8} e^{\frac s 4}\,, && \| \p_{c_1 c_2} A^{(n)} \|_\infty \le\frac 1 2 M^{2n^2} \eps^{\frac 5 8} e^{\frac s 4},
\end{align*}
which therefore verifies the bootstrap assumptions \eqref{posty:1} - \eqref{posty:2}. 
\end{lemma}
\begin{proof}  This follows immediately from Gr\"onwall, upon invoking the two right-most estimates in \eqref{jay:1} - \eqref{jay:2} for $Z$, and similarly \eqref{found:me:1} - \eqref{found:me:2} for $A$. 
\end{proof}

\section{Analysis of $W$ at $x = 0$}\label{s:wx:at:0}

In this section, we analyze $W$ and higher order derivatives of $W$ at $x = 0$. While $q^{(0)}(s), q^{(1)}(s), q^{(4)}(s)$ are constrained from \eqref{e:constraints}, the quantities $q^{(2)}, q^{(3)}$ and $q^{(5)}$ are not constrained and therefore must be determined through ODEs in $s$ that they obey.  

\subsection{ODE analysis of $q^{(2)}, q^{(3)}$}

In this series of estimates, we use the crucial inductive assumption, \eqref{induct:1}, in order to integrate \textit{backwards} the flow. First, we rewrite the ODEs in the following way: 
\begin{align} \label{mika}
(\p_s  - \frac 3 4 )q^{(2)} = & \mathcal{F}^{(2)}(s), \qquad (\p_s - \frac 1 2) q^{(3)} = \mathcal{F}^{(3)}(s)\,. 
\end{align}

\noindent where 
\begin{align} \label{back:2}
&\mathcal{F}^{(2)} := 3(\beta_\tau - 1) q^{(2)} - \mu q^{(3)} - 2 G_W^{(1)}(0, s) q^{(2)}  - G_W^{(2)}(0, s) + F_W^{(2)}(0, s)\,, \\ \label{back:3}
&\mathcal{F}^{(3)} := 4(\beta_\tau - 1) q^{(3)} - 3 G^{(1)}_W(0, s) q^{(3)} -3 \beta_\tau |q^{(2)}|^2 - 3 G_W^{(2)}(0, s) q^{(2)} - G_W^{(3)}(0, s) + F_W^{(3)}(0, s)\,.
\end{align}
and we recall the notation $q^{(n)}=W^{(n)}(0)$ specified in \eqref{q:def:q}. 

We first prove lemmas for the particular quantities $W_{\alpha_{N}, \beta_{N}}^{(2)}(0, s)$ and  $W_{\alpha_{N}, \beta_{N}}^{(3)}(0, s)$.  

\begin{lemma}Assume that $W^{(2)}_{\alpha_N, \beta_N}(0, s_N) = 0$ and $W^{(3)}_{\alpha_N, \beta_N}(0, s_N) = 0$. Then, for all $s_0 \le s \le s_{N+1}$, the following estimates hold: 
\begin{align} \label{britney:spears:1}
|\mathcal{F}^{(2)}| \lesssim M^8 e^{-s} , \qquad |\mathcal{F}^{(3)}| \le M^{18} e^{-s}, \qquad s_0 \le s \le s_{N+1}\,,
\end{align}
and in particular, this implies that 
\begin{align}\label{eq:2:3:W0}
|W^{(2)}_{\alpha_N, \beta_N}(0, s)| \le \frac{M^9}{2} e^{-s} , \qquad |W^{(3)}_{\alpha_N, \beta_N}(0, s)| \le \frac{M^{19}}{2} e^{-s}, \qquad s_0 \le s \le s_{N+1}\,. 
\end{align}
\end{lemma}
\begin{proof} The decay estimates \eqref{eq:2:3:W0} follow upon writing the Duhamel formula associated to the evolution of \eqref{eq:ODE:3}, and crucially using the vanishing at $s_{N}$: 
\begin{align} \label{frankie:1}
W^{(2)}_{\alpha_N, \beta_N}(0, s) = & \int^s_{s_{N}} e^{\frac 3 4 (s - s')} \mathcal{F}^{(2)}(s') \ud s', \qquad W^{(3)}_{\alpha_N, \beta_N}(0, s) =  \int^s_{s_{N}} e^{\frac 12 (s - s')} \mathcal{F}^{(3)}(s') \ud s'\,.
\end{align}
We will thus focus on proving estimates \eqref{britney:spears:1}, starting with 
\begin{align*}
|\mathcal{F}^{(2)}| \lesssim & |\beta_\tau - 1| |q^{(2)}| + |\mu| |q^{(3)}| + \| G_W^{(1)} \|_\infty |q^{(2)}| + \| G_W^{(2)} \|_\infty + \|F_W^{(2)} \|_\infty \\
\lesssim &  \eps^{\frac 4 {15}} e^{- \frac 3 2 s} + M^{40} \eps^{\frac 1 6} e^{-\frac 7 4 s} + M^{2}\eps^{\frac1{10}}e^{-\frac 7 4 s} + M^{8} e^{-s} + \eps^{\frac 3 4} e^{-s} \lesssim M^{8} e^{-s},
\end{align*}
where above we have used estimates \eqref{GW:transport:est} for the transport terms $G_W$, and the estimates \eqref{FW:est:1} for the $F_W^{(2)}$ term. We have also invoked  \eqref{boot:decay},  \eqref{e:GW0}, and \eqref{e:1beta:bnd}.

We now move to 
\begin{align*}
| \mathcal{F}^{(3)} | \lesssim & |\beta_\tau - 1| |q^{(3)}| + \| G_W^{(1)} \|_\infty |q^{(3)}| + |q^{(2)}|^2 + \| G_W^{(2)} \|_\infty |q^{(2)}| + \| G_W^{(3)} \|_\infty + \|F_W^{(3)}\|_\infty \\ 
\lesssim & M^{40}  \eps^{\frac 16}  e^{- \frac 7 4 s} + M^{42} e^{-2s} + \eps^{\frac15} e^{-\frac 3 2 s} + M^{8}\eps^{\frac1{10}} e^{-\frac 7 4 s} + M^{18} e^{-s} + \eps^{\frac 3 4} e^{-s} \\
\lesssim & M^{18} e^{-s},
\end{align*}
where we have invoked estimates \eqref{boot:decay} for the $q^{(2)}, q^{(3)}$ quantities, \eqref{e:1beta:bnd} for the $|\beta_\tau - 1|$ estimate, \eqref{GW:transport:est} for the estimate of $G_W^{(1)}, G_W^{(2)}, G_W^{(3)}$, and \eqref{FW:est:1} for the forcing estimate.  

To establish \eqref{eq:2:3:W0}, we appeal to \eqref{frankie:1} (which holds for all values of $s$)
\begin{align*}
|W^{(2)}_{\alpha_N, \beta_N}(0, s)| \lesssim \int^s_{s_N} e^{\frac 3 4 (s - s')} M^8 e^{-s'} \ud s' \lesssim
M^8e^{\frac34 s}\left(e^{-\frac74 s}+e^{-\frac74 s_N}\right) \lesssim M^8 e^{-{s}},
\end{align*}
for all $s_0 \le s \le s_{N+1}$, where we have used that $s_{N+1} - s_N = 1$ to estimate $e^{s_{N+1}}e^{-s_{N}}\leq e$.

 A similar argument applies to $W^{(3)}_{\alpha_N, \beta_N}(s, 0)$. 
\end{proof}

We now verify the bootstrap assumptions \eqref{boot:decay}, which apply to every $(\alpha, \beta) \in \mathcal{B}_N(\alpha_N, \beta_N)$. 
\begin{lemma} The following estimates are valid uniformly in the parameter set $\mathcal{B}_N$ given by \eqref{def:BN} 
\begin{align*}
|W^{(2)}(0, s)| \le \frac 1 2 \eps^{\frac{1}{10}} e^{- \frac 3 4  s} , \qquad |W^{(3)}(0, s)| \le \frac{M^{40}}{2} e^{-  s}\,, 
\end{align*}
\end{lemma}
\begin{proof} We use the fundamental theorem of calculus in the space of parameters via 
\begin{align} \n
|W_{\alpha, \beta}^{(2)}(0, s)| \le & |W^{(2)}_{\alpha_{N}, \beta_{N}}(0, s)| + |\alpha - \alpha_{N}| \sup_{\alpha \in \mathcal{B}_N} |\p_\alpha W^{(2)}(0, s)| + |\beta - \beta_{N}| \sup_{\beta \in \mathcal{B}_N} |\p_\beta W^{(2)}(0, s)| \\ \n
\le & M^9 e^{-s}+ \Big( M^{30} e^{-s_N} e^{- \frac 3 4 (s_N - s_0)} + \eps^{\frac 1 5} e^{-s_N} e^{- \frac 1 2 (s_N - s_0)} \Big) 4 e^{\frac 3 4 (s - s_0)} \\ \n
& + M^{30} e^{-s_N} e^{- \frac 1 2 (s_N - s_0)} \eps^{\frac 1 4} e^{\frac 3 4 (s - s_0)} \\
\le & \frac12 \eps^{\frac{1}{10}} e^{- \frac 3 4  s}\,,\n
\end{align}
where above we have used that $s_{N+1} - s_N = 1$, coupled with the particular estimates \eqref{eq:2:3:W0}, the two left-most bootstrap bounds in \eqref{mako:4} - \eqref{mako:5}, and the assumed size of the parameter rectangle in \eqref{def:BN}.   

Similarly, for the quantity $W^{(3)}_{\alpha, \beta}$, we have 
\begin{align} \n
|W_{\alpha, \beta}^{(3)}(0,s)| \le & |W_{\alpha_N, \beta_N}^{(3)}(0, s)| + |\alpha - \alpha_{N}| \sup_{\alpha \in \mathcal{B}_N} |\p_\alpha W^{(3)}(0, s)| + |\beta - \beta_N| \sup_{\beta \in \mathcal{B}_N} |\p_\beta W^{(3)}(0, s)| \\ \n
\le & M^{19} e^{-s} + \Big( M^{30} e^{-s_N} e^{- \frac 3 4 (s_N - s_0)} + \eps^{\frac 1 5} e^{-s_N} e^{- \frac 1 2 (s_N - s_0)} \Big) \eps^{\frac 1 2} e^{\frac 1 2 (s - s_0)} \\ \n
& +  M^{30} e^{- s_N} e^{- \frac 1 2 (s_N - s_0)} 4 e^{\frac 1 2 (s - s_0)} \\
\le & \frac{M^{40}}{2} e^{-s}\,.\n
\end{align}
Again, we have invoked the particular bound \eqref{eq:2:3:W0}, the two right-most estimates in \eqref{mako:4} - \eqref{mako:5}, as well as the size of the parameter rectangle in \eqref{def:BN}.
\end{proof}

Finally, we are left at estimating $W^{(5)}(0, s)$, and in particular to verify the bootstrap assumption \eqref{boot:W:5:0}. As a result, we write the ODE evolution for this quantity, equation \eqref{eq:ODE:6}, as
\begin{align} \label{calvin:1}
&\p_s \tilde{q}^{(5)} = \mathcal{F}^{(5)}\,, 
\end{align}
where
\begin{align}
 \label{calvin:2}
&\mathcal{F}^{(5)} := - \mu q^{(6)} + (1 - \beta_\tau) q^{(5)} - 10 |q^{(3)}|^2 - \sum_{j = 1}^5 \binom{5}{j} G_W^{(j)}(0, s) q^{(6-j)} + F_W^{(5)}(0, s)\,.
\end{align} 

We now verify the bootstrap assumptions \eqref{boot:W:5:0}. 
\begin{lemma} The following estimate is valid for the quantity $\tilde{q}^{(5)}(s)$
\begin{equation} \label{fifth:deriv:W:0}
\abs{\tilde{q}^{(5)}}\les \eps^{\frac 7 8}\,. 
\end{equation}
\end{lemma}
\begin{proof} We use \eqref{calvin:1} to integrate
\begin{align}\label{e:bluebottle}
\tilde{q}^{(5)}(s) = \tilde{q}^{(5)}(s_0) + \int_{s_0}^s \mathcal{F}^{(5)}(s') \ud s'\,, 
\end{align} 
and we estimate the $\mathcal{F}^{(5)}$ term on the right-hand side via   
\begin{align}
|\mathcal{F}^{(5)}| \lesssim & \eps^{\frac{11}{30} } e^{- \frac 3 4 s'} + \eps^{\frac 1 8} e^{- \frac 3 4 s'}+ 10 M^{36}e^{-2s'} + e^{-\frac{9}{10}s'}  + \eps^{\frac 3 4} e^{-s'} \lesssim \eps^{\frac 18} e^{- \frac 3 4 s'}. \label{e:bluebottle2}
\end{align}
 Above, we have used the bootstraps \eqref{e:GW0} on $\mu$, invoked estimate \eqref{FW:est:1} to control the forcing term, \eqref{GW:transport:est} to control the transport terms, $G_W^{(j)}$, \eqref{e:1beta:bnd} to estimate the $1 - \beta_\tau$ term, estimates \eqref{boot:decay} for the $q^{(2)}, q^{(3)}$ terms, and finally \eqref{e:W6:bootstrap} for the $q^{(6)}$ term, coupled with the fact that $\bar{W}^{(6)}(0) = 0$ so $q^{(6)} = \tilde{q}^{(6)}$. 

Next, we estimate the initial data via appealing to the specific form of \eqref{guitar:or:band} and also the parameter bootstraps, \eqref{apple:1} 
\begin{align*}
|\tilde{q}^{(5)}(s_0)| = | \hat{W}_0^{(5)}(0) + \alpha\p_x^5 (x^2 \chi(|x|))(0) + \beta \p_x^5 (x^3 \chi(|x|) )(0)  | \lesssim |\alpha| + |\beta| \lesssim_M \eps\,.
\end{align*}

\end{proof}

\subsection{{ODE analysis of $\nabla_{\alpha, \beta} q^{(n)}$ for $n = 2, 3, 5$}}

We start with the two formulas, which importantly, are valid for all values of the parameters $(\alpha, \beta) \in \mathcal{B}_n$: 
\begin{align} \label{W2:ab}
&q^{(2)}(s) = W^{(2)}(0, s) = e^{\frac 34 (s - s_0)} \alpha + \int_{s_0}^s e^{\frac 34 (s - s')} \mathcal{F}^{(2)}(s') \ud s'\,, \\ \label{W3:ab}
&q^{(3)}(s) = W^{(3)}(0, s) = e^{\frac 12 (s - s_0)} \beta + \int_{s_0}^s e^{\frac 12 (s - s')} \mathcal{F}^{(3)}(s') \ud s'\,,  
\end{align}

\noindent where the forcing terms are defined in \eqref{back:2}, \eqref{back:3}. We differentiate the above expressions in $\alpha$, recalling the notation that $q_\alpha := \p_\alpha q$ and $q_\beta := \p_\beta q$
\begin{align} \label{forward:1}
&{ q^{(2)}_\alpha} = e^{\frac 3 4 (s - s_0)} + \int_{s_0}^s e^{\frac 3 4 (s - s')} \p_\alpha \mathcal{F}^{(2)}(s') \ud s'\,, \\ \label{forward:2}
&{  q_\alpha^{(3)} } =   \int_{s_0}^s e^{\frac 1 2 (s - s')} \p_\alpha \mathcal{F}^{(3)}(s') \ud s'\,. 
\end{align} 

\noindent Similarly, differentiating in $\beta$ yields the expressions 
\begin{align} \label{forward:1:b}
&{ q_\beta^{(2)} } = \int_{s_0}^s e^{\frac 3 4 (s - s')} \p_\beta \mathcal{F}^{(2)}(s') \ud s'\,, \\ \label{forward:2:b}
& { q_\beta^{(3)} }= e^{\frac 1 2 (s - s_0)} + \int_{s_0}^s e^{\frac 1 2 (s - s')} \p_\beta \mathcal{F}^{(3)}(s') \ud s'\,. 
\end{align}
Third, by integrating \eqref{calvin:1} - \eqref{calvin:2} we have 
\begin{align*}
{ \tilde{q}^{(5)}_c =  \tilde{q}^{(5)}_c(s_0) }+ \int_{s_0}^s \p_c \mathcal{F}^{(5)}(s') \ud s'\,.
\end{align*}

 We now write the expressions: 
\begin{align} \n
\p_c \mathcal{F}^{(2)} = & 3 \dot{\tau}_c \beta_\tau^2 q^{(2)} + 3 (\beta_\tau - 1) q_c^{(2)} - \mu_c q^{(3)} - \mu q^{(3)}_c - 2 \p_c G^{(1)}_W(0,s) q^{(2)} \\ \label{joes:2}
& - 2 G^{(1)}_W(0,s) q^{(2)}_c + \p_c F^{(2)}_W(0, s) + \p_c G^{(2)}_W(0, s)\,,  
\end{align}
and 
\begin{align} \n
\p_c \mathcal{F}^{(3)} =& 4 \beta_\tau^2 q^{(3)} \dot{\tau}_c + 4 (\beta_\tau - 1) q^{(3)}_c - 3 \p_c G^{(1)}_W(0, s) q^{(3)} - 3 G^{(1)}_W(0, s) q_c^{(3)} \\ \n
&- 3 \beta_\tau^2 \dot{\tau}_c |q^{(2)}|^2  - 6 \beta_\tau q^{(2)}q_c^{(2)} - 3 \p_c G_W^{(2)}(0, s) q^{(2)} - 3 G_W^{(2)}(0, s) q^{(2)}_c \\ 
& \label{joes:2:2} + \p_c G_W^{(3)}(0, s)+ \p_c F^{(3)}_W(0, s)\,,  
\end{align}
for $c \in \{ \alpha, \beta \}$. 
We also record, by differentiating \eqref{calvin:1}, the expression 
\begin{align} \n
\p_c \mathcal{F}^{(5)} = &- \mu_c q^{(6)} - \mu q^{(6)}_c - \beta_\tau^2 \dot{\tau}_c q^{(5)} + (1 - \beta_\tau) q^{(5)}_c - 20 q^{(3)} q^{(3)}_c \\ \label{joes:3:3}
& - \sum_{j = 1}^{4} \binom{5}{j} (\p_c G_W^{(j)}(0, s) q^{(6-j)} + G_W^{(j)}(0, s) q_c^{(6-j)}) + \p_c G_W^{(5)}(0, s) + \p_c F_W^{(5)}(0, s)\,. 
\end{align} 

\begin{lemma} The following estimates are valid on the quantities defined in \eqref{joes:2}, \eqref{joes:2:2}, \eqref{joes:3:3}
\begin{align} \label{whole:1}
|\p_c \mathcal{F}^{(2)} | \le  \eps^{\frac 5 8}\,, && |\p_c \mathcal{F}^{(3)}| \le \eps^{\frac 5 8}\,, && |\p_c \mathcal{F}^{(5)}| \le \eps^{\frac 3 8}\,. 
\end{align}
\end{lemma} 
\begin{proof}  We now estimate each of the terms in the forcing above in \eqref{joes:2}: 
\begin{align} \n
|\p_c \mathcal{F}^{(2)}| &\lesssim  |\dot{\tau}_c| |q^{(2)}| + |\beta_\tau -1| |q_c^{(2)}| + |\mu_c| |q^{(3)}| + |\mu| |q^{(3)}_c| + |\p_c G_W^{(1)}(0, s)| |q^{(2)}| \\ \n
& \qquad \qquad + |G_W^{(1)}(0, s)| |q^{(2)}_c| + |\p_c F_W^{(2)}(0, s)| + |\p_c G_W^{(2)}(0, s)| \\  \label{wolf:alice}
&\lesssim_M  \eps^{\frac 1 2} e^{-\frac 3 4 s} + \eps^{\frac 1 6} e^{- \frac 3 4 s} + \eps^{\frac 1 2} e^{- \frac 5 4s} + \eps^{\frac{11}{12}}+  \eps^{\frac 1 2} e^{- s}  + \eps^{\frac 3 4} e^{- \frac s 4}  + \eps^{\frac 3 4 } e^{- \frac s 4} + \eps^{\frac 1 2} e^{- \frac s 4}\les_M \eps^{\frac34}\,,
\end{align} 
and similarly, we estimate  
\begin{align} \n
|\p_c \mathcal{F}^{(3)}| &\lesssim |q^{(3)}| |\dot{\tau}_c| + |\beta_\tau - 1| |q^{(3)}_c| + \| \p_c G_W^{(1)} \|_\infty |q^{(3)}| + \|G_W^{(1)}\| |q_c^{(3)}| + |\dot{\tau}_c| |q^{(2)}|^2 \\ \n
& \qquad  + |q^{(2)}| |q^{(2)}_c| + \| \p_c G_W^{(2)}\| |q^{(2)}| + \|G_W^{(2)}\| |q_c^{(2)}| + \|\p_c G_W^{(3)} \| + \|\p_c F_W^{(3)} \| \\ 
& \lesssim_M \eps^{\frac 1 2} e^{-s}  + \eps^{\frac{11}{12}} +   \eps^{\frac 1 2} e^{- \frac 5 4 s} + \eps^{\frac 3 4} e^{- \frac s 4}  + \eps^{\frac 1 2} e^{-\frac 3 2 s} + \eps^{\frac 3 4} + \eps^{\frac 1 2}  e^{-\frac 5 4 s} + \eps^{\frac 3 4} e^{-\frac 1 4 s} + \eps^{\frac 1 2} e^{- \frac s 4} + \eps^{\frac 3 4} e^{- \frac s 4}\n\\
& \lesssim_M \eps^{\frac34}\,.
\label{wolf:alice:2}
\end{align}
In both estimates above we have invoked the bootstrap estimate \eqref{e:GW0} on $\mu$, the estimate \eqref{e:1beta:bnd} on $|1 - \beta_\tau|$, the bootstraps \eqref{mako:1} on the $\dot{\tau}_c, \mu_c$ terms, \eqref{boot:decay} for the decay estimates on $q^{(2)}, q^{(3)}$, \eqref{mako:4} - \eqref{mako:5} for the estimates on $q^{(2)}_c, q^{(3)}_c$, and finally \eqref{r:est:1} and \eqref{Fwc.est.ult} for the transport and forcing terms, respectively. 

From \eqref{wolf:alice} and \eqref{wolf:alice:2}, we can take $\eps$ small relative to the implicit constant which depends on $M$ to conclude that 
\begin{align*}
|\p_c \mathcal{F}^{(2)}| \le \eps^{\frac 5 8}, \qquad |\p_c \mathcal{F}^{(3)}| \le \eps^{\frac 5 8}\,.
\end{align*}

 Finally, estimating $\p_c \mathcal{F}_5$ yields
\begin{align} \n
|\p_c \mathcal{F}_5| &\lesssim  |\mu_c| \|W^{(6)} \|_\infty + |\mu| \| W_c^{(6)} \|_\infty + |\dot{\tau}_c| |q^{(5)}| + |1 - \beta_\tau| |q^{(5)}_c| + |q^{(3)}| |q^{(3)}_c| \\ \n
&\qquad + \sum_{j  = 1}^4 (\| \p_c G_W^{(j)} \|_\infty |q^{(6-j)}| + \| G_W^{(j)} \|_\infty \| W_c^{(6-j)} \|_\infty ) + \| \p_c G_W^{(5)} \|_\infty +  \| \p_c F_W^{(5)} \|_\infty  \\ \n
&\lesssim_M \eps^{\frac 1 2} e^{- \frac s 4} + \eps^{\frac{11}{12}} + \eps^{\frac 1 2} + \eps^{\frac 16} e^{- \frac 3 4 s} (1 +  \eps^{\frac 3 8} e^{\frac s 8} ) + \eps^{\frac 3 4} e^{- \frac s 4} + \eps^{\frac 1 2} e^{- \frac s 4} + \eps^{\frac 3 4} e^{- \frac s 4}\\ \n
& \qquad  + \eps^{\frac 1 2} e^{- \frac s 4} + \eps^{\frac 3 4} e^{- \frac s 4}\\
&\les_M \eps^{\frac12}\,, \n
\end{align}
from which we can conclude $|\p_c \mathcal{F}^{(5)}| \le \eps^{\frac 3 8}$, establishing the final estimate of \eqref{whole:1}. We invoke the same set of bootstraps as in the estimate of $\p_c \mathcal{F}^{(2)}, \p_c \mathcal{F}^{(3)}$ above, and in addition we invoke \eqref{hotel:motel} on the estimate of $q^{(5)}_c$ and \eqref{mako:6} on the $W^{(n)}_c$ quantities. 
\end{proof}

\begin{corollary} The following estimates are valid \begin{align} \label{joes:oes:1}
|q^{(2)}_\alpha - \eps^{\frac 3 4} e^{\frac 3 4 s} |& \le \eps^{\frac 5 4} e^{\frac 3 4s}\,,  & |q^{(2)}_\beta| &\le \frac{1}{2} \eps^{\frac 5 4} e^{\frac 3 4 s} \,,\\ \label{joes:oes:2}
|q^{(3)}_\alpha| &\le \frac{1}{2} \eps e^{\frac 1 2 s}  \,,& |q^{(3)}_\beta - \eps^{\frac 1 2} e^{\frac s 2}| &\le  \eps e^{\frac 12 s} \,, \\ \label{joes:oes:3}
|\tilde{q}_c^{(5)}|& \le \frac 12 \eps^{\frac 3 8} e^{\frac 1 8 s}\,. 
\end{align}
In particular, this verifies the bootstrap estimates \eqref{mako:4} - \eqref{mako:5}, and \eqref{hotel:motel}. 
\end{corollary}
\begin{proof} For \eqref{joes:oes:1} - \eqref{joes:oes:2}, this follows immediately upon combining estimates \eqref{whole:1} with the expressions \eqref{forward:1} - \eqref{forward:2:b}. For the estimate on $\tilde{q}_c^{(5)}$, we need to use that  
\begin{align*}
\tilde{q}_\alpha^{(5)}(s_0) &= \p_x^5|_{x= 0} \Big( x^2 \chi(x) \Big) = 0\,, \\
\tilde{q}_\beta^{(5)}(s_0) &= \p_x^5|_{x = 0} \Big( x^3 \chi(x) \Big) = 0\,,
\end{align*}
according to \eqref{guitar:or:band}.
\end{proof}%

\section{Estimates for $W$}

In this section we will verify various pointwise bootstrap estimates on $W$, solving \eqref{eq:W:0}, and derivatives thereof. The main objective is to verify the bootstrap assumptions \eqref{W:boot:0} - \eqref{weds:1},  \eqref{eq:W1:bnd:1},  \eqref{e:Wtilde:bootstrap} - \eqref{gerrard:2}, \eqref{pc:W0} - \eqref{warrior:2}, as well as \eqref{buddy:1}. 

The following lemma verifies the bootstrap \eqref{eq:W1:bnd:1}. 
\begin{lemma} The following estimate is valid on $W^{(1)}$
\begin{align*}
|W^{(1)}| \le 1 + \frac \ell 2 M^{40} e^{-s}\,, 
\end{align*}
which in particular verifies \eqref{eq:W1:bnd:1}.
\end{lemma}
\begin{proof} We subdivide into three regions $|x| \le \ell$, $\ell\le |x| \le \eps^{-\frac14}$ and $\abs{x}\geq\eps^{-\frac14}$. In the middle region $\ell \le |x| \le \eps^{- \frac 1 4}$, we have 
\begin{align*}
|W^{(1)}(x, s)| \le |\bar{W}^{(1)}(x)| + |\tilde{W}^{(1)}(x, s)| \le 1 - \frac{\ell^7}{50} + |\tilde{W}^{(1)}(x, s)| \le 1 - \frac{\ell^7}{50} + \eps^{\frac 1 5} < 1,
\end{align*}
where we have invoked \eqref{truth:1} to bound $|\bar{W}^{(1)}|$ above in this region, and the bootstrap \eqref{e:Wtilde:1:bootstrap} which is also valid in this region. 

In the far-field region, $|x| \ge \ell$, we use 
\begin{align*}
|W^{(1)}(x)| \le M \eta_{- \frac 1 5}(x) \lesssim_M (\eps^{- \frac 1 4})^{\frac 4 5}\,. 
\end{align*}

In the region $|x| \le \ell$, we obtain by a Taylor expansion of $W^{(1)}$ for some $|x_\ast| \le \ell$.
\begin{align} \n
W^{(1)}(x, s) =& -1 + W^{(2)}(0, s) x + W^{(3)}(0, s) \frac{x^2}{2} + W^{(5)}(x_\ast, s) \frac{x^4}{24} \\ \n
= & -1 + W^{(2)}(0, s) x + W^{(3)}(0, s) \frac{x^2}{2} + \bar{W}^{(5)}(x_\ast) \frac{x^4}{24} +  \tilde{W}^{(5)}(x_\ast, s) \frac{x^4}{24} \\ \n
\ge &  (-1 + \bar{W}^{(5)}(x_\ast) \frac{x^4}{24} - | \tilde{W}^{(5)}(x_\ast, s) \frac{x^4}{24}|)  + W^{(2)}(0, s) x + W^{(3)}(0, s) \frac{x^2}{2} \\ 
\ge & -1 + \ell M^{40} e^{-s} + \ell^2 \frac{M^{40}}{2} e^{-s}\,.\n 
\end{align}

\noindent Above, we have used property \eqref{fifth:deriv:bar:W} to assert that $\bar{W}^{(5)}(x_\ast) > \frac 1 2$ via a further Taylor expansion: 
\begin{align*}
\bar{W}^{(5)}(x_\ast) > \bar{W}^{(5)}(0) - |x_\ast| \| \bar{W}^{(6)} \|_\infty > \bar{W}^{(5)}(0) - C \ell > \frac 1 2\,. 
\end{align*}

\noindent in which case we use \eqref{fifth:deriv:W:0} to bound 
\begin{align*}
\frac{x^4}{24} \Big( \bar{W}^{(5)}(x_\ast) - |\tilde{W}^{(5)}(x_\ast, s)| \Big) \ge \frac 1 2 - \eps \ge \frac 1 4\,.
\end{align*}
\end{proof}

We now collect various estimates on damping terms. To do so, we first make the following definitions. 
\begin{align} \label{country:coffee:1}
D_n &:=  \frac 1 4 (- 1 + 5n) + \beta_{\tau}(n+1_{n> 1}) W^{(1)}\,,\\ \label{country:coffee:2}
\tilde{D}_n&:=  \frac 1 4 (-1 + 5n) + \beta_{\tau}\left(\bar W^{(1)}+ nW^{(1)} \right) \,,\\ \label{country:coffee:3}
D_{n}^{c} &:= \frac{5n-1}{4} +(n+1) \beta_\tau W^{(1)}\,, \\ \label{bravo:16}
D_{n,r} &:= D_n - \eta_{- \frac r 4} \mathcal{V}_W \p_x \eta_{\frac r 4} =  \frac 1 4 (- 1 + 5n) + \beta_{\tau}(n+1_{n> 1}) W^{(1)}- \eta_{- \frac r 4} \mathcal{V}_W \p_x \eta_{\frac r 4}\,, \\ \label{country:coffee:4}
\tilde{D}_{n,r} &:= \tilde{D}_{n} - \eta_{- \frac r 4} \mathcal{V}_W \p_x \eta_{\frac r 4} =  \frac 1 4 (-1 + 5n) + \beta_{\tau}\left(\bar W^{(1)}+ nW^{(1)} \right) - \eta_{- \frac r 4} \mathcal{V}_W \p_x \eta_{\frac r 4} \,, \\ \label{country:coffee:5}
D_{n, r}^{c} &:= D_{n}^{c} - \eta_{- \frac r 4} \mathcal{V}_W \p_x \eta_{\frac r 4} =  \frac{5n-1}{4} +(n+1) \beta_\tau W^{(1)} -  \eta_{- \frac r 4} \mathcal{V}_W \p_x \eta_{\frac r 4}\,. 
\end{align}
We now state various estimates on these damping terms. 

\begin{lemma} Let $|x_0| \ge \ell$. Then, for $D \in \{ \tilde{D}_6, D_7^{c} \}$, $\bar{D} \in \{ \tilde{D}_{1, \frac 4 5}, \tilde{D}_{0,- \frac 1 5} \}$, and for $n \ge 2$, $j \ge 1$, the following estimates are valid 
\begin{align} \label{headphones:1}
D  \ge & \frac 1 8\,,   \\ \label{street:beat:2}
-\int_{s_0}^s \bar{D} \circ \Phi^{x_0}_W  \le & \frac{1}{50} \log M\,, \\ \label{street:beat:3}
- \int_{s_0}^s D_{n, \frac 4 5} \circ \Phi^{x_0}_W \le & - \frac 1 9 (s - s_0) + \frac{1}{50} \log M\,, \\ \label{thor:4}
- \int_{s_0}^s W^{(1)} \circ \Phi^{x_0}_W \le &\frac{1}{50} \log M\,,  \\ \label{thor:5}
- \int_{s_0}^s D^{c}_{j, \frac{1}{5}} \circ \Phi^{x_0}_W \le & \frac{1}{50} \log M\,.
\end{align}
\end{lemma}
\begin{proof} First, for \eqref{headphones:1}, 
\begin{align}
 \tilde{D}_{6} =\frac 1 4 (-1 + 30) + \beta_{\tau}\left(\bar W^{(1)}+ 6W^{(1)} \right)  \geq \frac{1}{4}- 6\abs{1-\beta_\tau}\geq \frac{1}{8}\,.\label{eq:damping:W6}
 \end{align}
 where we have used that $\bar W^{(1)}\geq -1$, \eqref{eq:W1:bnd:1} and \eqref{e:GW0}. An analogous estimate applies for the $D_7^{c}$ term. 
 
We turn now to \eqref{street:beat:2}. By a simple calculation, we have
\begin{align*}
&\tilde{D}_{0, -\frac 1 5}= \beta_{\tau} \bar W^{(1)}+ \frac{1}{4} \eta_{-1}+\frac{x^3}{5}\eta_{-1}g_W\,,\\
&\tilde{D}_{1, \frac 4 5}= \beta_{\tau} ( \bar W^{(1)}+W^{(1)})- \eta_{-1}-\frac{4x^3}{5}\eta_{-1}g_W\,.
\end{align*}

Observe, that for either the case $D_q=\tilde{D}_{0, -\frac 1 5}, \tilde{D}_{1, \frac 4 5}$, we have from \eqref{W:boot:0}, \eqref{e:uniform:W1}, \eqref{e:1beta:bnd}, \eqref{GW:transport:est}
\begin{align*}
\abs{D_q}&\leq  3 \ell \log M \eta_{-\frac15}+\eta_{-1}(1+\abs{x}(\abs{W}+\abs{G_W}))\\
&\leq  4 \ell \log M \eta_{-\frac15}+\abs{x}\eta_{-1}( \frac{1}{1000} \log M \eta_{\frac1 {20}}+ \eta_{\frac14})\\
&\leq 6\ell \log M\,.
\end{align*}
Thus, using in addition \eqref{e:escape:from:LA}, we have
\begin{align} \n
 -\int_{s_0}^s  D_{q}\circ\Phi_W^{x_0}(s')  \,ds' \leq &  6 \ell \log M \int_{s_0}^s  \left(\eta_{-\frac15}(\ell \eps^{\frac15}e^{\frac15 s})+e^{- s}\right) ds' \\  \label{steve:aoki}
 \leq & 6 \ell \log M (20 \log \ell^{-1}) \leq \frac{1}{50} \log M\,.
\end{align}
The same calculation establishes estimate \eqref{street:beat:3}, \eqref{thor:4}, \eqref{thor:5}, with minor modifications. 
\end{proof}

\subsection{Transport Estimates for $W$}

We now prove a uniform estimate on $\tilde{ W}^{(6)}$ in the region $\abs{x}\leq \ell$.  We will prove the estimates along trajectories originating at $\abs{x_0}\leq \ell$. Note that no trajectory originating outside the ball of radius $\ell$ may enter the ball of radius $\ell$. This is a consequence of \eqref{e:escape:from:LA}. The following establishes the bootstrap bounds \eqref{e:W6:bootstrap} - \eqref{gerrard:2}.

\begin{lemma} The following localized estimates hold in the region $\abs{x}\leq \ell $
\begin{align}\label{iniesta:1}
 & |\tilde{W}^{(n)}| \le  \frac 1 2 (\abs{x}^{6-k}\eps^{\frac{1}{5}}+\eps^{\frac12})\leq \abs{\ell}^{6-n}\eps^{\frac{1}{5}},\quad \mbox{for }  n=0,\dots, 5\,, \\ 
 \label{messi:1}
 &  |\tilde{W}^{(6)}| \le \frac 1 2 \eps^{\frac{1}{5}}\,,\\   \label{xavi:1}
 & |\tilde{W}^{(7)}| \le \frac M 2 \eps^{\frac 1 5}\,, \\ \label{rooney:1}
 & |\tilde{W}^{(8)}| \le \frac{M^3}{2} \eps^{\frac 1 5} \,. 
\end{align}
\end{lemma}
\begin{proof} Composing with the flow we have
 \[\frac{d}{ds}\left(\tilde W^{(6)}\circ \Phi_W^{x_0}\right) +\left(\tilde{D}_{6}\circ \Phi_W^{x_0}\right)\left(\tilde W^{(6)}\circ \Phi_W^{x_0}\right)=\tilde{F}_{W,n}\circ \Phi_W^{x_0}\,.\]
 Hence, applying Gr\"onwall,  \eqref{e:FW6:decay} and the lower bound \eqref{eq:damping:W6}, we obtain
 \begin{equation*}
 \abs{\tilde W^{(6)}\circ \Phi_W^{x_0}}\les \abs{\tilde W^{(6)}(x_0,-\log\eps)}+\ell\eps^{\frac 1 5}\les \ell\eps^{\frac 1 5}\,.
 \end{equation*}
 The same argument applies for \eqref{xavi:1} and \eqref{rooney:1} using the latter two estimates in \eqref{e:FW6:decay}. 
 
 From the constraints \eqref{e:constraints} and the estimate \eqref{e:W6:bootstrap}, we have
\begin{align*}
\tilde W(x) =\frac{\tilde W^{(2)}(0)}{2!}x^2 +  \frac{\tilde W^{(3)}(0)}{3!}x^3  +  \frac{\tilde W^{(5)}(0)}{5!}x^5  + \OO(\eps^{\frac15}|x|^6) \, .
\end{align*}
Then applying \eqref{eq:2:3:W0} and \eqref{fifth:deriv:W:0}, we obtain \eqref{iniesta:1}.

\end{proof}

%
%
%

\begin{lemma} 
For $\ell\leq \abs{x}\leq \eps^{-\frac14}$ we have
\begin{align} \label{rover:1}
 |\tilde{W}| &\le \frac 1 2 \eps^{\frac{3}{20}}\eta_{\frac{1}{20}} \,, \\ \label{3eb:4}
 |\tilde{W}^{(1)}| &\le \frac 1 2 \eps^{\frac{1}{20}}\eta_{-\frac15}\,, 
\end{align}
which thus verifies the bootstraps \eqref{thurs:1} - \eqref{e:Wtilde:1:bootstrap}. 
\end{lemma}
\begin{proof} We write  
\begin{align} \label{stuck:1}
(\p_s + \tilde{D}_{0, - \frac 1 5}) (\eta_{- \frac{1}{20}} \tilde{W}) + \mathcal{V}_W \p_x (\eta_{- \frac{1}{20}} \tilde{W}) &= \eta_{- \frac{1}{20}} \tilde{F}_{W,0}\,, \\ \label{stuck:2}
(\p_s + \tilde{D}_{1, \frac 4 5}) (\eta_{\frac{1}{5}} \tilde{W}^{(1)}) + \mathcal{V}_W \p_x (\eta_{\frac 1 5} \tilde{W}^{(1)}) &= \eta_{\frac{1}{5}} \tilde{F}_{W,1}\,.
\end{align}

We now fix any $|x_0| \ge \ell$. We will consider trajectories starting with $(s_\ast, x_0 = \pm \ell)$ or $(s_0, x_0)$ for $|x_0| > \ell$. Writing the solution to \eqref{stuck:1} we obtain
\begin{align*}
\eta_{- \frac{1}{20}} \tilde{W} \circ \Phi^{x_0}_W = \eta_{- \frac{1}{20}} \tilde{W}(s_\ast, \Phi^{x_0}_W(s_\ast)) e^{- \int_{s_\ast}^s \tilde{D}_{0, - \frac 1 5} \circ \Phi^{x_0}_W} + \int_{s_\ast}^s e^{- \int_{s'}^s \tilde{D}_{0, - \frac 1 5} \circ \Phi^{x_0}_W} \eta_{- \frac{1}{20}} \tilde{F}_{W} \circ \Phi_W^{x_0} \ud s'\,.
\end{align*}
We now estimate both sides to produce 
\begin{align*} 
|\eta_{- \frac{1}{20}} \tilde{W} \circ \Phi^{x_0}_W | \le & ( \eps^{\frac 34}  + 2 \ell^6 \eps^{\frac 1 5} ) M^{\frac{1}{50}} + \int_{s_\ast}^s M^{\frac{1}{50}} e^{- \frac 3 4 s'} \ud s' \le \frac 1 2 \eps^{\frac{3}{20}}\,. 
\end{align*}
Above, we have invoked estimate \eqref{FW:est:2} on the forcing term and \eqref{street:beat:2} for the damping term. We have moreover estimated the initial data by using \eqref{guitar:or:band} to write 
\begin{align} \label{truck}
\tilde{W}(x, s_0) = & \hat{W}_0 + \alpha x^2 \chi + \beta x^3 \chi - \bar{W} (1 - \chi(\eps^{\frac 1 4}x) )\,.
\end{align}
When $|x| \le \eps^{- \frac 1 4}$, the last term above is zero, and so we estimate, for $|x| \le \eps^{- \frac 1 4}$, 
\begin{align*}
|\tilde{W}(x, s_0) \eta_{- \frac{1}{20}}| \le \|\hat{W}_0 \eta_{- \frac{1}{20}} \|_\infty + |\alpha| + |\beta| \le \eps^{\frac 3 4}\,,
\end{align*}
by the estimates \eqref{assume:1} and \eqref{apple:1}.  

Writing the solution to \eqref{stuck:2} yields 
\begin{align*}
\eta_{\frac 1 5} \tilde{W}^{(1)} \circ \Phi^{x_0}_W = \eta_{\frac 1 5} \tilde{W}^{(1)}(s_\ast, x_0) e^{- \int_{s_\ast}^s \tilde{D}_{1, \frac 4 5} \circ \Phi^{x_0}_W} + \int_{s_\ast}^s e^{- \int_{s'}^s \tilde{D}_{1, \frac 4 5} \circ \Phi^{x_0}_W} \eta_{\frac 1 5} \tilde{F}_{W,1} \circ \Phi_W^{x_0} \ud s' \,.
\end{align*}
We now estimate the right-hand side via 
\begin{align} \n
|\eta_{\frac 1 5} \tilde{W}^{(1)} \circ \Phi^{x_0}_W| \le &( \eps^{\frac 3 4}  + 2 \ell^5 \eps^{\frac 1 5}  ) M^{\frac{1}{50}} + \eps^{\frac{1}{10}} M^{\frac{1}{50}} \int_{s_\ast}^s |\eta_{- \frac{1}{20}}(x_0 e^{\frac 1 5 (s' - s_0)}) \ud s' \le \frac 1 2 \eps^{\frac{1}{20}}\,,
\end{align}
where above we have invoked estimate \eqref{street:beat:2} for the damping term, and \eqref{FW:est:2} for the forcing term. For the initial data, we differentiate \eqref{truck} to obtain 
\begin{align*}
\tilde{W}^{(1)}(x, s_0) = \hat{W}_0' + \p_x \Big( \alpha x^2 \chi + \beta x^3 \chi \Big) - \p_x \Big( \bar{W} (1 - \chi(\eps^{\frac 1 4}x)) \Big),
\end{align*}
which upon noting that the latter term is identically zero on $|x| \le \eps^{- \frac 1 4}$, we obtain 
\begin{align*}
|\tilde{W}^{(1)}(x, s_0) \eta_{\frac 1 5}| \le \| \eta_{\frac 1 5} \hat{W}_0' \|_\infty + |\alpha| + |\beta| \le \eps^{\frac 3 4}\,, 
\end{align*} 
upon invoking estimates \eqref{assume:1} and \eqref{apple:1}. 

\end{proof}

\begin{lemma} For $\abs{x}\geq \ell$ we have
\begin{align} \label{3eb:1}
 |W| &\leq  \frac{\ell}{2} \log M \eta_{\frac{1}{20}}\,, \\ \label{3eb:2}
 |W^{(1)}| &\leq \frac \ell 2  \log M \eta_{- \frac 1 5}\,, \\ \label{3eb:3}
 \abs{W^{(n)}}&\leq \frac 1 2 M^{k^2} \eta_{- \frac 1 5}\quad \text{ for } n =2,\dots, 8\,,  
\end{align}
which verifies the bootstraps \eqref{W:boot:0} - \eqref{weds:1}. 
\end{lemma}
\begin{proof}We write, for $n \ge 1$, 
\begin{align} \label{stuck:3}
&(\p_s + D_{n, \frac 4 5}) \eta_{\frac 1 5} W^{(n)} + \mathcal{V}_W \p_x (\eta_{\frac 1 5} W^{(n)}) = \eta_{\frac{1}{5}} F_{W,n}\,, \\ \label{stuck:4}
&(\p_s + D_{0, -\frac 1 5}) (\eta_{- \frac{1}{20}} W) + \mathcal{V}_W \p_x (\eta_{- \frac{1}{20}}W) = \eta_{- \frac{1}{20}} F_{W,0}\,. 
\end{align}
We will treat the cases $n = 0$, $n = 1$, and $n \ge 2$ cases separately. 

Writing Gr\"onwall for \eqref{stuck:3} gives 
\begin{align} \label{nat:1}
\eta_{\frac 1 5}W^{(n)} \circ \Phi_W^{x_0} = \eta_{\frac 1 5} W^{(n)}(s_\ast, x_0) e^{- \int_{s_\ast}^s D_{n, \frac 4 5} \circ \Phi_W^{x_0}} + \int_{s_\ast}^s e^{- \int_{s'}^s D_{n, \frac 4 5} \circ \Phi_W^{x_0} }  \eta_{\frac 1 5} F_{W,n} \circ \Phi_W^{x_0} \ud s'\,. 
\end{align}
Estimating both sides for $n \ge 2$ gives 
\begin{align*}
|\eta_{\frac 1 5}W^{(n)} \circ \Phi_W^{x_0}| \le ( M + 10 \eps^{\frac 1 5} ) e^{- \frac{1}{9} (s - s_\ast)} M^{\frac{1}{50}} +M^{\frac{1}{50}} \int_{s_\ast}^s e^{- \frac{1}{9} (s - s')} M^{-\frac{9}{10}} M^{n^2}  \ud s',
\end{align*}
where we have appealed to estimate \eqref{street:beat:3} for the damping term and estimate \eqref{FW:est:3} for the forcing. 

For the $n = 0, 1$ cases, it suffices to prove estimates \eqref{3eb:1} and \eqref{3eb:2} in the region $|x| \ge \eps^{- \frac 1 4}$ due to \eqref{rover:1} - \eqref{3eb:4}. In this case, we select $|x_0| \ge \eps^{- \frac 1 4}$ and $s_\ast \ge s_0$ such that $(s_\ast, x_0)$ is the origin of the trajectories consider. More specifically, we take either $|x_0| > \eps^{- \frac 1 4}$ and $s_\ast = s_0$ or $|x_0| = \eps^{- \frac 1 4}$ and any $s_\ast \ge s_0$. In this case, \eqref{nat:1} continues to hold for $n = 1$, and we estimate via 
\begin{align} \n
|\eta_{\frac 1 5} W^{(1)} \circ \Phi_W^{x_0}| \le & |\eta_{\frac 1 5} W^{(1)}(x_0, s_\ast) | | e^{- \int_{s_\ast}^s D_{1, \frac 4 5} \circ \Phi_W^{x_0}} | + \int_{s_\ast}^s |e^{- \int_{s'}^s D_{1, \frac 4 5} \circ \Phi_W^{x_0}}| \| \eta_{\frac 1 5} F_{W,1} \|_\infty \ud s' \\ \label{martin:1}
\lesssim & \Big( \sup_{|x| \ge \eps^{- \frac 1 4}} |\eta_{\frac 1 5} W^{(1)}(x, s_0)|  + | \eta_{\frac 1 5} W^{(1)}(\eps^{- \frac 1 4}, s_\ast)| \Big)  + \int_{s_\ast}^s e^{- \frac 1 2 s'} \ud s' \\ \label{martin:2}
\lesssim & \Big(1 +|\eta_{\frac 1 5} \bar{W}^{(1)}(\eps^{- \frac 1 4})| + |\eta_{\frac 1 5} \tilde{W}^{(1)}(\eps^{- \frac 1 4}, s_\ast)| \Big)  + \int_{s_\ast}^s e^{- \frac 1 2 s'} \ud s' \\ \n
\le & \frac \ell 2 \log M\,. 
\end{align}
To evaluate the size of the initial data, from \eqref{martin:1} to \eqref{martin:2}, we have used \eqref{guitar:or:band} to compute
\begin{align*}
|\eta_{\frac 1 5} W^{(1)}(x, s_0)| = \Big| \Big( \bar{W}^{(1)} \chi(\eps^{\frac 1 4}x) + \bar{W} \eps^{\frac 1 4} \chi'(\eps^{\frac 1 4}x) + \hat{W}_0' + \p_x \Big( \alpha x^2 \chi(x) + \beta x^3 \chi(x) \Big) \Big) \eta_{\frac 1 5} \Big| \lesssim 1\,. 
\end{align*} 

Above, we have invoked the choice \eqref{choice:M} to ensure that $\ell \log M$ can be selected larger than the implicit constants appearing in the above estimate. We have also invoked bootstrap \eqref{e:Wtilde:1:bootstrap} to control the $\tilde{W}^{(1)}$ term above. We have also invoked \eqref{Motor:1} to control the forcing term, and used the fact that 
\begin{align*}
\exp \Big( - \int_{s_0}^s D_{1, \frac 4 5} \circ \Phi_W^{x_0} \Big) \le 10 \quad\text{ for } |x_0| \ge \eps^{- \frac 1 4}\,. 
\end{align*}

An analogous series of estimates applies to \eqref{3eb:1}.  
\end{proof}

\subsection{Transport estimates of $\nabla_c W$}

We now verify the bootstrap estimates \eqref{warrior:1} - \eqref{warrior:2}.  
\begin{lemma} For $n=0,\dots, 6$ and $\abs{x}\le \ell$ we have the following estimates
\begin{align} \label{warrior:1:kim}
&|W_c^{(n)}| \le M \ell^{\frac 3 4} \eps^{\frac 3 4} e^{\frac 3 4 s}\,, \\  \label{warrior:2:kim}
&|W_c^{(7)}| \le \frac M 2 \eps^{\frac 3 4} e^{\frac 3 4 s} \,. 
\end{align}
\end{lemma}
\begin{proof} The first inequality above follows for $n = 0$ upon Taylor expanding and noting that $W_c(0, s) = 0$ via 
\begin{align*} 
| W_c | \le \ell \sup_{|x| \le \ell} |W_c^{(1)}| \le \ell M \ell^{\frac 1 2} e^{\frac 3 4(s- s_0)}\,. 
\end{align*}

\noindent The exact same argument works for the $n = 1$ inequality. For the $n = 2$ inequality, we also Taylor expand, but must factor in the value at $x = 0$ via
\begin{align*}
|W^{(2)}_c| \le |W^{(2)}_c(0, s)| + \ell  \sup_{|x| \le \ell} |W^{(3)}| \le  4 e^{\frac 3 4(s - s_0)} + \ell M e^{\frac 3 4(s - s_0)}\,. 
\end{align*}

 Finally, for the $n = 7$ case, we directly apply Gr\"onwall to integrate which gives 
\begin{align*}
W_c^{(7)}( \Phi_W(x, s), s) = W_c^{(7)}(x, s) e^{- \int_{s_0}^s D_7^{c} \circ \Phi_W} + \int_{s_0}^s e^{- \int_{s'}^s D_7^{c} \circ \Phi_W} F_{W,7}^{c} \circ \Phi_W \ud s'\,. 
\end{align*}

\noindent We note that \eqref{headphones:1} implies that 
\begin{align*}
e^{- \int_{s_0}^s D_7^{c} \circ \Phi_W^{x_0}} \le e^{- \frac 1 8 (s - s_0)}\,.
\end{align*}

\noindent Thus, we have 
\begin{align} \n
|W^{(7)}_c \circ \Phi_W^{x_0}| \le & 2 W^{(7)}_c(x_0, s_0) e^{-\frac 1 8 (s - s_0)} + \int_{s_0}^s e^{-\frac 1 8(s - s')} \| F_{W,7}^{c} \circ \Phi_W \| \ud s' \\ \n
\le & 2  e^{-\frac 1 8(s - s_0)} + \int_{s_0}^s e^{-\frac 1 8 (s - s')} M \ell^{\frac 1 5} e^{\frac 3 4 (s' - s_0)} \ud s' \\
\le & 2 e^{-\frac 1 8 (s - s_0)} + 2M \ell^{\frac 1 5} e^{\frac 3 4 (s - s_0)}\,,\n
\end{align}
where we have invoked the enhanced localized estimate, \eqref{better:1}. 
\end{proof}

We now verify \eqref{pc:W0} - \eqref{mako:6}.
\begin{lemma} For $n=1,\dots,7$  and $\abs{x}\le \ell$ we have the following estimates
\begin{align} \label{prince:1}
|W_c| &\le \frac{M^{4}}{2} \eps^{\frac 3 4} e^{\frac 3 4 s} \,, \\ \label{prince:2}
|W^{(n)}_c \eta_{\frac{1}{20}} | &\le \frac{M^{(n+2)^2}}{2} \eps^{\frac 34} e^{\frac 3 4 s}  \,. 
\end{align}
\end{lemma}
\begin{proof} Consider equation \eqref{eq.dcw.0} for $\p_c W$. First, define the rescaled quantity $Q := \p_c W e^{- \frac 1 4 (s -s_0)}$, which satisfies 
\begin{align*}
(\p_s + \beta_\tau W^{(1)}) Q  + \mathcal{V}_W \p_x Q = e^{- \frac 1 4 (s - s_0)} F^c_{W,0}
\end{align*} 
By Gr\"onwall, we have 
\begin{align} \n
|Q \circ \Phi_W^{x_0}| \le & |Q(x_0, s_\ast)| e^{- \int_{s_\ast}^s \beta_\tau W^{(1)} \circ \Phi_W^{x_0}} + \int_{s_\ast}^s e^{- \int_{s'}^s \beta_\tau W^{(1)} \circ \Phi_W^{x_0} } |e^{- \frac 1 4 (s' - s_0)} F_{W,0}^{c} \circ \Phi_W^{x_0}| \ud s' \\
\lesssim & (\| W_c(\cdot, s_0) \|_\infty + \ell^{\frac 1 2}M e^{\frac 1 2(s_\ast - s_0)}  ) M^{\frac{1}{50}} + M^{\frac{1}{50}} \int_{s_\ast}^s e^{- \frac 1 4 (s'  - s_0)} \eps^{\frac 18} \ud s' \,,\n
\end{align}
where we have invoked \eqref{thor:4} for the estimate on the damping term, and estimate \eqref{Fwc.est.ult:0} for the forcing term. Multiplying through by $e^{\frac 1 4 (s - s_0)}$ and using that $s_\ast \le s$ generates the desired bound. 

For \eqref{prince:2}, we again use Gr\"onwall to estimate 
\begin{align} \n
|\eta_{\frac{1}{20}} W_c^{(n)} \circ \Phi_W^{x_0}| \le& |W_c^{(n)}(x_0, s_\ast)| e^{- \int_{s_\ast}^s D^{c}_{n, \frac 1 5} \circ \Phi_W^{x_0}} + \int_{s_\ast}^s e^{- \int_{s'}^s D_{n, \frac 1 5}^{c} \circ \Phi_W^{x_0}} |\eta_{\frac{1}{20}} F_{W,n}^{c} \circ \Phi_W^{x_0}| \ud s' \\ \n
\lesssim & ( \| W_c^{(n)}( \cdot, s_0) \|_\infty + M e^{\frac 3 4 (s_\ast - s_0)} ) M^{\frac{1}{50}} + M^{\frac{1}{50}}  \int_{s_0}^s M^{-1} M^{(n+2)^2} e^{\frac 3 4 (s'-s_0)} \ud s' \\ 
\lesssim & M e^{\frac 3 4 (s_\ast - s_0)} M^{\frac{1}{50}} + M^{\frac{1}{50}} M^{(n+2)^2} M^{-1} e^{\frac 3 4 (s - s_0)}\,,\n
\end{align}
where we have invoked the estimate \eqref{thor:5} on the damping term, and estimate \eqref{Fwc.est.ult} to estimate the forcing term. This concludes the proof of the lemma. 

\end{proof}

\subsection{Transport estimates for $\nabla_c^2 W$}

The following verifies the bootstraps \eqref{buddy:1}. 
\begin{lemma} Let $0 \le n \le 6$. 
\begin{align*}
\| \p_{c_1 c_2} W^{(n)} \|_\infty \le \frac{ M^{(n+5)^2}}{2} \eps^{\frac 3 4} e^{\frac 3 4 s}\,. 
\end{align*}
\end{lemma}
\begin{proof} Using equation \eqref{Bernie:1}, we write via Gr\"onwall upon noting that $W_{c_1 c_2}(s_0, x) = 0$, 
\begin{align} \n
|W^{(n)}_{c_1 c_2} \circ \Phi_{W}^{x_0}| \le &\int_{s_0}^s e^{- \int_{s'}^s D^{c}_n \circ \Phi_W^{x_0}} | F_{W,n}^{c_1, c_2} \circ \Phi_W^{x_0}| \ud s' \\
 \lesssim & \int_{s_0}^s e^{\frac{11}{8} (s - s')} M^{(n+5)^2 - 1} e^{\frac 3 2 (s' - s_0)} \ud s' \lesssim M^{(n+5)^2 - 1} e^{\frac 3 2 (s - s_0)}\,,\n
\end{align}
where above we have used the definition \eqref{country:coffee:3} to produce the trivial bound 
\begin{align*}
D_n^{c} \ge - \frac{11}{8}\,,
\end{align*}
and estimate \eqref{disclosure:1} - \eqref{disclosure:2} for the forcing. 
\end{proof}

\section{Proof of main theorem}

We are now ready to establish all of the assertions in Theorem \ref{thm:general}. While the bootstrap estimates put forth in Section \ref{section:Bootstraps} have all been verified, the first task is to now establish the inductive proposition, Proposition \ref{induct:prop}. 

\subsection{Newton iteration} \label{subsection:proof}

We now prove the main theorem by designing a Newton scheme on appropriately defined maps $\mathcal{T}_N$. 
\begin{proof}[Proof of Proposition \ref{induct:prop}] First, we will define the map $\mathcal{T}_N: \mathcal{B}_N(\alpha_N, \beta_N) \subset \mathbb{R}^2 \rightarrow \mathbb{R}^2$ by
\begin{align*}
\mathcal{T}_N(\alpha, \beta) := (W^{(2)}_{\alpha, \beta}(0, s_{N+1}), W_{\alpha, \beta}^{(3)}(0, s_{N+1}))\,. 
\end{align*}
Define now the \emph{error} quantities via 
\begin{align*}
&E_{N}^{(2)} := W_{\alpha_N, \beta_N}^{(2)}(0, s_{N+1}) = \mathcal{T}_N^{(1)}(\alpha_N, \beta_N)\,, \\
&E_N^{(3)} := W_{\alpha_N, \beta_N}^{(3)}(0, s_{N+1}) = \mathcal{T}_N^{(2)}(\alpha_N, \beta_N)\,.
\end{align*}
An immediate consequence of \eqref{eq:2:3:W0} is the estimate 
\begin{align*}
|E_N^{(2)}| + |E_N^{(3)}| \le M^{25} e^{- s_{N}}.
\end{align*}

We now compute the matrix 
\begin{align*}
\nabla_{\alpha, \beta} \mathcal{T}_N = \begin{pmatrix} \p_\alpha W^{(2)}_{\alpha, \beta}(0, s_{N+1}) & \p_\beta W^{(2)}_{\alpha, \beta}(0, s_{N+1}) \\  \p_\alpha W^{(3)}_{\alpha, \beta}(0, s_{N+1}) &  \p_\beta W^{(3)}_{\alpha, \beta}(0, s_{N+1}) \end{pmatrix},
\end{align*}
which, when we evaluate at the point $(\alpha_N, \beta_N)$ produces 
\begin{align*}
\nabla_{\alpha, \beta}|_{\alpha_N, \beta_N} \mathcal{T}_N = \begin{pmatrix} \p_\alpha W^{(2)}_{\alpha_N, \beta_N}(0, s_{N+1}) & \p_\beta W^{(2)}_{\alpha_N, \beta_N}(0, s_{N+1}) \\  \p_\alpha W^{(3)}_{\alpha_N, \beta_N}(0, s_{N+1}) &  \p_\beta W^{(3)}_{\alpha_N, \beta_N}(0, s_{N+1}) \end{pmatrix}\,.
\end{align*}

The bootstrap assumptions \eqref{mako:4} - \eqref{mako:5}, coupled with the estimates on the second derivatives, \eqref{buddy:1} enable us to apply the Implicit Function Theorem on $\mathcal{T}_N$ in a neighborhood $\mathcal{B}_N(\alpha_N, \beta_N)$ of $(\alpha_N, \beta_N)$, defined in \eqref{def:BN},  to conclude that
\begin{align} \label{crickets:2}
|\alpha_{N+1} - \alpha_{N}| &\le M^{25} e^{- \frac 3 4 (s_N - s_0)} e^{-s_N} + \eps^{\frac 1 4} e^{-s_N} e^{- \frac 1 2 (s_N - s_0)}\,, \\ \label{crickets:3}
|\beta_{N+1} - \beta_{N}| &\le 2M^{25} e^{- \frac 1 2(s_N - s_0)} e^{-s_N}\,, 
\end{align}
which in particular verifies the bootstraps \eqref{def:BN}. More specifically, we have used that in the neighborhood $\mathcal{B}_N(\alpha_N, \beta_N)$, we have uniform bounds on the $(\alpha,\beta)$ Hessian  of $ \mathcal{T}_N$. Estimating $\p_{\alpha \alpha} \mathcal{T}_N$ yields
\begin{align} \n
\sup_{\alpha, \beta \in \mathcal{B}_N} |\p_{\alpha \alpha} \mathcal{T}_N | |\alpha - \alpha_N| \lesssim_M & e^{\frac 3 2 (s_{N+1} - s_0)} \Big( e^{- \frac 34 (s_N - s_0)} e^{-s_N} + \eps^{\frac 1 5} e^{-s_N} e^{- \frac 1 2(s_N - s_0)} \Big) \\
\lesssim_M & e^{- s_N}\Big( e^{\frac 3 4 (s_N - s_0)} + \eps^{\frac 1 5} e^{s_N - s_0)} \Big) \ll  \p_\alpha|_{\alpha_N, \beta_N} \mathcal{T}_N\,.\n
\end{align}
Similarly, for $\p_{\alpha \beta} \mathcal{T}_N$ we have
\begin{align} \n
\sup_{\alpha, \beta \in \mathcal{B}_N} |\p_{\alpha \beta} \mathcal{T}_N| |\alpha - \alpha_N| \lesssim_M & e^{\frac 3 2 (s_{N+1} - s_0)} \Big(  e^{- \frac 34 (s_N - s_0)} e^{-s_N} + \eps^{\frac 1 5} e^{-s_N} e^{- \frac 1 2(s_N - s_0)} \Big) \\
\lesssim_M & e^{- s_N}\Big( e^{\frac 3 4 (s_N - s_0)} + \eps^{\frac 1 5} e^{(s_N - s_0)} \Big) \ll  \p_\beta |_{\alpha_N, \beta_N}  \mathcal{T}_N\,,\n
\end{align}
and 
\begin{align} \n
\sup_{\alpha, \beta \in \mathcal{B}_N} |\p_{\alpha \beta} \mathcal{T}_N| |\beta - \beta_N| \lesssim_M e^{\frac 3 2 (s_{N+1} - s_0)} \Big( e^{-s_N} e^{- \frac 1 2(s_N - s_0)}  \Big) \ll \p_\alpha|_{\alpha_N, \beta_N}  \mathcal{T}_N\,.
\end{align}
Finally, estimating $\p_{\beta \beta} \mathcal{T}_N$ yields
\begin{align} \n
\sup_{\alpha, \beta \in \mathcal{B}_N} |\p_{\beta \beta} \mathcal{T}_N| |\beta - \beta_N| \lesssim_M e^{\frac 3 2 (s_{N+1} - s_0)} \Big( e^{-s_N} e^{- \frac 1 2(s_N - s_0)}  \Big) \ll \p_\beta|_{\alpha_N, \beta_N} \mathcal{T}_N\,.
\end{align}
\end{proof}

We can now send $N \rightarrow \infty$ to obtain our limiting profiles. To make matters precise, we define the following norm, specific to a given $s_\ast \in [s_0, \infty)$. 
\begin{align} \n
 \Big\| (W, Z, A) \Big\|_X :=  &\Big\| \| W \eta_{- \frac{1}{20}} \|_{L^\infty} \Big\|_{L^\infty(s_0, s_\ast)} + \sum_{j = 1}^6 \Big\| \| W^{(j)} \eta_{\frac 1 5} \|_{L^\infty} \Big\|_{L^\infty(s_0, s_\ast)} \\ \n
 & + \Big\| e^{\frac 3 4 s} W^{(2)}(0, s) \Big\|_{L^\infty(s_0, s_\ast)} +  \Big\| e^{ \frac 3 4 s} W^{(3)}(0, s) \Big\|_{L^\infty(s_0, s_\ast)} \\ \n
 & + \eps^{- \frac 5 4} \Big\| \| Z \|_\infty \Big\|_{L^\infty(s_0, s_\ast)} +  \eps^{- \frac 3 4}  \Big\| \| A\|_\infty \Big\|_{L^\infty(s_0, s_\ast)} \\ \label{norm:X}
 & + \sum_{j = 1}^6 \Big\| e^{\frac 5 4 s} \| Z^{(j)}  \|_{L^\infty} \Big\|_{L^\infty(s_0, s_\ast)} + \sum_{j = 1}^6 \Big\| e^{\frac 5 4 s} \| A^{(j)}  \|_{L^\infty} \Big\|_{L^\infty(s_0, s_\ast)}\,,
\end{align}
and the corresponding Banach space 
\begin{align*}
X := \text{Closure of } C^\infty([s_0, s_\ast], \mathbb{R} )^3 \text{ with respect to } \| \cdot \|_X\,. 
\end{align*}
We also define the following norms in which we measure the modulation variables
\begin{align} \n
\| (\mu, \tau, \kappa, \xi) \|_{Y} := & \eps^{- \frac 1 7} \| e^{\frac 3 4 s} \mu \|_{L^\infty(s_0, s_\ast)} + \eps^{- \frac 1 7} \| e^{\frac 3 4 s} \dot{\tau} \|_{L^\infty(s_0, s_\ast)} + \eps^{- \frac 1 8} \| \dot{\kappa} \|_{L^\infty(s_0, s_\ast)} \\
&+ \frac{1}{\kappa_0} \| \dot{\xi} \|_{L^\infty(s_0, s_\ast)}\,,\n
\end{align}
and the corresponding Banach space
\begin{align*}
Y := \text{Closure of } C^\infty([s_0, s_\ast])^4 \text{ with respect to } \| \cdot \|_Y\,. 
\end{align*}

\begin{corollary}  \label{corr:infty} There exist values $(\alpha_\infty, \beta_\infty)$ so that the data $W_0$ given according to \eqref{guitar:or:band} yields a global solution, $(W, Z, A) \in X$ and $(\mu, \tau, \kappa, \xi) \in Y$ on $s_0 \le s <\infty$ which satisfies
\begin{align} \label{X:estimate}
\| (W, Z, A) \|_X + \| (\mu, \tau, \kappa, \xi) \|_Y \lesssim_M 1\, \qquad \text{ for all } s_{\ast} \in [s_0, \infty), 
\end{align}
the constraints 
\begin{align*}
W(0, s) = 0\,, && W^{(2)}(0, s) = -1\,, && W^{(4)}(0, s) = 0\,, 
\end{align*}
the following asymptotic behavior for the second and third derivatives: 
\begin{align*}
|W^{(2)}(0, s)| \lesssim e^{- \frac 3 4 s}\,, && |W^{(3)}(0, s)| \lesssim e^{-\frac 3 4 s}\,.
\end{align*}
Finally, for the fifth derivative $W^{(5)}(0, s)$, there exists a number $\nu$ such that 
\begin{align} \label{nu:def}
W^{(5)}(0, s) \rightarrow \nu, \qquad |\nu - 120| \lesssim \eps^{\frac 78}\,. 
\end{align}
\end{corollary}
\begin{proof} Fix any $s_\ast$ satisfying $s_0 \le s_\ast < \infty$, and consider the sequences  
\begin{align*}
\{ W_{\alpha_N, \beta_N}, Z_{\alpha_N, \beta_N}, A_{\alpha_N, \beta_N} \}_{N \ge \lfloor s_\ast \rfloor + 1}  &=: \{W_N, Z_N, A_N\}_{N \ge \lfloor s_\ast \rfloor + 1}\,, \\
\{ \mu_{\alpha_n, \beta_n}, \dot{\tau}_{\alpha_N, \beta_N}, \dot{\kappa}_{\alpha_N, \beta_N}, \dot{\xi}_{\alpha_N, \beta_N} \}_{N \ge \lfloor s_\ast \rfloor + 1} &=: \{ \mu_N, \dot{\tau}_N, \dot{\kappa}_N, \dot{\xi}_N  \}_{N \ge \lfloor s_\ast \rfloor + 1}\,. 
\end{align*}
Our assertion will be that these sequences are Cauchy in the spaces $X$ and $Y$, respeectively. Let now $s_0 \le s \le s_\ast$. Recall from the definition of $\mathcal{B}_N$ in \eqref{def:BN}, that 
\begin{align*}
|\alpha_{N+1} - \alpha_N| \lesssim_M e^{-s_N} e^{- \frac 1 2 (s_N - s_0)}, \qquad |\beta_{N+1} - \beta_N| \lesssim_M e^{-s_N} e^{- \frac 1 2(s_N - s_0)}\,. 
\end{align*}
Considering the first term in definition of \eqref{norm:X}, we now estimate 
\begin{align} \n
\| (W_{N+1} - W_N ) \eta_{- \frac{1}{20}} \|_{L^\infty} \lesssim_M & e^{-s_N} e^{- \frac 1 2 (s_N - s_0)}  \sup_{(\alpha, \beta) \in \mathcal{B}_N} \| \p_c W \eta_{- \frac{1}{20}} \|_{L^\infty} \\ \label{above:above}
\lesssim_M & e^{- s_N} e^{- \frac 1 2 (s_N - s_0)} e^{\frac 3 4 (s -s_0)}\,, 
\end{align}
where we have invoked the estimate \eqref{pc:W0}. Second, for $k \ge 1$, we have a nearly identical estimate using \eqref{mako:6}. Third, we estimate using \eqref{mako:4} - \eqref{mako:5}
\begin{align*} \n
e^{\frac 3 4 s} |W_{N+1}^{(2)}(0, s) - W_{N}^{(2)}(0, s)| \lesssim_M &e^{\frac 3 4 s}  e^{-s_N} e^{- \frac 1 2 (s_N - s_0)}  \sup_{(\alpha, \beta) \in \mathcal{B}_N} |\p_c W^{(2)}(0, s)|  \\
\lesssim_M & e^{\frac 3 4 s} e^{-s_N} e^{- \frac 1 2 (s_N - s_0)} e^{\frac 3 4(s - s_0)}\,.
\end{align*}
An analogous estimate applies to the fourth quantity in \eqref{norm:X}.

For the quantities in the third and fourth lines of \eqref{norm:X}, we use \eqref{Z:boot:0} - \eqref{A:boot:9}, coupled with \eqref{mako:2} - \eqref{mako:3}, in essentially the identical manner to the quantities above.  Similarly, for the quantities in $Y$, we couple the estimates \eqref{e:GW0} - \eqref{mod:sub}, with the estimates \eqref{mako:1} - \eqref{Baba:O:1}.

As $s \le s_\ast \le s_N \rightarrow \infty$, the estimates above clearly imply that $\{W_N, Z_N, A_N\}$ is a Cauchy sequence in the norm X and $\{ \mu_N, \dot{\tau}_N, \dot{\kappa}_N, \dot{\xi}_N  \}_{N \ge \lfloor s_\ast \rfloor + 1}$ form a Cauchy sequence in $Y$, upon taking supremum in $s \in [s_0, s_\ast]$. We conclude by sending $s_\ast \rightarrow \infty$. 

For the final step, we note that the norms $X$ and $Y$ are clearly strong enough to pass to the limit in the equation \eqref{basic:w} - \eqref{basic:a}. Furthermore, applying \eqref{e:bluebottle} and \eqref{e:bluebottle2} yields that
\[\nu=\lim_{s\rightarrow \infty}W^{(5)}(0, s)\,,\]
exists, and by \eqref{fifth:deriv:W:0} we have
\[\abs{\nu-120}\lesssim \eps^{\frac78}\,.\]
\end{proof}

\subsection{Consequential quantitative properties for $(w, a, z)$}\label{s:bandicoot}

We finish by providing a proof of the following consequence of our construction. 
\begin{lemma}[Holder $1/5$ Regularity] The solution $w(\theta, s)$ satisfies the following Holder $1/5$ regularity estimate uniformly in $t$ up to the shock time $T_\ast$
\begin{align} \label{reaCH:1}
\sup_{t \in [-\eps, T_\ast]} [w(\cdot, t)]_{\frac{1}{5}} \lesssim 1\,.
\end{align}
\end{lemma}
\begin{proof} Due to bootstrap bounds \eqref{thurs:1}, \eqref{e:Wtilde:1:bootstrap} on $\tilde{W}$, and properties \eqref{bound:bar:W:1} on $\bar{W}$ we obtain the following on $W = \bar{W} + \tilde{W}$,  
\begin{align*}
|\p_x W(x, s)| \lesssim \langle x \rangle^{- \frac 4 5},  
\end{align*}
where the implicit constant is uniform, and in particular, independent of $s$. Using this, we write 
\begin{align} \n
[W(\cdot, x)]_{\frac 1 5} = & \sup_{(x, x')} \frac{|W(x, s) - W(x', s)|}{|x - x'|^{\frac 1 5}} = \sup_{(x, x')} \frac{1}{|x - x'|^{\frac 1 5}} |\int_{x}^{x'} \p_x W(y, s) \ud y| \\ \label{tho:1}
\lesssim & \sup_{(x, x')} \frac{1}{|x - x'|^{\frac 1 5}} \int_{x}^{x'} \langle y \rangle^{- \frac 4 5} \ud y = \sup_x \frac{1}{|x|^{\frac 1 5}} \int_0^x \langle y \rangle^{- \frac 4 5} \ud y \lesssim 1\,. 
\end{align}
Finally, we use \eqref{sec3:1} to argue as follows. Select any $(\theta, \theta') \in \mathbb{T}$. Then there exists a corresponding $(x, x')$ determined through \eqref{sec3:1} so that 
\begin{align*}
\frac{|w(\theta, t) - w(\theta', t)|}{|\theta - \theta'|^{\frac 1 5}} = \frac{|W(x, s) - W(x', s)|}{|x - x'|^{\frac 1 5}}\,.
\end{align*} 
From here, we take supremum over $\theta$ and apply estimate \eqref{tho:1} to reach \eqref{reaCH:1}.
\end{proof}

\begin{lemma} The following estimates hold for a constant $C_M$ that depends on $M$,   
\begin{align*}
\sup_{t \in [-\eps, T_\ast)} \sup_{\theta \in \mathbb{T}} |\p_\theta a(\cdot, t)| &\le C_M\,, \\
\sup_{t \in [-\eps, T_\ast)} \sup_{\theta \in \mathbb{T}} |\p_\theta z(\cdot, t)| &\le C_M\,, \\
\sup_{t \in [-\eps, T_\ast)} \sup_{\theta \in \mathbb{T}} |w(\theta, t)| &\le 2 \kappa_0\,. 
\end{align*}
\end{lemma}
\begin{proof} This follows upon pulling back to the original coordinate system via \eqref{whilk:miskey:1} and \eqref{whilk:miskey:4} which gives 
\begin{align*}
\sup_{t} \sup_{\theta} |\p_\theta a| &= \sup_{s} \sup_{x} e^{\frac 5 4 s} |A^{(1)}| \le M^{2}\,, \\
\sup_{t} \sup_{\theta} |\p_\theta z| &= \sup_{s} \sup_{x} e^{\frac 5 4 s} |Z^{(1)}| \le M^{2}\,,
\end{align*}
upon invoking bootstraps \eqref{Z:boot:0} and \eqref{A:boot:9}, and upon invoking Corollary \ref{corr:infty} to ensure that these bootstraps are satisfied globally. 

We now arrive at the pointwise estimate for $w(\theta, t)$. For this, we use the bootstraps \eqref{W:boot:0}, \eqref{e:support}, and \eqref{mod:sub} to obtain 
\begin{align} \n
|w| \le & e^{- \frac s 4} |W| + |\kappa| \lesssim  e^{- \frac s 4} \sup_{\substack{ - \log(\eps) \le s < \infty \\ x \in B_f }} \langle x \rangle^{\frac 1 5} + |\kappa_0| + \eps \\
\lesssim & e^{- \frac s 4} (M \eps e^{\frac 5 4 s} )^{\frac 1 5} + |\kappa_0| + \eps \le 2 |\kappa_0|\,.\n
\end{align} 
 \end{proof}

We now provide a final lemma to obtain the shock dynamics of $\p_\theta w(x, t)$. 
\begin{lemma} The following asymptotic behavior is valid for $w(x, t)$,
\begin{align} \label{BUD:1}
&\lim_{t \rightarrow T_\ast} \p_\theta w(\xi(t), t) = - \frac{1}{T_\ast - t}\,.
\end{align}
\end{lemma}
\begin{proof} First, \eqref{BUD:1} follows upon using \eqref{BUDS:2}, evaluating at $x = 0$, and using the constraint $W^{(1)}(s, 0) = -1$ which yields 
\begin{align} \label{pre:lim:1}
\p_\theta w(\xi(t), t) = - \frac{1}{\tau(t) - t}\,. 
\end{align}
We now note that, while $\dot{\tau}(t)$ satisfies the bootstrap \eqref{e:GW0}, $\tau(t)$ is itself uniquely defined upon enforcing 
\begin{align*}
\tau(T_\ast) = T_\ast\,. 
\end{align*} 
Thus, we may take the limit of \eqref{pre:lim:1} to get \eqref{BUD:1}.
\end{proof}

We now establish the following pointwise asymptotic stability result. 
\begin{lemma}\label{l:wombat} Let $W$ be the global solution from Corollary \ref{corr:infty} and let $\nu$ be as in \eqref{nu:def}. Then, for any fixed $x \in \mathbb{R}$, the following asymptotic behavior holds 
\begin{align} \label{pointwise:limit}
\lim_{s \rightarrow \infty} W^{(n)}(x, s) = \bar{W}^{(n)}_\nu(x),  \qquad n=0,\dots,5\,,  
\end{align}
where $\bar{W}_\nu$ is the exact, self-similar Burgers profile 
\begin{align} \label{bar:nu}
\bar{W}_\nu(x) := \left(\frac{\nu}{120}\right)^{-\frac14}\bar W\left(\left(\frac{\nu}{120}\right)^{\frac14} x\right)\,.
\end{align}
\end{lemma}
\begin{remark}
We note that the parameter $\nu$ in \eqref{bar:nu} is directly related to the spatial rescaling invariance of Burgers' equation, listed in Section \ref{s:burgers}.
\end{remark}
\begin{proof}
Let $(W,Z,A)$ be the global solution defined in Corollary \ref{corr:infty}. First, it is easily verified that $\bar W_{\nu}$ is an exact solution to the self-similar Burgers' equation \eqref{Burger:1}, and that the first 5 Taylor coefficients of $\bar W_{\nu}$ are given by
\begin{align*}
W_{\nu}(0)=W_{\nu}^{(2)}(0)=W_{\nu}^{(3)}(0)=W_{\nu}^{(4)}(0)=0,\quad W_{\nu}^{(1)}(0)=-1\quad\mbox{and } W_{\nu}^{(5)}(0)=\nu\,.
\end{align*}
In particular, at the limit $s\rightarrow \infty$, the first 5 Taylor coefficients of $W$ and $\bar W_{\nu}$ match. Let us define the difference
\[\tilde W_{\nu}=W-\bar W_{\nu}\,.\]
Hence, by definition
\begin{align}\label{eq:Taylor:cancellation}
\lim_{s\rightarrow \infty }W_{\nu}^{(n)}(0)=0\,,
\end{align}
for all $n=0,\dots,5$.
By a similar calculation to \eqref{diff:eq0} -- although we will rearrange the terms on the left-hand-side and right-hand-side -- we obtain
\begin{align*}
&(\p_ s- \frac 1 4+ \bar W_{\nu}^{(1)}) \tilde W_{\nu} +(W+\frac54 x) \p_x \tilde W_{\nu} = -\beta_\tau e^{- \frac 3 4 s} \dot{\kappa} +  F_{ W}+((1-\beta_{\tau}W)-G_W)\p_x W:=\tilde F_\nu\,.
\end{align*}

Using \eqref{e:GW0}, \eqref{GW:transport:est} and \eqref{FW:est:1}, we have that for any fixed $x_*$ that
\begin{equation}\label{eq:tilde:Fnu:decay}
\int_{s_0}^{\infty}\abs{\tilde F_\nu (x_*,s)}\,ds<\infty\,.
\end{equation}
Now fix $\delta>0$, $x_*\in\mathbb R$ and $s_*\geq -\log \eps$. Then as a consequence of \eqref{eq:Taylor:cancellation} and \eqref{weds:1} we have
\begin{equation}\label{eq:Taylor:est}
\abs{\tilde W(x_*,s_*)}\les_M \abs{x_*}^6 + \delta\,,
\end{equation}
assuming that $s_*$ is taken sufficiently large dependent on the choice of $\delta$. Now define $\Phi$ to be the trajectory
\begin{align*}
\p_s \Phi(s) = \left(W+\frac54 x\right) \circ \Phi, \qquad \Phi(s_*) = x_*\,. 
\end{align*}
If we in addition define $q= e^{-\frac54(s-s_*)}\tilde W_\nu$, then $q\circ \Phi$ satisfies the equation
\begin{align*}
(\partial_s+1+ \bar W_{\nu}^{(1)} )(q\circ \Phi)=e^{-\frac54(s-s_*)}\tilde{F}_{\nu}\circ \Phi\,.
\end{align*}
Since $\bar W_{\nu}^{(1)} \geq -1$, then by Gr\"onwall and \eqref{eq:tilde:Fnu:decay}, it follows that
\begin{equation}\label{eq:q:phi:est}
\abs{q\circ \Phi (s)}\leq  \abs{q\circ \Phi (s_*)}+\delta
\end{equation}
for $s\geq s_*$, assuming that $s_*$ is taken to be sufficiently large, dependent on $\delta$. Combining \eqref{eq:Taylor:est} and \eqref{eq:q:phi:est} we obtain that for $s_* \leq s\leq s_*-\frac{23}{5}\log \abs{x_*}$ and assuming $\delta\leq \abs{x_*}^6$ 
\begin{equation}\label{eq:wallaby}
\abs{\tilde W_{\nu}\circ \Phi (s)}\les_M  e^{\frac54(s-s_*)}(\abs{x_*}^6 + \delta)
\les_M \abs{x_*}^{\frac{1}{4}}
\,.
\end{equation}
Let us restrict to the case $x_*>0$ and assume the lower bound
\begin{equation}\label{eq:escape:lower:bound}
\Phi\left(s_*-\frac{23}{5}\log \abs{x_*}\right)\geq \Gamma\,.
\end{equation}
In particular, by continuity, \eqref{eq:escape:lower:bound} implies that for any $x_*\leq x\leq \Gamma$, there exists an $s_*\leq s\leq (s_*-\frac{23}{5}\log \abs{x_*}$ such that $\Phi(s)=x$ and hence by \eqref{eq:wallaby}
\begin{equation*}
\abs{\tilde W_{\nu}(x,s)} \les_M \abs{x_*}^{\frac{1}{4}}\,.
\end{equation*}
By taking the limit $s_*\rightarrow \infty$, this implies
\begin{equation}\label{eq:kangaroo}
\lim_{s\rightarrow \infty }\abs{\tilde W_{\nu}(x,s)} \les_M \abs{x_*}^{\frac{1}{4}}\,,
\end{equation}
for any  $x_*\leq x\leq \Gamma$.

 It remains to prove a $x_*$ dependent lower bound  on $\Gamma$ that increases as $x_*\rightarrow 0$. First note that by \eqref{eq:W1:bnd:1} and the Fundamental Theorem of Calculus
 \begin{align*}
 W+\frac54 x\geq x\left( \frac54 -\norm{W^{(1)}}_{\infty}\right)\geq \frac29 x\,.
 \end{align*}
 Thus by Gr\"onwall $\Phi(s)\geq e^{\frac15(s-s_*)}s_*$, which implies
\begin{equation*}
\Phi\left(s_*-\frac{23}{5}\log \abs{x_*}\right)\geq \abs{x_*}^{-\frac{1}{45}}\,,
\end{equation*}
and hence we can take $\Gamma= \abs{x_*}^{-\frac{1}{45}}$. Thus by taking $x_*\rightarrow 0$, from \eqref{eq:kangaroo} we obtain
\begin{equation}\label{eq:bilby}
\lim_{s\rightarrow \infty}\abs{\tilde W_{\nu}(x,s)} =0 \,,
\end{equation}
for all $x>0$. An analogous argument yields \eqref{eq:bilby} for the case $x<0$. The case $x=0$ is trivial since $\tilde W_{\nu}(0,s)=0$ for all $s$. Thus, $W$ converges pointwise to $\bar W_{\mu}$. The proof for $n = 1,\dots,5$ works in an analogous manner. 
\end{proof}

\begin{remark} We remark that the asymptotic profile that is picked out in \eqref{pointwise:limit} is consistent with our estimates \eqref{thurs:1}. Indeed, by using estimate \eqref{nu:def}, we can estimate 
\begin{align*}
\| (\bar{W}_\nu - \bar{W}) \eta_{- \frac{1}{20}} \|_\infty \lesssim \eps^{\frac 78}\,, 
\end{align*}
which shows that $W$ can simultaneously lie in a ball of size $\eps^{\frac{3}{20}}$ within $\bar{W}$ (in the weighted norm above) and converge pointwise to $\bar{W}_\nu$. 
\end{remark}

It is now possible to prove asymptotic stability in a much stronger sense. To do so, we define the slightly weaker weighted space by first fixing a $0 < \delta \ll 1$, 
\begin{align} \label{norm:X:delta}
\| W \|_{\mathcal{X}_{-\delta}} := \| W \eta_{- \frac{1}{20} - \delta} \|_\infty + \sum_{j = 1}^5 \| W^{(j)} \eta_{\frac1 5 - \delta} \|_\infty\,. 
\end{align}

\begin{lemma} For any $\delta > 0$, 
\begin{align} \label{asy:conv:1}
\Big\| W - \bar{W}_\nu \Big\|_{\mathcal{X}_{-\delta}} \rightarrow 0\quad \text{ as } s \rightarrow \infty\,. 
\end{align}
\end{lemma}
\begin{proof} This is  a standard consequence of pointwise convergence (\eqref{pointwise:limit}), uniform estimates on six derivatives, guaranteed by the specification of the norm $X$, \eqref{norm:X}, and finally, the compactness afforded by the weaker weight of $\langle x \rangle^{-\delta}$ in our norm \eqref{norm:X:delta}. For the purpose of completeness, we include the argument for the lowest order part of the $X_{-\delta}$ norm, while the higher order components work in an exactly analogous fashion. 

To prove \eqref{norm:X:delta}, specifically $\| (W - \bar{W}_\nu) \eta_{- \frac{1}{20} - \delta} \|_\infty \rightarrow 0$,  we will first fix an arbitrary $\tilde{\eps} > 0$, and demonstrate the existence of $S = S(\tilde{\eps})$ large, such that $s > S$ implies $\| (W - \bar{W}_\nu) \eta_{- \frac{1}{20} - \delta} \|_\infty \le \tilde{\eps}$. 

First, there exists $X = X(\tilde{\eps}, \delta)$ so that 
\begin{align*}
\| (W - \bar{W}_\nu) \eta_{- \frac{1}{20} - \delta} \|_{L^\infty(|x| \ge X)} \le \frac{\tilde{\eps}}{10}\,,  
\end{align*} 
according to the estimate \eqref{X:estimate} on $W$ and \eqref{decay:bar:2} on $\bar{W}$ (and hence, $\bar{W}_\nu$). 

We thus restrict to the compact interval $|x| \le X$, which we now subdivide into $N = N(\tilde{\eps}, M)$ sub-intervals with centers $x_k$, $k = 0, ..., N$. $N$ will be selected according to the rule:  
\begin{align*}
(\| W^{(1)} \|_\infty + \| \bar{W}^{(1)}_\nu \|_\infty) \frac{1}{N} < \frac{\tilde{\eps}}{10}\,.
\end{align*}

By the pointwise convergence guaranteed by \eqref{pointwise:limit}, there exists an $s_k$ so that 
\begin{align*}
|W(s_k, x_k) - \bar{W}_\nu(x_k)| \le \frac{\tilde{\eps}}{10}\,.
\end{align*}
Define now $S := \max_{k} s_k$. Estimating, we have 
\begin{align} \n
|W(s, x) - \bar{W}_\nu(x)| \le & |W(s, x) - W(s, x_k)| + |W(s, x_k) - \bar{W}_\nu(x_k)| \\ \n
& + |\bar{W}_\nu(x_k) - \bar{W}_\nu(x)|  \\ \n
\le & (\| W^{(1)} \|_\infty + \| \bar{W}^{(1)}_\nu \|_\infty) |x - x_k| + \frac{\tilde{\eps}}{10} \\
\le & \frac{\tilde{\eps}}{10} +  \frac{\tilde{\eps}}{10}\,,\n
\end{align}
for $s > S$. Taking supremum over $|x| \le X$ gives the desired conclusion. 

\end{proof}

\begin{proof}[Proof of Corollary \ref{c:open}]
We note that the proof follows in a very similar manner to the proof of Corollary 4.7 of \cite{BuShVi2019}. 

By finite speed of propagation, the strict support properties imposed in Section \ref{ss:initial}, can be replaced by the condition that $(w_0,z_0,a_0)$ satisfy the conditions modulo a small perturbation in the $C^8$ topology.

The conditions \eqref{eq:W0:diff} for the cases $n=0,1$ impose no obstruction to $\check w_0$ been chosen within an open set since the conditions may be enforced by choosing $\eps$ and $\kappa_0$ appropriately (it should be noted that these two parameters are free to be chosen from an open set). In order to weaken the condition \eqref{eq:W0:diff} for the case $n=4$, we note that by a Taylor expansion
\begin{align*}
\partial_{\theta}^4 w_0(\theta)&=\partial_{\theta}^4 w_0(0)+\theta \partial_{\theta}^5 w_0(0)+\OO(\eps^{-\frac{29}{4}}\theta^2)\\
&=\partial_{\theta}^4 w_0(0)+120\eps^{-6}\theta +\theta(\partial_{\theta}^3 w_0(0)-120\eps^{-6})+\OO(\eps^{-\frac{29}{4}}\theta^2)
\, ,
\end{align*}
here implicitly we used \eqref{eq:C8:bnd} and that 
\[\norm{\partial_{\theta}^6\eps^{\frac14} \bar W\left(\eps^{-\frac54}\theta\right)}_{\infty}\les \eps^{-\frac{29} 4} \,.\]
By continuity, given $\bar \eps$, then assuming $\partial_{\theta}^4 w_0(0)$ and $\partial_{\theta}^5 w_0(0)-120\eps^{-6}$ to be sufficiently small, there exists a $\theta\in(-\bar\eps,\bar\eps)$ such that  $\partial_{\theta_0}^4 w_0(\theta)=0$. Thus, up to a coordinate translation $\theta\mapsto \theta+\theta_0$, and under the assumptions $\partial_{\theta}^4 w_0(0)$ and $\partial_{\theta}^5 w_0(0)-120\eps^{-6}$ are both sufficiently small, we can remove the assumption \eqref{eq:W0:diff} for the case $n=4$. The strict assumption \eqref{eq:W0:diff} for the case $n=5$ may be removed by applying the rescaling
\[\tilde a(\theta,t)=\mu^{-1}a(\mu \theta,t),\quad\tilde w(\theta,t)=\mu^{-1}w(\mu \theta,t),\quad \tilde z(\theta,t)=\mu^{-1}z(\mu \theta,t)\,,\] 
for $\mu$ sufficiently close to $1$. As was noted in \cite{BuShVi2019}, such a rescaling would modify the domain; however, since by finite-speed of propagation we restrict our analysis to a strict subset of the domain, such a rescaling does not impose any problem.
\end{proof}

\bibliography{euler}
\bibliographystyle{plain}

\end{document}